\documentclass[a4paper,11pt,final]{article}

\usepackage[colorlinks=true]{hyperref} 
\usepackage{url}
\usepackage[english]{babel}
\usepackage[utf8]{inputenc}
\usepackage[T1]{fontenc}
\usepackage{lmodern}
\usepackage[pdftex]{graphicx}
\usepackage[top=2cm, bottom=2cm, left=2cm, right=2cm]{geometry}
\usepackage{setspace}
\usepackage[french]{varioref}
\usepackage{lastpage}
\usepackage{fancyhdr}
\usepackage[table]{xcolor}
\usepackage{tikz}
\usepackage{titling}
\usepackage{graphicx}
\usepackage{amsmath}
\usepackage{amsthm}
\usepackage{amssymb}  
	\newtheorem{theorem}{Theorem}
    \newtheorem{problem}[theorem]{Problem}
	\newtheorem{lemma}[theorem]{Lemma}
	
	\newtheorem{corollary}[theorem]{Corollary}

	\newtheorem{definition}[theorem]{Definition}
    \newtheorem{property}[theorem]{Property}
\usepackage{multirow}
\usepackage{rotating}
\usepackage{subcaption} 
\usepackage[gen]{eurosym}
\usepackage{dirtree}

\usepackage{rotating}
\usepackage{colortbl}

\usepackage{footmisc}

\usepackage{diagbox} 

\usepackage[T1]{fontenc}
\DeclareFontFamily{T1}{calligra}{}
\DeclareFontShape{T1}{calligra}{m}{n}{<->s*[1.44]callig15}{}
\DeclareMathAlphabet\mathcalligra   {T1}{calligra} {m} {n}
\DeclareMathAlphabet\mathzapf       {T1}{pzc} {mb} {it}
\DeclareMathAlphabet\mathchorus     {T1}{qzc} {m} {n}
\DeclareMathAlphabet\mathrsfso      {U}{rsfso}{m}{n}

\usepackage[linesnumbered,ruled,vlined,boxed]{algorithm2e}
\DontPrintSemicolon 
\SetKw{Continue}{continue}

\newcommand{\floor}[1]{\lfloor #1 \rfloor}

\usepackage{tikz}
\usetikzlibrary{
calc,
backgrounds,
mindmap,
patterns
}

\newcommand*{\equal}{=}

\newcommand{\convexpath}[2]{
  [   
  create hullcoords/.code={
    \global\edef\namelist{#1}
    \foreach [count=\counter] \nodename in \namelist {
      \global\edef\numberofnodes{\counter}
      \coordinate (hullcoord\counter) at (\nodename);
    }
    \coordinate (hullcoord0) at (hullcoord\numberofnodes);
    \pgfmathtruncatemacro\lastnumber{\numberofnodes+1}
    \coordinate (hullcoord\lastnumber) at (hullcoord1);
  },
  create hullcoords
  ]
  ($(hullcoord1)!#2!-90:(hullcoord0)$)
  \foreach [
  evaluate=\currentnode as \previousnode using \currentnode-1,
  evaluate=\currentnode as \nextnode using \currentnode+1
  ] \currentnode in {1,...,\numberofnodes} {
    let \p1 = ($(hullcoord\currentnode) - (hullcoord\previousnode)$),
    \n1 = {atan2(\y1,\x1) + 90},
    \p2 = ($(hullcoord\nextnode) - (hullcoord\currentnode)$),
    \n2 = {atan2(\y2,\x2) + 90},
    \n{delta} = {mod(\n2-\n1,360) - 360}
    in 
    {arc [start angle=\n1, delta angle=\n{delta}, radius=#2]}
    -- ($(hullcoord\nextnode)!#2!-90:(hullcoord\currentnode)$) 
  }
}

\usepackage{tabularx}

\newcolumntype{C}[1]{>{\centering\let\newline\\\arraybackslash\hspace{0pt}}m{#1}}
\newcolumntype{L}[1]{>{\raggedright\let\newline\\\arraybackslash\hspace{0pt}}m{#1}}
\newcolumntype{R}[1]{>{\raggedleft\let\newline\\\arraybackslash\hspace{0pt}}m{#1}}

\usepackage{doi}
\usepackage[autostyle]{csquotes}
\usepackage[style=numeric, citestyle=numeric-comp, 
	maxcitenames=2, maxbibnames=99, 
    backend=bibtex
]{biblatex}
\makeatletter
\def\myAbovefromStoT@overarrow@#1#2#3{\vbox{\ialign{##\crcr{\scriptsize{$s$}}#1#2{\scriptsize$t$}\crcr\noalign{\nointerlineskip}$\m@th \hfil #2#3\hfil $\crcr}}}
\newcommand{\myOverArrowFromStoT}{%
  \mathpalette{\small\myAbovefromStoT@overarrow@\vectfill@}}
\def\vectfill@{\arrowfill@\relbar\relbar{\raisebox{-3.81pt}[\p@][\p@]{$\mathord\mathchar"017E$}}}
\def\myAbovefromTtoS@overarrow@#1#2#3{\vbox{\ialign{##\crcr{\scriptsize{$t$}}#1#2{\scriptsize$s$}\crcr\noalign{\nointerlineskip}$\m@th \hfil #2#3\hfil $\crcr}}}
\newcommand{\myOverArrowFromTtoS}{%
  \mathpalette{\small\myAbovefromTtoS@overarrow@\vectfill@}}
\def\vectfill@{\arrowfill@\relbar\relbar{\raisebox{-3.81pt}[\p@][\p@]{$\mathord\mathchar"017E$}}}
\def\myAbovefromStoI@overarrow@#1#2#3{\vbox{\ialign{##\crcr{\scriptsize{$s$}}#1#2{\scriptsize$t^\prime$}\crcr \noalign{\nointerlineskip}$\m@th \hfil #2#3\hfil $\crcr}}}
\newcommand{\myOverArrowFromStoI}{%
  \mathpalette {\small\myAbovefromStoI@overarrow@\vectfill@}}
\def\vectfill@{\arrowfill@\relbar\relbar{\raisebox{-3.81pt}[\p@][\p@]{$\mathord\mathchar"017E$}}}

\def\myAbovefromItoT@overarrow@#1#2#3{\vbox{\ialign{##\crcr{\scriptsize{$t^\prime$}}#1#2{\scriptsize{$t$}}\crcr \noalign{\nointerlineskip}$\m@th \hfil #2#3\hfil $\crcr}}}
\newcommand{\myOverArrowFromItoT}{%
  \mathpalette {\small\myAbovefromItoT@overarrow@\vectfill@}}
\def\vectfill@{\arrowfill@\relbar\relbar{\raisebox{-3.81pt}[\p@][\p@]{$\mathord\mathchar"017E$}}}

\def\mySimpleAbove@overarrow@#1#2#3{\vbox{\ialign{##\crcr{\scriptsize{}}#1#2{\scriptsize{}}\crcr \noalign{\nointerlineskip}$\m@th \hfil #2#3\hfil $\crcr}}}
\newcommand{\mySimpleOverArrow}{%
  \protect\mathpalette {\small\mySimpleAbove@overarrow@\vectfill@}}
\def\vectfill@{\arrowfill@\relbar\relbar{\raisebox{-3.81pt}[\p@][\p@]{$\mathord\mathchar"017E$}}}
\makeatother

\usepackage[d]{esvect} 

\newcommand{\myOverArrow}[4]
{
  \ifnum`#2=115 
    \ifnum`#3=116 
      \ifx&#4&
        \protect\myOverArrowFromStoT{#1}
      \else
        \protect\myOverArrowFromStoT{#1}_{\hspace{-0.1cm}#4}
      \fi
    \else
      \ifx&#4&
        \protect\myOverArrowFromStoI{#1}
      \else
        \protect\myOverArrowFromStoI{#1}_{\hspace{-0.1cm}#4}
      \fi
    \fi
  \else
    \ifnum`#2=116 
        \ifx&#4&
          \protect\myOverArrowFromTtoS{#1}
        \else
          \protect\myOverArrowFromTtoS{#1}_{\hspace{-0.1cm}#4}
        \fi
    \else
        \ifx&#4&
          \protect\myOverArrowFromItoT{#1}
        \else
          \protect\myOverArrowFromItoT{#1}_{\hspace{-0.1cm}#4}
        \fi
    \fi
  \fi
}
      
\renewcommand{\vec}[1]{\mathbf{#1}}

\usepackage
	{todonotes}

\title{Efficient Enumeration of the Optimal Solutions to the Correlation Clustering problem}
\author{Nejat Ar\i n\i k, Rosa Figueiredo \& Vincent Labatut}

\hypersetup{
    pdftitle={\thetitle},
    pdfauthor={\theauthor},
    bookmarksnumbered=true,bookmarksopen=true,
	unicode=true,colorlinks=true,linktoc=all,
	linkcolor=blue,citecolor=blue,filecolor=blue,urlcolor=blue,
	pdfstartview=FitH
}

\usepackage{enumitem}
\setlist{nolistsep}

\begin{document}
\maketitle
\sloppy

\abstract{According to the \textit{structural balance} theory, a signed graph is considered structurally balanced when it can be partitioned into a number of modules such that positive and negative edges are respectively located inside and between the modules. In practice, real-world networks are rarely structurally balanced, though. In this case, one may want to measure the \textit{magnitude} of their imbalance, and to identify the set of edges causing this imbalance. The correlation clustering (CC) problem precisely consists in looking for the signed graph partition having the least imbalance. Recently, it has been shown that the space of the optimal solutions of the CC problem can be constituted of numerous and diverse optimal solutions. Yet, this space is difficult to explore, as the CC problem is NP-hard, and exact approaches do not scale well even when looking for a \textit{single} optimal solution. To alleviate this issue, in this work we propose an efficient enumeration method allowing to retrieve the complete space of optimal solutions of the CC problem. It combines an exhaustive enumeration strategy with neighborhoods of varying sizes, to achieve computational effectiveness. Results obtained for middle-sized networks confirm the usefulness of our method.}

\textbf{Keywords:} Correlation Clustering, Enumeration Strategy, Local Search, Optimal Solution Space, Signed Graph, Structural Balance.


\textcolor{red}{\textbf{Cite as:} N. Ar\i n\i k, R. Figueiredo \& V. Labatut. Efficient Enumeration of the Optimal Solutions to the Correlation Clustering problem. Journal of Global Optimization, 2023 (forthcoming)}

\section{Introduction}
\label{sec:Intro}

Many real-world social systems involve antagonistic relations, a feature that can be directly modeled through signed networks, which contain both positive and negative edges. Alternatively, in some cases, signed edges can be extracted from originally unsigned relations~\cite{Gursoy2020}. The exact semantic of these edges depends on the considered system: friendship/foe~\cite{Kunegis2009a} and trust/distrust~\cite{Massa2005} in social networks, inhibition/activation in biological networks~\cite{DasGupta2007,Tan2016a}, agreement/disagreement in political vote networks~\cite{Doreian2013,Arinik2017,Arinik2020} and so on.

Determining the structural balance of a signed network is a key aspect in the investigation of the structure of polarized relations. According to the \textit{structural balance} theory, a signed graph is considered to be balanced if it can be partitioned into two~\cite{Cartwright1956} or more~\cite{Davis1967} modules (i.e. clusters), such that all positive edges are located inside these modules, whereas negative ones lie between them.

However, it is very rare for a real-world network to be \textit{perfectly} balanced, in which case one wants to assess the \textit{magnitude} and the \textit{causes} of the imbalance. For a given partition, this imbalance is traditionally measured by counting the number of \textit{frustrated edges}~\cite{Cartwright1956,Zaslavsky1987}, i.e. positive edges located in-between modules and negative ones located inside them. Computing the graph imbalance amounts to identifying the partition corresponding to the lowest imbalance measure over the space of all possible partitions. The optimal solution found exhibits the source of the imbalance, i.e., the frustrated edges. This minimization problem is known as the \textit{correlation clustering} (CC) problem, proven to be NP-hard~\cite{Bansal2002}.

When solving an instance of the CC problem, the standard approach in the literature is to find a \textit{single} solution and focus the rest of the analysis on it, as if it were the \textit{only} optimal solution~\cite{Figueiredo2013b,Queiroga2021}. Yet, as shown empirically, it is possible that several, and even many, alternative optimal solutions exist for the considered instance~\cite{Camm1996,Brusco2010a,Arinik2020b}. All these alternate solutions are equally relevant in terms of the objective function of the CC problem. Yet, they can be very different, in terms of how they partition the graph. One can then wonder how many solutions shape the space of optimal solutions, and if many, how different/diverse they are. 
From an application point of view, these are important questions to answer in order to identify how many of these solutions are more suitable to the situation at hand.


In order to answer these questions, we recently conducted an empirical study in our previous work~\cite{Arinik2020b}. Therein, we first enumerated the optimal solutions of a collection of instances of the CC problem through Integer Linear Programming (ILP) modeling~\cite{Demaine2006}. We applied the best state-of-the-art optimal solution enumeration method, as proposed by Danna \textit{et al.}~\cite{Danna2007} and incorporated in the commercial solver CPLEX~\cite{cplex12}. Then, we studied the obtained spaces of optimal solutions through complex network analysis. In that work, we found out that slightly imbalanced networks tend to have fewer and structurally more similar optimal solutions forming a single solution class, whereas a higher imbalance leads to many diverse solutions possibly forming multiple classes of solutions. Consequently, we concluded that it is of great importance that the decision-maker generates the full set, or as many optimal solutions as possible, for a given instance.

In such an analysis, guaranteeing the completeness of the space of optimal solutions requires employing exact approaches. However, it is well known that, due to their NP-hard nature, generic exact approaches able to solve most clustering problems (including the CC problem) are very costly in terms of execution time and do not scale well, even when looking for a \textit{single} optimal solution~\cite{Grotschel1989}. In that respect, the enumeration method by Danna \textit{et al.} is the best state-of-the-art enumeration method able to handle the CC problem so far. However, as it is a generic method designed to handle any combinatorial problem formulated in ILP, it cannot take full advantage of the problem knowledge. Consequently, the maximum number of vertices (i.e. graph order) that we could handle with their method in our previous work 
was only $36$. Moreover, the execution time of the enumeration process for such small graphs was very long (typically more than one day). This highlights the need for an efficient optimal space enumeration method for the CC problem.

In this work, we aim to obtain this efficiency by putting the problem knowledge into the enumeration process. Based on the high similarity observed empirically between optimal solutions belonging to the same class of solutions, it integrates into an exact approach a local search mechanism, which relies on the use of neighborhoods of different sizes. In order to accelerate the proposed enumeration method, we develop pruning strategies based on optimal conditions satisfied by partial solutions. Our main contribution is the combination of local search and pruning strategies, which gives rise to an efficient enumeration method. We present extensive computational experiments established for the proposed method, relying on sparse and dense signed networks.

The rest of the article is organized as follows. In Section~\ref{sec:RelatedWork}, we review the literature related to the exact enumeration of all optimal solutions for the CC problem. In Section~\ref{sec:CCproblem}, we give the formal definition of the CC problem. Next, we introduce our enumeration method in Section~\ref{sec:EnumarationCCspace}, and detail our complete neighborhood search and pruning strategies in Section~\ref{sec:RecurrentNeighborhoodSearch} and~\ref{sec:PruningStrategies}, respectively. We put the proposed method into practice on a selection of partially and fully connected signed networks in Section~\ref{sec:Dataset} and discuss our results in Section~\ref{sec:Results}. Finally, we review our main findings in Section~\ref{sec:Conclusion}, and identify some perspectives for our work.

\section{Related Work}
\label{sec:RelatedWork}
There are several methods proposed in the literature to solve the CC problem exactly~\cite{Demaine2006,Figueiredo2013b,Keuper2020}, however very few were designed to enumerate the whole space of its optimal solutions. In fact, very few methods were proposed to enumerate all the optimal solutions of \textit{any} combinatorial problem (e.g., \cite{Arthur1997} for maximal covering problem). In this section, we review the existing optimal solution enumeration methods for the CC problem and for related problems.

The literature provides two main approaches to enumerate all optimal solutions of a combinatorial problem: \textit{Branch-and-Bound} and \textit{Fixed-Parameter Enumeration}. 
In the case of the CC problem, most works focus on the former.

\paragraph{Branch-and-Bound Approach} The enumeration of all optimal solutions of the CC problem (like for any combinatorial problem) through a branch-and-bound method can be performed through two algorithmic methods: Integer Linear Programming (ILP) vs. \textit{ad hoc} branch-and-bound programming. Their main difference is that the former can tackle any problem translated in the language of mathematical programming through any industrial optimization solver, whereas in the latter the construction of the branch-and-bound tree relies on problem-specific idiosyncrasies. 

Based on how the branch-and-bound tree is used, the ILP approach can be implemented in two different ways. The first one relies on a simple sequential process: a slightly modified version of the original problem is solved as many times as there are solutions in the optimal solution space (like in~\cite{Arthur1997} for the maximal covering problem). At each iteration, already-found optimal solutions are sequentially added as constraints into the mathematical model, in order to exclude them during future iterations. However, the drawback of this approach is that a branch-and-bound tree needs to be built from scratch for each new optimal solution found. 
The second type of ILP implementation is called \textit{OneTree}, and was proposed by Danna \textit{et al.}~\cite{Danna2007}. Instead of the previous simple sequential approach, it relies on a more efficient two-step method that allows reducing the computational effort. As its name suggests, it constructs a single branch-and-bound tree. It first builds and explores the search tree in order to find efficiently the first optimal solution, and then enumerate all the alternative optimal solutions based on the same tree. This method is available in the industrial optimization solver CPLEX~\cite{cplex12}, and we explain in Section~\ref{subsec:enumeratingAllOptSol} how we use it for the CC problem. In our previous work~\cite{Arinik2020b}, we used this approach for generating the optimal solution space of 
complete random graphs up to $36$ vertices.

An \textit{ad hoc} B\&B programming method constructs the branch-and-bound tree differently from what we described before. For a given clustering problem, these methods  systematically construct partial solutions by assigning the vertices to one of the existing modules. This produces a search tree, whose branches correspond to assignments of vertices to modules, and whose nodes correspond to partial assignments. Similar to Danna \textit{et al.}~\cite{Danna2007}, Brusco and Steinly~\cite{Brusco2010a} propose a two-step method for generating all optimal CC solutions. They first identify the optimal objective function value by finding an optimal solution, then use the optimal value as input when building another branch-and-bound tree from scratch. As shown by Figueiredo \& Moura~\cite{Figueiredo2013b}, 
this method does not deal well with finding a single optimal solution for graphs with more than 20 vertices.

\paragraph{Fixed-Parameter Enumeration Approach}
This is also an \textit{ad hoc} method which
requires to transform a part of the considered problem into a new input parameter. This in turn guarantees to bound the overall running time as a function of an input parameter. Assuming that this input parameter is small, the parameterized enumeration is efficient in practice. Damaschke~\cite{Damaschke2010} proposes an FPT (Fixed-Parameter Tractable) algorithm to enumerate all optimal solutions in a given graph for the \textit{Cluster Editing} problem, which is equivalent to the CC problem when the input graph is complete and unweighted. Concretely, this FPT algorithm makes the number of frustrated edges in structural balance parameterized to solve the problem. However, as shown in our previous work~\cite{Arinik2020b}, a drawback of this approach is that the amount of imbalance, the input parameter, can be very large. Furthermore, Damaschke does not provide any computational results in his theoretical work.


\medskip
The computational results presented in the works mentioned in this section identify the ILP branch-and-bound as the more efficient method for exactly solving the CC problem. This is explained by a tightened initial linear relaxation of the ILP formulation and the quality of the cuts used. In this work, we apply an ILP branch-and-bound method for the complete enumeration of the optimal CC solutions, and propose a compromise strategy between the sequential approach and the \textit{OneTree} method from Danna \textit{et al.}~\cite{Danna2007}.

\section{Correlation Clustering Problem}
\label{sec:CCproblem}
Let us now introduce the notations necessary for defining the correlation clustering problem. Let $G=(V,E)$ be an \textit{undirected graph}, where $V$ and $E$ are the sets of vertices and edges, respectively. We note $n=|V|$ and $m=|E|$ the numbers of vertices (i.e. graph order) and edges, respectively. We assume that the graph contains no loops, i.e. edges connecting a vertex to itself. Moreover, $G[S]$ denotes the subgraph induced by vertex subset $S \subset V$.

Consider a function $s : E \rightarrow \{+,-\}$ that assigns a sign to each edge in $E$. An undirected graph $G$ together with a function $s$ is called a \textit{signed graph}, denoted by $G = (V, E, s)$. An edge $(u,v)$ is called negative if $s(u,v) = -$ and positive if $s(u,v) = +$. We note $E^-$ and $E^+$ the sets of negative and positive edges in the signed graph, respectively. Let also define the positive graph $G^+$ of a given signed graph $G$ as the subgraph $(V,E^+, \{+\})$ (i.e. same vertices, but only the positive edges). The entries $a_{uv}$ of the signed adjacency matrix $\textbf{A}$ associated with a signed graph $G$ are defined in Equation~\ref{eq:weight-definition}.
\begin{equation}
    a_{uv} = 
    \begin{cases} 
        1, & \mbox{if $(u,v) \in E^+$,} \\
        -1, & \mbox{if $(u,v) \in E^-$,} \\
        0, & \mbox{otherwise}.
    \end{cases}
    \label{eq:weight-definition}
\end{equation}

Let $P = \{M_1,...,M_\ell\}$ ($1 \leq \ell \leq n$) be an $\ell$-partition of $V$, i.e. a division of $V$ into $\ell$ non-overlapping and non-empty subsets $M_i$ ($1 \leq i \leq \ell$) called \textit{modules}. The partition $P$ is called a \textit{solution} of the CC problem for the given graph $G$. Given a solution $P$, an edge $(u,v)$ is called \textit{internal} if it is located inside a module, i.e., $u$ and $v$ belong to the same module. Likewise, an edge $(u,v)$ is called \textit{external} if it is located between any two modules, i.e., $u$ and $v$ belong to two different modules. Given $\sigma\in\{+,-\}$, we define the set of positive/negative (depending on $\sigma$) edges connecting two modules $M_i, M_j \in P$ as $E^\sigma(M_i,M_j) = \{(u,v)\ |\ (u,v) \in E^\sigma, u \in M_i\ \text{and}\ v \in M_j\}$. We also define $E(M_i,M_j) = E^{-}(M_i,M_j) \cup E^{+}(M_i,M_j)$ as the set of all edges connecting these modules. Similarly, we note $\Omega^\sigma(M_i,M_j) = \sum\limits_{(u,v) \in E^\sigma(M_i,M_j)} a_{uv}$ the \textit{signed} number of positive/negative edges connecting these modules: note that this value can be negative if $\sigma = -$. We also define $\Omega(M_i,M_j) = \Omega^{-}(M_i,M_j) + \Omega^{+}(M_i,M_j)$ as the signed sum of the edges connecting the same modules. It can also be negative, if there are more negative than positive edges between $M_i$ and $M_j$.

The \textit{Imbalance} $I(P)$ of a partition $P$ is defined as the total number of frustrated edges, i.e. the numbers of \textit{positive} edges located \textit{between} modules and of \textit{negative} edges located \textit{inside} them:

\begin{equation}
	I(P) = \sum_{1\leq i<j\leq \ell} \Omega^+(M_i,M_j) - \sum_{1\leq i\leq \ell} \Omega^-(M_i,M_i).
	\label{eq:IP}
\end{equation}

The objective of the CC problem is to minimize the number of frustrated edges. This problem can be formally described as follows.

\begin{problem}[CC problem]
	Given a signed graph $G=(V,E,s)$, the \textit{correlation clustering problem} consists in finding a partition $P$ of $V$ such that the imbalance $I(P)$ is minimized.
\end{problem}

To the best of our knowledge, this $NP$-hard minimization problem appears under this name for the first time in Bansal's paper~\cite{Bansal2002}, although it is addressed before in the literature (e.g.~\cite{Doreian1996}).

\section{Enumeration of Optimal CC Solutions}
\label{sec:EnumarationCCspace}

Any method that enumerates optimal solutions needs to find an initial optimal solution by optimally solving the problem at hand. This is why we first describe an efficient method for finding an optimal solution for the CC problem (Section~\ref{subsec:findingInitialOptSol}), before presenting the methods for identifying an alternative optimal solution (Section~\ref{subsec:findingAlternativeOptSol}) and enumerating all optimal solutions (Section~\ref{subsec:enumeratingAllOptSol}).

\subsection{Finding an Initial Optimal Solution}
\label{subsec:findingInitialOptSol}
As largely illustrated in the literature (e.g. \cite{Demaine2006}, \cite{Figueiredo2013b}), the CC problem can be modeled by means of an ILP proposed for the \textit{Uncapacitated Graph Clustering} problem~\cite{Mehrotra1998}. Note that the model can handle weighted graphs, but we use it on unweighted ones in the context of this article. 

For each pair of vertices $u,v \in V: u < v$, a binary variable is first defined to describe the relative module membership of pairs of vertices, i.e.,
\begin{equation}
    x_{uv} = 
    \begin{cases} 
        1, & \mbox{if $u$ and $v$ are in a same module,} \\
        0, & \mbox{otherwise}.
    \end{cases}
\end{equation}

Then, the ILP formulation of the CC problem for unweighted signed graphs is written as follows:
\begin{align}
    \text{Min} \quad & \sum_{u,v \in V:uv \in E^-} x_{uv} + \sum_{u,v \in V:uv \in E^+} (1- x_{uv}) \label{eq:cc_obj} \\
    \text{s.t.} \quad  & x_{uv} + x_{vr} - x_{ur} \leq 1 ,\quad \forall u<v<r \in V \label{eq:S1_cc2}
    \\
    & x_{uv} - x_{vr} + x_{ur} \leq 1 ,\quad \forall u<v<r \in V \label{eq:S1_cc3}
    \\ 
    & -x_{uv} + x_{vr} + x_{ur} \leq 1 ,\quad \forall u<v<r \in V \label{eq:S1_cc4}
    \\
    & x_{uv} \in \{0,1\} ,\quad \forall u,v \in V. \label{eq:S1_cc5}
\end{align}

The objective function in~\eqref{eq:cc_obj} looks for a feasible solution minimizing the imbalance defined by Equation~\eqref{eq:IP}. A feasible solution, i.e., decision variables corresponding to a partition, must satisfy the triangle inequalities \eqref{eq:S1_cc2}, \eqref{eq:S1_cc3} and \eqref{eq:S1_cc4}. Finally, \eqref{eq:S1_cc5} describes the domain constraints for the variables in the formulation. 

This ILP model is sufficient to find an optimal solution through an optimization solver, but it can be very time-consuming. One way to deal with this issue is to strengthen it through a \textit{cutting plane} approach~\cite{Nemhauser1999}. We use the 2-partition and the 2-chorded cycle valid inequalities as described by Grötschel and Wakabayashi in~\cite{Grotschel1990} (see Appendix~\ref{secapx:2Partition2ChordedCycleIneqs} for more details). In the cutting plane approach, we add these two tight valid inequalities (only) during the root relaxation phase as described by Ales \textit{et al.}~\cite{Ales2016a}, before proceeding to the construction of the branch-and-bound search tree. As reported in~\cite{Ales2016a} for a similar clustering problem, we also verified that adopting this cut strategy improves the overall processing time. There are other inequalities in the literature that can further tighten the LP relaxation for the CC problem. These are the \textit{$2$-chorded even wheel}~\cite{Grotschel1990} inequality and those defined in \cite{Oosten2001}. Nevertheless, in practice, together the 2-partition and the 2-chorded cycle inequalities describe a tight linear relaxation~\cite{Grotschel1990}.

Throughout this work, we denote $ILP(G)$ as the formulation defined by Equations~\eqref{eq:cc_obj}--\eqref{eq:S1_cc5} and \textit{B\&C(ILP(G))} as the branch-and-cut procedure described above. Also, we denote $ILP{+}(G)$ as the strengthened ILP obtained after executing $B\&C(ILP(G))$, i.e., the formulation obtained by adding to $ILP(G)$ all the cuts generated by $B\&C(ILP(G))$.
Finally, let $B\&B(ILP{+}(G))$ be the branch-and-bound procedure based on $ILP{+}(G)$. In the following, for the sake of simplicity, the term \textit{solution} refers to a graph partition, as well as a feasible solution to the formulation $ILP{+}(G)$.

\subsection{Finding an Alternative Optimal Solution}
\label{subsec:findingAlternativeOptSol}
An enumeration method needs to constantly find an optimal solution that would be different from those identified before, and stored in a set $S$. For this purpose, we need to define an extended model of $ILP{+}(G)$, that we denote $ILP{+}jump(G,S)$. 
 
Consider a partition $P$. Let $\vec{X^p} \in \{0,1\}^{n\times n}$ be the representative matrix of $P$ defined as: $x^p_{ij}=1$ if $i$ and $j$ belong to the same module in $P$; and $x^p_{ij}=0$ otherwise. $ILP{+}jump(G,S)$ includes the set of constraints defined in Equation~\eqref{eq:alternative_constr} on top of $ILP{+}(G)$:

\begin{equation}
    \label{eq:alternative_constr}
    \sum_{\substack{u,v \in V: u < v}} | x^p_{uv} - x_{uv} | > 0, \forall P \in S.
\end{equation}
We refer the reader to \cite{Fischetti2005} for the implementation details of these constraints.

We see that this extended formulation allows us to find an alternative optimal solution other than the ones in $S$. Furthermore, given an optimal solution $P \in S$, to ensure that this alternative solution has the same objective value as $I(P)$, we add the objective function in~\eqref{eq:cc_obj} as a constraint, as shown in Equation~\eqref{eq:optimality_constr}. 

\begin{equation}
    \label{eq:optimality_constr}
    \sum_{u,v \in V:uv \in E^-} x_{uv} + \sum_{u,v \in V:uv \in E^+} (1- x_{uv}) \leq I(P).
\end{equation}

Equation~\eqref{eq:alternative_constr} along with Equation~\eqref{eq:optimality_constr} ensures that each feasible solution to $ILP{+}jump(S)$ is an optimal solution to the original problem. Finally, we denote the corresponding branch-and-bound procedure based on formulation $ILP{+}jump(G,S)$ by $B\&B(ILP{+}jump(G,S))$.

\subsection{Enumerating All Optimal Solutions}
\label{subsec:enumeratingAllOptSol}

As we have mentioned before, the CC problem can have multiple optimal solutions. In this section, we are interested in methods enumerating all optimal solutions of the CC problem for a given signed graph $G$. In the following, we first present $OneTreeCC$~\cite{Danna2007}, which is the best general method to enumerate solutions available in the literature. It is incorporated in CPLEX~\cite{cplex12}, and we apply it to the formulation $ILP{+}(G)$. Then, we introduce our \textit{ad hoc} method $EnumCC$, in the aim of obtaining a faster method able to handle relatively larger graphs.

We start with $OneTreeCC(G)$, which takes as input a signed graph $G$. The sketch of this two-step method is shown in Algorithm~\ref{algo:OneTreeCC}. In the first step (line~\ref{algoline:firstOptSol_OneTreeCC}), it finds an initial optimal solution based on $ILP{+}(G)$. We note $T_{B\&B}$ the branch-and-bound search tree constructed during this step. The search tree as well as the nodes considered at this first step are stored for further examination during the second step. Moreover, Danna \textit{et al.} completely turn off dual tightening in the branch-and-bound procedure in order to guarantee exhaustive enumeration. According to the authors, this fact can have a negative impact on performance, in terms of computational time. In the second step (line~\ref{algoline:otherOptSols_OneTreeCC}), $OneTreeCC(G)$ enumerates the set $S$ of all alternative optimal solutions by expanding the search tree $T_{B\&B}$ until obtaining the complete set of all optimal solutions.

\begin{algorithm}[!ht]
\SetAlgoLined
\KwResult{The set of all optimal solutions $S$}
$S = \emptyset$ \;
\tcc{1\textsuperscript{st} step}
Solve $B\&B(ILP{+}(G))$ to obtain an optimal solution $P$, and let $T_{B\&B}$ be the constructed branch-and-bound search tree \; \label{algoline:firstOptSol_OneTreeCC}

\tcc{2\textsuperscript{nd} step}
Expand fathomed nodes in $T_{B\&B}$ until enumerating the set $S$ of all alternative optimal solutions. \; \label{algoline:otherOptSols_OneTreeCC}
\caption{$OneTreeCC(G)$}
\label{algo:OneTreeCC}
\end{algorithm}

To compete with $OneTreeCC(G)$, we propose a new method for the CC problem in order to completely enumerate the optimal partitions of a given signed graph. We call it $EnumCC(G,r_{max})$, and it takes as input a signed graph $G$ and a maximum distance parameter $r_{max}$. As shown in Algorithm~\ref{algo:EnumCC}, it can be viewed as an improved version of the sequential approach proposed by Arthur \textit{et al.}~\cite{Arthur1997} for the Maximal Covering problem (described in Section~\ref{sec:RelatedWork}). In the first step (line~\ref{algoline:firstOptSol_EnumCC}), $EnumCC(G,r_{max})$ obtains an initial optimal solution thanks to $B\&B(ILP{+}(G))$. In the second step, instead of directly "\textit{jumping}" onto undiscovered optimal solutions one by one through $ILP{+}jump(G,S)$ (like in the sequential approach), we slightly change this process. Our method first discovers the recurrent neighborhood $RN_{\leq r_{max}}(P)$ of the current optimal solution $P$ (line~\ref{algoline:RN_EnumCC}), with the hope of discovering new optimal solutions. The recurrent neighborhood $RN_{\leq r_{max}}(P)$ of an optimal solution $P$, is the set of optimal solutions reached directly or indirectly from $P$, depending on the maximum distance parameter $r_{max}$. The complete description of the Recurrent Neighborhood Search (RNS) is given in Section~\ref{sec:RecurrentNeighborhoodSearch}. Whether a new solution is found or not through $RNS(G,P,r_{max})$, the method then moves on to the next step, consisting in jumping onto a new solution $P$ (line~\ref{algoline:jump_EnumCC}). These RNS and jumping phases are repeated, until all optimal solutions are discovered.

Clearly, $EnumCC(G,r_{max})$ can only improve the sequential approach as used in \cite{Arthur1997}, provided that $RNS$ discovers at least one new solution and that its execution time is shorter than that of $B\&B(ILP{+}jump(G,S))$. Furthermore, the choice of $r_{max}$ is sensitive: it does not contribute much to the method when it is small, whereas the method becomes slower than the sequential approach when $r_{max}$ is large.

\begin{algorithm}[!ht]
\SetAlgoLined
\KwResult{The set of all optimal solutions $S$}
$S = \emptyset$ \;
\tcc{1\textsuperscript{st} step}
Solve $B\&B(ILP{+}(G))$ to obtain an optimal solution $P$ \; \label{algoline:firstOptSol_EnumCC}
\tcc{2\textsuperscript{nd} step}
\While{$P$ $\neq$ null}{
    \tcc{step 2.1}
    Apply $RNS(G,P,r_{max})$ to generate the recurrent neighborhood $RN_{\leq r_{max}}(P)$ of $P$ \; \label{algoline:RN_EnumCC}
    $S = S \bigcup RN_{\leq r_{max}}(P)$ \;
    \tcc{step 2.2}
    Jump onto an undiscovered optimal solution $P$, with $P \notin S$, through $B\&B(ILP{+}jump(G,S))$ // if not possible, then $P = null$\; \label{algoline:jump_EnumCC}
}
\caption{$EnumCC(G,r_{max})$}
\label{algo:EnumCC}
\end{algorithm}

The jumping phase in step 2.2 of $EnumCC(G,r_{max})$ can be time-consuming for two reasons. First, even finding an alternative solution can be costly. Second, the number of jumps, denoted by the notation $n_{jump}(.)$ (here, $n_{jump}(EnumCC(G,r_{max}))$), can be very large, despite the use of the neighborhood $RN_{\leq r_{max}}(P)$. Nonetheless, we expect to save time and compensate the time spent in the jumping phase thanks to $RN_{\leq r_{max}}(P)$. 

In the rest of this work, we leave out the common parameter $G$ from the notations of the mentioned methods and ILP models, for the sake of convenience: OneTreeCC(), $EnumCC(r_{max})$ and $ILP{+}jump(S)$. 

In the next section, we detail the recurrent neighborhood search $RNS$ used in Algorithm~\ref{algo:EnumCC}.

\section{Recurrent Neighborhood Search (RNS)}
\label{sec:RecurrentNeighborhoodSearch}
\textit{Recurrent Neighborhood Search (RNS)} is the common and undoubtedly the most important part of our method \textit{EnumCC}. As mentioned in the previous section, if we want this method to compete with $OneTreeCC(G)$, \textit{RNS} needs to be efficient and fast.

Before defining the neighborhood of a solution, i.e. of a partition, we need to select a distance function between partitions. In the following, we first present the edit distance (Section~\ref{subsec:ED}), then introduce the algorithmic details of the \textit{Complete Neighborhood Search (CoNS)} used in our Recurrent Neighborhood Search (Section~\ref{subsec:CoNS}). Finally, we explain the main procedure for \textit{Recurrent Neighborhood Search} (Section~\ref{subsec:RNS}).

Let us first introduce some definitions used in the rest of the article. A partition $P$ with $\ell$ modules can be associated with one or more \textit{membership vectors}. Such a vector of size $n$, denoted by $\pi$, defines a function over $V \rightarrow \{ 1, 2, ..., \ell \}$. For a given vertex $u$, $\pi(u)$ denotes the \textit{label} of the module of $P$ that contains $u$, while $M_{\pi(u)}$ denotes the module itself. Finally, given a partition $P$, we define an $r$-neighborhood structure, denoted by $N_r(P)$, as the family of partitions obtained by moving $r$ vertices into different modules of $P$. We also define $N_r(\pi)$ as the family of corresponding membership vectors obtained by moving $r$ vertices.

\subsection{Edit Distance}
\label{subsec:ED}
In Natural Language Processing, and more generally in Computer Science, the edit distance~\cite{Jurafsky2000} 
is defined as the minimum number of edit operations required to transform one string into another (or more precisely, their total cost). In the context of graph partitioning, with fixed vertices and edges, edit operations are used as \textit{neighborhood search operators} in local search heuristics developed for graph problems~\cite{Blum2003,Ma2016,Ahuja2002}. Such an operation consists in moving one or more vertices (called \textit{moving vertices}) from their current modules (called \textit{source} modules) to other ones (called \textit{target} modules). From the perspective of the membership vectors corresponding to the concerned partitions, this transforms a \textit{source} vector into a \textit{target} one. The number of moving vertices constitutes the cost of the edit operation. In this work, we are interested in edit operations whose cost is minimal.
\begin{definition}[Min-edit operation]
    Consider an edit operation transforming a source membership vector into a target membership one. If the set of moving vertices is the minimum set required for this transformation, then it is a \textbf{min-edit operation}. Otherwise, we call it \textbf{non-min-edit operation}.
\end{definition}

Notice that the cost of an edit operation is not always minimal. This is because two membership vectors, i.e. the module labels of two partitions, can be very different, but essentially suggest very similar module assignments for the vertices. The distinction between \textit{min-edit} and \textit{non-min-edit} operations is illustrated in Figure~\ref{fig:IllustrativeExampleEditForNotations}.

Importantly, the \textit{edit distance} between two membership vectors is the cost of the min-edit operation allowing to turn one vector into the other.
The calculation of such distance between two membership vectors can be done in polynomial time by solving an assignment problem (see Appendix~\ref{secapx:CalculatingEditDistance} for more details).

Let us now introduce our notations related to the edit distance.
Let $\boldsymbol\pi^s$ and $\boldsymbol\pi^t$ represent the source and target membership vectors for an edit operation, respectively. They are associated with the source and target partitions $P^s = \{ M_{1}^{s}, M_{2}^{s}, ..., M_{\ell^s}^{s} \}$ and $P^t = \{ M_{1}^{t}, M_{2}^{t}, ..., M_{\ell^t}^{t} \}$, containing $\ell^s$ and $\ell^t$ modules, respectively. Since the edit distance is symmetric, without loss of generality, let $\ell^s \leq \ell^t$. Moreover, we note $\pi^s \rightarrow \pi^t$ an edit operation applied onto $P^s$ to obtain $P^t$, whose set of moving vertices is defined as $\myOverArrow{V}{s}{t}{} = \{u\ |\ \pi^s(u)\neq \pi^t(u)\}$. This means that the remaining vertices, called \textit{non-moving} vertices, do not change modules, hence their module labels are the same in $\pi^s$ and $\pi^t$. The cost of this edit operation is simply the number of vertices whose module has changed. We denote this cost with $cost(\pi^s \rightarrow \pi^t)$ and compute it by taking the cardinality of $\myOverArrow{V}{s}{t}{}$. In the rest of the text, we use $r$ for short whenever we are interested only in the value of $cost(\pi^s \rightarrow \pi^t)$. Also, a \textit{$r$-edit operation} (resp. \textit{min-$r$-edit operation}) is any edit operation (resp. min-edit operation) whose cost equals $r$. Consider an edit operation $\pi^s \rightarrow \pi^t$ whose moving vertices are in set $\myOverArrow{V}{s}{t}{}$. The distinct labels of the source and target modules of a subset $V'\subset V$ of vertices are denoted, respectively, by $\pi^s(V') = \bigcup_{u \in V'} \pi^s(u)$ and $\pi^t(V') = \bigcup_{u \in V'} \pi^t(u)$. Finally, we sometimes need to refer to several modules in an edit operation $\pi^s \rightarrow \pi^t$. For this reason, we also define $M^s_{L} = \bigcup_{l \in L} M^s_l$ for any $L \subseteq \{1, .., \ell^s\}$ in the source partition $P^s$ and $M^t_{L} = \bigcup_{l \in L} M^t_l$ for any $L \subseteq \{1, .., \ell^t\}$ in the target partition $P^t$.

To illustrate the aforementioned concepts, let us consider the example of Figure~\ref{fig:IllustrativeExampleEditForNotations}. As shown in Figure~\ref{fig:IllustrativeExampleEditForNotations}a, $\pi^s = (1,1,2,2,2)$, where $M_{1}^{s} = \{v_1,v_2\}$ and $M_{2}^{s} = \{v_3,v_4,v_5\}$. Let us consider an edit operation $\pi^s \rightarrow \pi^t$ with the moving vertices $\myOverArrow{V}{s}{t}{} = \{v_2,v_4\}$ (Figure~\ref{fig:IllustrativeExampleEditForNotations}b). The edit operation consists in moving $v_2$ into $M^{s}_2$ and $v_4$ into $M^{s}_1$, which produces $\pi^t$. Hence, we have $M^{t}_1 = \Bigl(M^{s}_1 \setminus \{v_2\}\Bigr) \cup \{v_4\}$ and $M^{t}_2 = \Bigl(M^{s}_2 \setminus \{v_4\}\Bigr) \cup \{v_2\}$. The labels of the source and target distinct modules of $\myOverArrow{V}{s}{t}{}$ are $\pi^s(\myOverArrow{V}{s}{t}{}) = \{1,2\}$ and $\pi^t(\myOverArrow{V}{s}{t}{}) = \{1,2\}$, respectively. Also, we have $M^s_{\pi^s(\myOverArrow{V}{s}{t}{})} = M^s_1 \cup M^s_2$ and $M^t_{\pi^t(\myOverArrow{V}{s}{t}{})} = M^t_1 \cup M^t_2$
for this example. The sets $M^s_{\pi^s(\myOverArrow{V}{s}{t}{})}$ and $M^t_{\pi^t(\myOverArrow{V}{s}{t}{})}$ could be different, if some of the moving vertices move into a module $M^s_i$ other than $M^s_{\pi^s(\myOverArrow{V}{s}{t}{})}$.

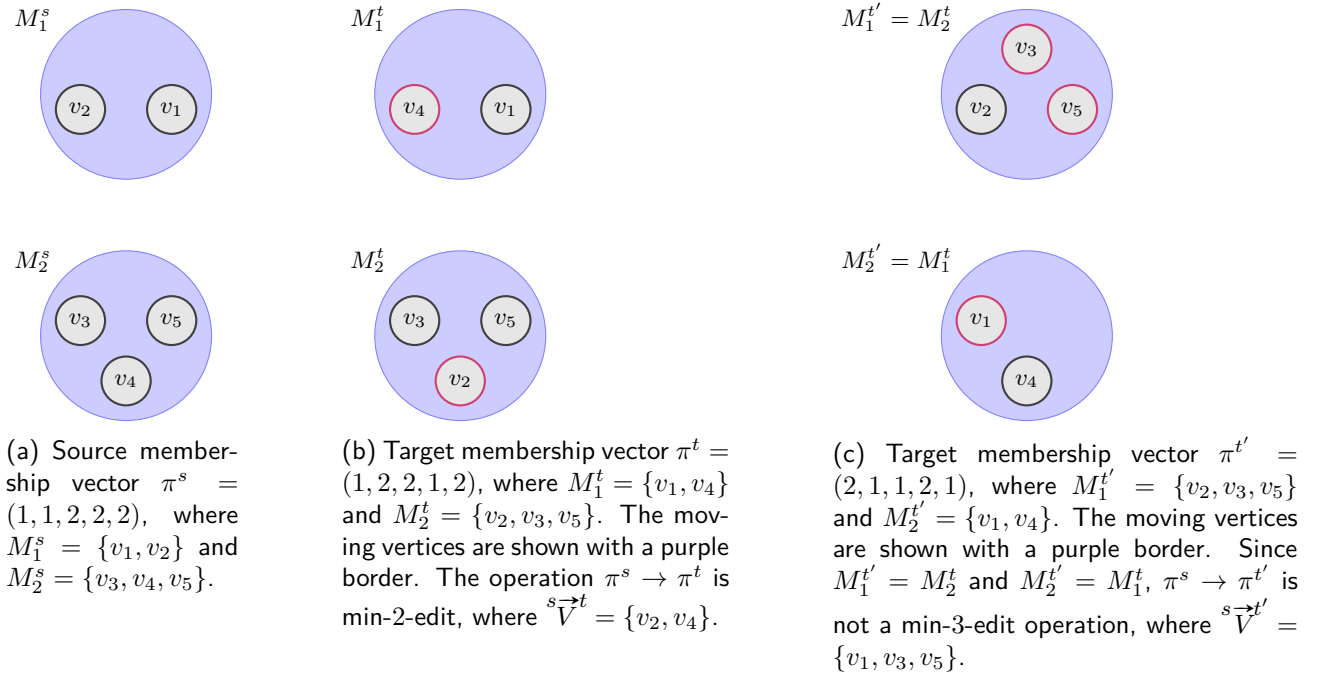
\begin{figure}[!h]
    \centering
    \begin{subfigure}[t]{0.18\linewidth}
    \begin{tikzpicture}[scale=0.80]
        \footnotesize
		\tikzstyle{nn}=[circle,thick,draw=black!75,fill=black!10,minimum size=3mm]
		\tikzstyle{clu}=[circle,draw=blue!50,fill=blue!20,minimum size=2.25cm,inner sep=0pt]
		\tikzstyle{pl}=[thick,draw=green!55!black!75,text=green!55!black!75]
		\tikzstyle{nl}=[thick,draw=red,text=red]
		
        \node (cluster1) at (-2,7.75) [clu, label=140:$M_{1}^{s}$] {};
        \node (cluster2) at (-2,3.75) [clu, label=140:$M_{2}^{s}$] {};
                
		\node[nn] (v1) at (-1.25,7.5) {$v_{1}$};
	    \node[nn] (v2) at (-2.75,7.5) {$v_{2}$};
		\node[nn] (v3) at (-2.75,4) {$v_{3}$};
		\node[nn] (v4) at (-2,3) {$v_{4}$};
		\node[nn] (v5) at (-1.25,4) {$v_{5}$};
		
	\end{tikzpicture}
	\caption{Source membership vector $\pi^s = (1,1,2,2,2)$, where $M_{1}^{s} = \{v_1,v_2\}$ and $M_{2}^{s} = \{v_3,v_4,v_5\}$.}
	\end{subfigure}
	\hfill
	\begin{subfigure}[t]{0.3\linewidth}
    \begin{tikzpicture}[scale=0.80]
        \footnotesize
		\tikzstyle{nn}=[circle,thick,draw=black!75,fill=black!10,minimum size=3mm]
		\tikzstyle{nn2}=[circle,thick,draw=purple!75,fill=black!10,minimum size=3mm]
		\tikzstyle{clu}=[circle,draw=blue!50,fill=blue!20,minimum size=2.25cm,inner sep=0pt]
		\tikzstyle{pl}=[thick,draw=green!55!black!75,text=green!55!black!75]
		\tikzstyle{nl}=[thick,draw=red,text=red]
		
        \node (cluster1) at (-2,7.75) [clu, label=140:$M_{1}^{t}$] {};
        \node (cluster2) at (-2,3.75) [clu, label=140:$M_{2}^{t}$] {};
                
		\node[nn] (v1) at (-1.25,7.5) {$v_{1}$};
	    \node[nn2] (v4) at (-2.75,7.5) {$v_{4}$};
		\node[nn] (v3) at (-2.75,4) {$v_{3}$};
		\node[nn2] (v2) at (-2,3) {$v_{2}$};
		\node[nn] (v5) at (-1.25,4) {$v_{5}$};
		
	\end{tikzpicture}
	\caption{Target membership vector $\pi^t = (1,2,2,1,2)$, where $M_{1}^{t} = \{v_1,v_4\}$ and $M_{2}^{t} = \{v_2,v_3,v_5\}$. The moving vertices are shown with a purple border. The operation $\pi^s \rightarrow \pi^{t}$ is min-$2$-edit, where $\myOverArrow{V}{s}{t}{} = \{v_2,v_4\}$.}
	\end{subfigure}
	\hfill
    \begin{subfigure}[t]{0.36\linewidth}
    \begin{tikzpicture}[scale=0.80]
        \footnotesize
		\tikzstyle{nn}=[circle,thick,draw=black!75,fill=black!10,minimum size=3mm]
		\tikzstyle{nn2}=[circle,thick,draw=purple!75,fill=black!10,minimum size=3mm]
		\tikzstyle{clu}=[circle,draw=blue!50,fill=blue!20,minimum size=2.25cm,inner sep=0pt]
		\tikzstyle{pl}=[thick,draw=green!55!black!75,text=green!55!black!75]
		\tikzstyle{nl}=[thick,draw=red,text=red]
		
        \node (cluster1) at (-2,7.75) [clu, label=140:$M_{1}^{t'} \equal M_{2}^{t}$] {};
        \node (cluster2) at (-2,3.75) [clu, label=140:$M_{2}^{t'} \equal M_{1}^{t}$] {};
                
		\node[nn] (v2) at (-2.75,7.5) {$v_{2}$};
		\node[nn2] (v3) at (-2,8.5) {$v_{3}$};
		\node[nn2] (v5) at (-1.25,7.5) {$v_{5}$};
		\node[nn] (v4) at (-2,3) {$v_{4}$};
	    \node[nn2] (v1) at (-2.75,4) {$v_{1}$};

	\end{tikzpicture}
	\caption{Target membership vector $\pi^{t'} = (2,1,1,2,1)$, where $M_{1}^{t'} = \{v_2,v_3,v_5\}$ and $M_{2}^{t'} = \{v_1,v_4\}$. The moving vertices are shown with a purple border. Since $M_{1}^{t'} = M_{2}^{t}$ and $M_{2}^{t'} = M_{1}^{t}$, $\pi^s \rightarrow \pi^{t'}$ is not a min-$3$-edit operation, where $\myOverArrow{V}{s}{i}{} = \{v_1,v_3,v_5\}$.}
	\end{subfigure}
	\caption{Illustrative example for two edit operations applied to the same source membership vector $\pi^s$, whose associated partition is shown in Fig.~\ref{fig:IllustrativeExampleEditForNotations}a. The first operation, $\myOverArrow{V}{s}{t}{}$, is shown in Fig.~\ref{fig:IllustrativeExampleEditForNotations}b, and the second, $\myOverArrow{V}{s}{i}{}$, in Fig.~\ref{fig:IllustrativeExampleEditForNotations}.c. 
	The underlying partition in $P^{t}$ and $P^{t'}$ is the same, but the corresponding membership vectors $\pi^{t}$ and $\pi^{t'}$ are different. Consequently, $\myOverArrow{V}{s}{t}{}$ is a \textit{min-edit} operation, whereas $\myOverArrow{V}{s}{i}{}$ is not.}
	\label{fig:IllustrativeExampleEditForNotations}
\end{figure}

\subsection{Complete Neighborhood Search (\textit{CoNS})}
\label{subsec:CoNS}

Our neighborhood search method $CoNS(G, \pi^s, r)$ takes as input a signed graph $G$, a membership vector $\pi^s$ associated with an optimal partition $P^s$ of $\ell$ modules and an edit distance $r$. It returns the set $N_{r}(\pi^s)$ of membership vectors associated with \textit{all} optimal partitions, obtained by applying min-$r$-edit operations to a given membership vector $\pi^s$. To ensure the completeness of set $N_{r}(\pi^s)$, $CoNS(G, \pi^s, r)$ analyses all combinations of moving vertices and their target modules. We use two pruning strategies in order to  eliminate some combinations that do not lead to an optimal partition when an edit operation is applied. Both of them are described in this section, whereas the properties upon which they are based are detailed in Section~\ref{sec:PruningStrategies}. 
In simple terms, they are based on the relations between a set of moving vertices and their target modules. In the following, we present $CoNS(G, \pi^s, r)$ as a pipeline procedure consisting of four parts, as shown in Algorithm~\ref{algo:CoNS}. 
\begin{algorithm}[!htp]
\SetAlgoLined
    \footnotesize
    \SetKwInOut{Input}{Input}
    \SetKwInOut{Output}{Output}
    \Input{signed graph $G(V, E, s)$, source membership vector $\pi^s$ with $\ell$ module labels, edit distance $r$}
    \Output{$N_{r}(\pi^s)$}
    $T = \{1, 2, .., \ell\}$ // module labels in $\pi^s$ \;
    \tcc{1\textsuperscript{st} part} 
    \For{each $\myOverArrow{V}{s}{t}{} \in \binom{V}{r}$}{\label{algoline:CoNS-part1}
        %
        \If{\textit{internalPruning}($G$, $\pi^s$, $\myOverArrow{V}{s}{t}{}$)}{ \label{algoline:CoNS-internalPruning}
            go to line~\ref{algoline:CoNS-part1} \;
        }
        \tcc{2\textsuperscript{nd} part} \label{algoline:CoNS-part2}
        $T_0 = \pi(\myOverArrow{V}{s}{t}{}) \bigcup\ \{\textit{Unknown}\}$ // initial target module labels \; \label{algoline:CoNS-initTargetModules} 
        $T_{rem} = T \setminus T_0$ // remaining module labels \label{algoline:CoNS-initRemainingTargetModules} \;
        \For(\tcp*[h]{$\mathcal{P}(.)$: permutations}){each $T'_0 \in \mathcal{P}(\binom{T_0}{r})$, s.t. $\pi^s(u) \notin T'_0, \forall u \in \myOverArrow{V}{s}{t}{}$}{\label{algoline:CoNS-selection_T0} 
            Let $\pi^t$ be the partition after assigning $T'_0$ to $\myOverArrow{V}{s}{t}{}$ 
            \; \label{algoline:CoNS-update1}
            \If{\textit{externalPruning}($G$, $\pi^s$, $\pi^t$, $\myOverArrow{V}{s}{t}{}$)}{ \label{algoline:CoNS-externalPruning1}
                go to line~\ref{algoline:CoNS-selection_T0} \;
            }
            \tcc{3\textsuperscript{rd} part}  \label{algoline:CoNS-part3}
            Let $B$ be the the number of \textit{Unknown} labels in $\pi^t$ \;
            \eIf{$B \neq 0$}{ \label{algoline:CoNS-atLeastOneUnknownTargetModule}
                Let $\myOverArrow{V}{s}{t}{}_{rem}$ be the moving vertices, whose $\pi^t(\myOverArrow{V}{s}{t}{}_{rem}) = $ \textit{Unknown} \;
                \For{$b \in \{1, .., B\}$}{ \label{algoline:CoNS-fix-b}
                    Let $\mathcal{I}$ be the set of all vectors of size $B$ with elements belonging to $\{1,..,b\}$, 
                    s.t. each element for each vector appears once \; \label{algoline:CoNS-couplingInfo}
                    \For{each $I \in \mathcal{I}$}{\label{algoline:CoNS-foreach-couplingInfo}
                        Update $\pi^t$ with the assignment of $I$ to $\pi^t(\myOverArrow{V}{s}{t}{}_{rem})$ \; \label{algoline:CoNS-update2}  
                        \If{\textit{externalPruning}($G$, $\pi^s$, $\pi^t$, $\myOverArrow{V}{s}{t}{}$)}{ \label{algoline:CoNS-externalPruning2}
                            go to line~\ref{algoline:CoNS-foreach-couplingInfo} \;
                        }
                        \tcc{4\textsuperscript{th} part} \label{algoline:CoNS-part4}
                        $T^+_{rem}= T_{rem} \cup \{\ell+1, .., \ell+b\}$ // expand for empty modules \;
                        \For{each $T'_r \in \mathcal{P}(\binom{T^+_r}{b})$}{ \label{algoline:CoNS-foreach-finalTargetModules}
                            Replace the assignment of $I$ to $\myOverArrow{V}{s}{t}{}_{rem}$ by $T_{rem}$ in $\pi^t$ \; \label{algoline:CoNS-finalTargetModules}
                           
                          \If{$I(\pi^s) = I(\pi^t)$ and $\neg$\textit{externalPruning}($G$, $\pi^s$, $\pi^t$, $\myOverArrow{V}{s}{t}{}$)}{ \label{algoline:CoNS-optimalityCond2}
                                $N_{r}(\pi^s) = N_{r}(\pi^s) \bigcup \pi^t$ \; 
                            }
                        }
                    }
                }
            } { \label{algoline:CoNS-knownTargetModules}
                \If{$I(\pi^s) = I(\pi^t)$}{ \label{algoline:CoNS-optimalityCond1}
                    $N_{r}(\pi^s) = N_{r}(\pi^s) \bigcup \pi^t$ \;
                }
            }
        }
    }
\caption{$CoNS(G, \pi^s, r)$.}
\label{algo:CoNS}
\end{algorithm}

In the first part, $CoNS(G, \pi^s, r)$ starts by selecting a candidate set of $r$ moving vertices $\myOverArrow{V}{s}{t}{}$ among all vertices in $V$ (line \ref{algoline:CoNS-part1}). The set of all possible $r$ moving vertices is obtained by generating all combinations $\binom{V}{r}$. 
We know that probably not every $\myOverArrow{V}{s}{t}{}$ leads to a min-$r$-edit operation and an optimal partition $P^t$. To detect such cases, 
we apply the first pruning procedure (line~\ref{algoline:CoNS-internalPruning}), called \textit{internalPruning} and depicted 
in Algorithm~\ref{algo:internalPruning}. 
This procedure checks, without the knowledge of target modules, if $\myOverArrow{V}{s}{t}{}$ satisfies some connectivity conditions related to the atomicity property (defined in Section~\ref{subsec:Decomposable_MED}) and the MVMO property up to three moving vertices (defined in Section~\ref{subsec:MVMO-property}). As we will see in Section~\ref{sec:PruningStrategies}, whenever one of these two conditions is not satisfied, the candidate set $\myOverArrow{V}{s}{t}{}$ can be pruned. 

\begin{algorithm}[!ht]
\SetAlgoLined
    \SetKwInOut{Input}{Input}
    \SetKwInOut{Output}{Output}
    \Input{signed graph $G(V, E, s)$, source membership vector $\pi^s$, moving vertices $\myOverArrow{V}{s}{t}{}$}
    \Output{Boolean variable}
     
    \tcc{1\textsuperscript{st} part}
    \If(\tcp*[h]{Properties~\ref{prop:EdgeConnectivity} and~\ref{prop:SpuriousEdgeConnectivity}a}){$\neg$\textit{intAtomic}($G$, $\pi^s$, $\myOverArrow{V}{s}{t}{}$)}{
        \Return true
    }
    
    \tcc{2\textsuperscript{nd} part}
    \If(\tcp*[h]{Lemma~\ref{lemma:MVMO-2Vertices}.1 and Lemma~\ref{lemma:MVMO-3Vertices}}){$2 \leq |\myOverArrow{V}{s}{t}{}| \leq 3$ and $\neg$\textit{intMVMO}($G$, $\pi^s$, $\myOverArrow{V}{s}{t}{}$)}{
        \Return true;
    }
    \Return false;
\caption{\textit{internalPruning}($G$, $\pi^s$, $\myOverArrow{V}{s}{t}{}$)}
\label{algo:internalPruning}
\end{algorithm}

At this point, the set of moving vertices $\myOverArrow{V}{s}{t}{}$ is fixed. Next, the algorithm defines step by step the target membership vector $\pi^t$ by setting the target module of each vertex in $\myOverArrow{V}{s}{t}{}$. We handle the process of generating the target membership vector $\pi^t$ in three steps, which constitutes the second, third and fourth parts of Algorithm~\ref{algo:CoNS}, in order to take advantage of our pruning strategy. The pruning strategy used in these parts is slightly different from that of the first part: it is based on external connections from $\myOverArrow{V}{s}{t}{}$ to $\myOverArrow{V}{s}{t}{}$. The concerned procedure is called \textit{externalPruning} and depicted in Algorithm~\ref{algo:externalPruning}. It verifies, based on external relations among $\myOverArrow{V}{s}{t}{}$, if $\myOverArrow{V}{s}{t}{}$ satisfies the min-edit (Section~\ref{subsec:NonMinimumEditOperationPruning}), atomicity (Section~\ref{subsec:Decomposable_MED}) and MVMO (Section~\ref{subsec:MVMO-property}) properties. As discussed in Section~\ref{sec:PruningStrategies}, whenever one of these three conditions is not satisfied, the possibility of target modules for candidate set $\myOverArrow{V}{s}{t}{}$ is pruned. 
These pruning strategies can be applied even when some target modules are not already decided. 
Therefore, they are invoked, whenever the target modules of $\myOverArrow{V}{s}{t}{}$ are updated (lines \ref{algoline:CoNS-externalPruning1}, \ref{algoline:CoNS-externalPruning2} and \ref{algoline:CoNS-optimalityCond2}).

In the second part, a moving vertex is allowed to move into a module, whose label is in $T_0$. The set $T_0$, defined at line \ref{algoline:CoNS-initTargetModules}, contains the labels of all unique source modules of $\myOverArrow{V}{s}{t}{}$, as well as a label \textit{Unknown}. We use this label \textit{Unknown} to indicate that the target module of a vertex is not already assigned. We also define $T_{rem}$ as the remaining modules. Notice that, from this point, a target membership vector $\pi^t$ can contain multiple vertices with target modules labeled as \textit{Unknown}. In the third part, this \textit{Unknown} label allows us to handle the other modules in $T_{rem}$ as target possibilities. At line~\ref{algoline:CoNS-selection_T0}, we consider all possible permutations in $\mathcal{P}(\binom{T_0}{r})$ such that, for each vertex in $\myOverArrow{V}{s}{t}{}$, the target and source modules are different.

At this point, the \textit{externalPruning} procedure can be applied for vertices whose target modules are already assigned. If the pruning conditions are not satisfied, the algorithm proceeds with the third part. Line~\ref{algoline:CoNS-atLeastOneUnknownTargetModule} checks if no target module in $\pi^t$ is unknown ($B = 0$). If so, then a new optimal partition has been found, and the algorithm goes on with the final test (line~\ref{algoline:CoNS-knownTargetModules}). However, if it is not the case, the algorithm enters the fourth part: all the combinations in $T_{rem}$ (more precisely, $T^+_r$) are considered as a possible assignment of unknown target modules. One can expect this task to be computationally expensive. Thus, instead of directly determining the target modules of each vertex in $\myOverArrow{V}{s}{t}{}_{rem}$, we first generate all the \textit{coupling} scenarios of moving vertices going into the same \textit{Unknown} target module. To handle this, let $b$ be the number of distinct \textit{Unknown} target modules (line~\ref{algoline:CoNS-fix-b}), for each value in the range of $\{1,..,B\}$. We build all possible \textit{coupling} scenarios $\mathcal{I}$ for the value $b$ (line~\ref{algoline:CoNS-couplingInfo}). We couple a pair of remaining moving vertices, if they move into the same \textit{Unknown} target module. For instance, $\{Unknown1,Unknown1,Unknown2\}$ indicates that the last moving vertex moves into a different \textit{Unknown} target module than the first two. Once the target membership vector $\pi^t$ is enriched with a \textit{coupling} scenario $\mathcal{I}$, the external pruning is repeated in line~\ref{algoline:CoNS-externalPruning2}. If the pruning conditions are not satisfied, the algorithm finally determines the target modules of $\myOverArrow{V}{s}{t}{}_{rem}$, by respecting the actual \textit{coupling} scenario $\mathcal{I}$ (line~\ref{algoline:CoNS-finalTargetModules}). In the end, we add the current $\pi^s \rightarrow \pi^t$ into $N_{r}(\pi^s)$, if the optimality condition is met (line~\ref{algoline:CoNS-optimalityCond2}).

\begin{algorithm}[!ht]
\SetAlgoLined
    \SetKwInOut{Input}{Input}
    \SetKwInOut{Output}{Output}
    \Input{signed graph $G(V, E, s)$, source membership vector $\pi^s$, target membership vector $\pi^t$, moving vertices $\myOverArrow{V}{s}{t}{}$}
    \Output{Boolean variable}
    
    \tcc{1\textsuperscript{st} part}
    \If(\tcp*[h]{Properties~\ref{prop:Non-min-edit-operation}, \ref{prop:atomic}, \ref{prop:SpuriousEdgeConnectivity}b and~\ref{prop:InteractionConnectivity}}){$\neg$\textit{minEdit}($\pi^s$, $\pi^t$, $\myOverArrow{V}{s}{t}{}$) or $\neg$\textit{extAtomic}($G$, $\pi^s$, $\pi^t$, $\myOverArrow{V}{s}{t}{}$)}{
        \Return true \;
    }
    
    \tcc{2\textsuperscript{nd} part}
    \If(\tcp*[h]{Lemmas~\ref{lemma:MVMO-2Vertices}.2, \ref{lemma:MVMO-3Vertices} and~\ref{lemma:MVMO-d-Vertices}}){$2 \leq |\myOverArrow{V}{s}{t}{}|$ and $\neg$\textit{extMVMO}($G$, $\pi^s$, $\pi^t$, $\myOverArrow{V}{s}{t}{}$)}{
        \Return true \;
    }

    \Return false;
\caption{\textit{externalPruning}($G$, $\pi^s$, $\pi^t$, $\myOverArrow{V}{s}{t}{}$)}
\label{algo:externalPruning}
\end{algorithm}


Without the problem-specific pruning strategies (internal and the external verification of properties from Section~\ref{subsec:MVMO-property}), Algorithm~\ref{algo:CoNS} amounts to a brute force method. 
Although $CoNS(G, \pi^s, r)$ can be very costly (even for small values of $r$), the properties described in Section~\ref{sec:PruningStrategies} allow us to develop a method whose cost is relatively lower than the brute force method.
Indeed, given a membership vector of size $n$ with $\ell$ module labels, the number of non-repetitive permutations of all the combinations of $r$ moving vertices is $n^r \times r!$ and the number of non-repetitive permutations of all the combinations of target modules is $\ell^r \times r!$. In the end, the brute force method is $\Theta(n^r \times r! \times \ell^r \times r!)$, where we assume that $\ell > r$. Computational results in Section~\ref{subsec:EvaluationMVMO} show that the pruning strategies reduce the computational cost of this procedure.

\subsection{Main Procedure}
\label{subsec:RNS}
The main procedure for \textit{Recurrent Neighborhood Search (RNS)} consists in applying \textit{CoNS} from Section~\ref{subsec:CoNS} in a recursive manner, in order to build a search tree called the \textit{RNS tree}. This recursive procedure is depicted in Figure~\ref{subfig:RNS-tree}. The root node of \textit{RNS tree} corresponds to the initial partition $P$. All other nodes of the RNS tree constitute the set of optimal solutions reached directly or indirectly from $P$. A solution associated with a node at some level $h$ of the \textit{RNS tree} is obtained after $h-1$ applications of the \textit{CoNS} procedure. The height of the \textit{RNS} tree is unknown but necessarily finite, as we will see next.

\begin{figure}
    \centering
	\begin{subfigure}[b]{0.65\linewidth}
    \centering
    \includegraphics[width=1.0\textwidth]{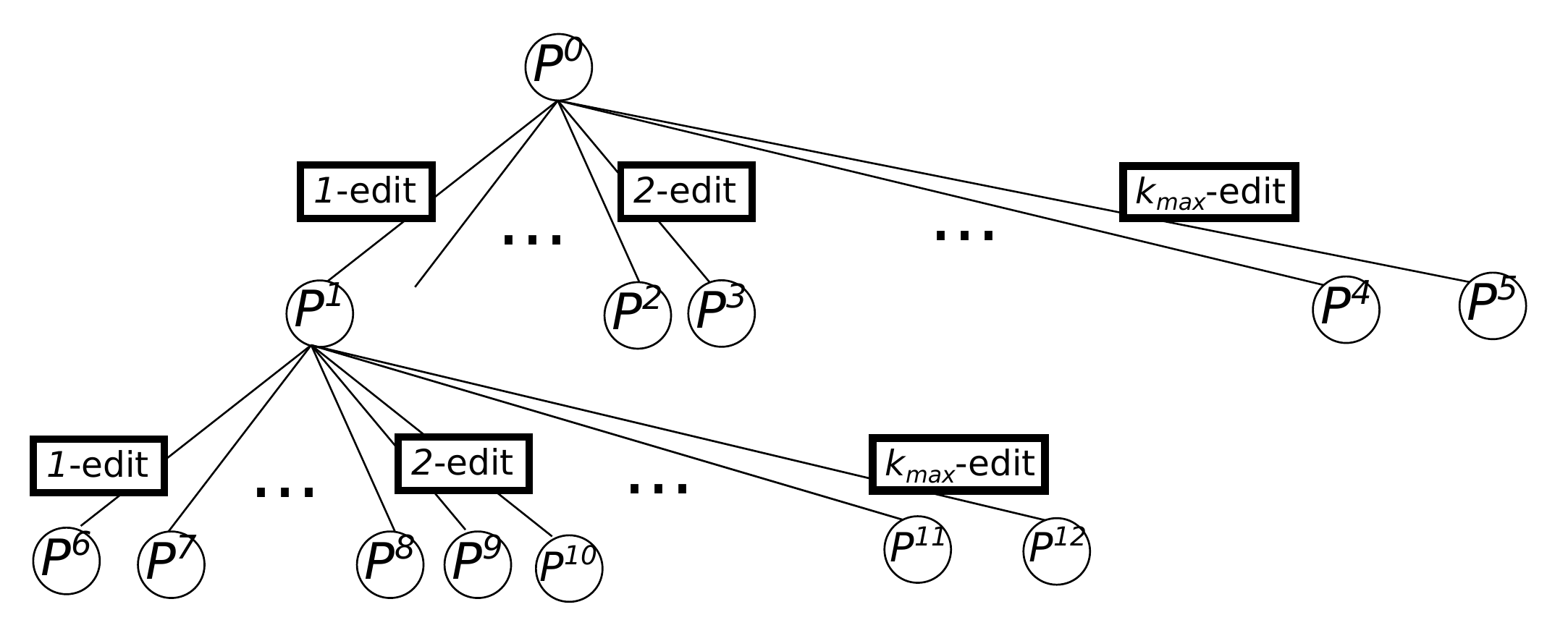}
    \caption{RNS Tree.}
    \label{subfig:RNS-tree}
    \end{subfigure}
    \hfill
    \begin{subfigure}[b]{0.3\linewidth}
    \centering
    \includegraphics[width=1.0\textwidth]{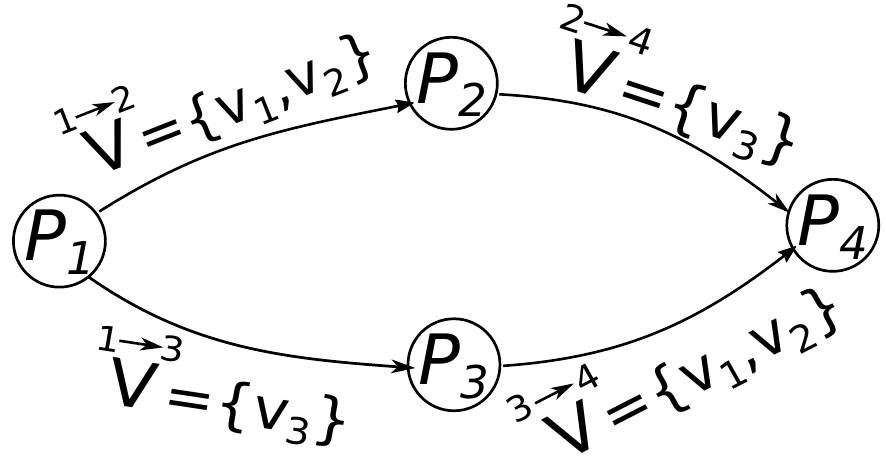}
	\caption{Example of duplication in RNS. Both $P^2$ and $P^3$ generate $P^4$.}
	\label{subfig:RNS-duplication-example}
	\end{subfigure}
	\caption{Illustrations for the Recurrent Neighborhood Search (RNS). In both figures, circles represent membership vectors associated with optimal partitions.}
\end{figure}

Method $RNS(G,P,r_{max})$, which takes as input a signed graph $G$, an optimal partition $P$ and a maximum edit distance $r_{max}$, is detailed in Algorithm~\ref{algo:RNS}. It starts from the initial partition $P$ (line~\ref{algoline:initPartition_in_RNS}), and enumerates a set of its neighbor optimal partitions up to edit distance $r_{max}$ (line~\ref{algoline:k_vertices_in_RNS}), denoted by $RN_{\leq r_{max}}(P)$. 
Each partition $P^t$ associated with $\pi^t \in N_{r}(\pi^s)$ is considered as a potential initial partition for seeking new neighbor partitions (line~\ref{algoline:for_each_Pt_in_RNS}). Despite the pruning techniques which avoid some repetitions, it is possible that $P^t$ has been already discovered and belongs to $RN_{\leq r_{max}}(P)$ (see such an example in Figure~\ref{subfig:RNS-duplication-example}). Line~\ref{algoline:duplicationRemovalRNS} checks this case, and discards $P^t$, if required. This process is repeated in a recursive manner, until there is no new partition obtained. 

\begin{algorithm}[!ht]
\SetAlgoLined
    \SetKwInOut{Input}{Input}
    \SetKwInOut{Output}{Output}
    \Input{signed graph $G(V,E,s)$, partition $P$, maximum edit distance $r_{max}$}
    \Output{$RN_{\leq r_{max}}(P)$}
    $U = \{ P \}$ // a set of unprocessed partitions \label{algoline:initPartition_in_RNS} \;
    $RN_{\leq r_{max}}(P) = \emptyset$ // a set of discovered partitions\;
    \While{$U \neq \emptyset$}{ 
        \For{each $r \in \{1,..,r_{max}\}$}{ \label{algoline:k_vertices_in_RNS}
            Let $P^s$ be an element of $U$, where $\pi^s$ is an associated membership vector of $P^s$ \;
            $RN_{\leq r_{max}}(P) = RN_{\leq r_{max}}(P) \bigcup P^s$ \;
            $N_{r}(\pi^s) = \mathit{CoNS(G, \pi^s, r)}$ // Complete Neighborhood Search \; \label{algoline:CoNS_in_RNS}
            \For(\tcp*[h]{$P^t$ is the associated partition of $\pi^t$}){$\pi^t \in N_{r}(\pi^s)$}{ \label{algoline:for_each_Pt_in_RNS}
                \tcc{memory-based duplication removal}
                \If{$P^t \notin RN_{\leq r_{max}}(P)$}{ \label{algoline:duplicationRemovalRNS}
                    $U = U \bigcup P^t$
                }
            }
        }
    }
\caption{$RNS(G,P,r_{max})$}
\label{algo:RNS}
\end{algorithm}

In the rest of this work, for the sake of convenience, we leave out the common parameter $G$ and $P^s$ (or $P$) when referring to the mentioned methods: $CoNS(r)$ and $RNS(r_{max})$.

\section{Pruning Strategies}
\label{sec:PruningStrategies}
In this section, we present the pruning strategies incorporated in the method $CoNS(G, \pi^s,r)$, described in Section~\ref{sec:RecurrentNeighborhoodSearch}. But before this, we need to introduce some additional definitions. Consider an edit operation $\pi^s \rightarrow \pi^t$ with moving vertices in set $\myOverArrow{V}{s}{t}{}$, where membership vector $\pi^s$ (resp. $\pi^t$) is associated with source partition $P^s$ (resp. target partition $P^t$). Let us define $\mySimpleOverArrow{V}^s_{\pi^s(u)} = \myOverArrow{V}{s}{t}{} \cap M^s_{\pi^s(u)}$ (resp. $\mySimpleOverArrow{V}^t_{\pi^s(u)} = \myOverArrow{V}{s}{t}{} \cap M^t_{\pi^s(u)}$) as the set of moving vertices being in the 
source module of vertex $u$ in $P^s$ (resp. in $P^t$). Also, let $\mySimpleOverArrow{V}^s_{\pi^t(u)} = \myOverArrow{V}{s}{t}{} \cap M^s_{\pi^t(u)}$ (resp. $\mySimpleOverArrow{V}^t_{\pi^t(u)} = \myOverArrow{V}{s}{t}{} \cap M^t_{\pi^t(u)}$) be the set of moving vertices being in the target module of $u$ in $P^s$ (resp. in $P^t$). Finally, we define 
$\myOverArrow{V}{s}{t}{u} = \Bigl( \mySimpleOverArrow{V}^s_{\pi^s(u)} \cup \mySimpleOverArrow{V}^t_{\pi^s(u)} \cup \mySimpleOverArrow{V}^t_{\pi^t(u)} \cup \mySimpleOverArrow{V}^s_{\pi^t(u)} \Bigr) \setminus \{u\}$
as the set of moving vertices connected to vertex $u$ in $\widetilde{G}[\myOverArrow{V}{s}{t}{}]$. 
To illustrate the aforementioned notations related to $\myOverArrow{V}{s}{t}{}$, we consider again the same example from Figure~\ref{fig:IllustrativeExampleEditForNotations}. For $v_1 \in \myOverArrow{V}{s}{t}{}$, we have $\pi^s(v_1) = 1$ and $\pi^t(v_1) = 3$. Then, $\mySimpleOverArrow{V}^s_{\pi^s(v_1)} = \{v_1, v_2\}$ is the set of vertices in $M^s_1$ and $\mySimpleOverArrow{V}^t_{\pi^t(v_1)} = \{v_1\}$ since $v_1$ is the only moving vertex belonging to $M^t_3$. Also, $\mySimpleOverArrow{V}^s_{\pi^t(v_1)} = \mySimpleOverArrow{V}^s_{\{3\}} = \emptyset$ since there is no moving vertex in $M^s_3$, and $\mySimpleOverArrow{V}^t_{\pi^s(v_1)} = \{v_3\}$ since $v_3\in M^t_1$. Finally, we get $\myOverArrow{V}{s}{t}{v_1} = \{v_2, v_3\}$.

\subsection{Non-Minimum Edit Operation Pruning}
\label{subsec:NonMinimumEditOperationPruning}
When generating the set $N_r(\pi^s)$ in $RNS$, only considering min-edit operations is sufficient. Next, we define a property allowing to detect, in some cases, that an edit operation is not minimal. The idea behind it is to define two sufficient conditions on the set of moving vertices implying that there exist an operation of smaller cost leading to the same partition. In the following, let $\pi^s \rightarrow \pi^t$ be an edit operation with moving vertices in $\myOverArrow{V}{s}{t}{}$.

\begin{property}[Non-min-edit operation]
	Let $\mySimpleOverArrow{V_a} \subseteq \myOverArrow{V}{s}{t}{}$ be a subset of $\myOverArrow{V}{s}{t}{}$. Assume that the vertices in $\mySimpleOverArrow{V_a}$ constitute the majority part of module $M^s_{a}$ (i.e. $|\mySimpleOverArrow{V_a}| > |M^s_a \setminus \mySimpleOverArrow{V_a}|$) and that they are in the same target module $M^t_b$ in $P^t$ (i.e. $|\pi^t(\mySimpleOverArrow{V_a})| = 1$). Moreover, suppose there exists a subset of moving vertices $\mySimpleOverArrow{V_b}\subseteq \myOverArrow{V}{s}{t}{}$, such that $\emptyset \subseteq \mySimpleOverArrow{V_b} \subseteq M^s_b$.
	Let us define $\overline{V}_a$ (resp. $\overline{V}_b$) as $M^s_a \setminus \myOverArrow{V}{s}{t}{}$ (resp. $M^s_b \setminus \myOverArrow{V}{s}{t}{}$) in $P^s$.
	Each condition in the following is a sufficient condition for $\pi^s \rightarrow \pi^t$ to be a non-min-edit operation:
	\item[\textit{(a)}] If $\mySimpleOverArrow{V_b}$ does not move into $M^s_a$ 
	and $|\mySimpleOverArrow{V_a}| > |\overline{V}_a| + |\overline{V}_b|$, then $\pi^s \rightarrow \pi^t$ is not a min-edit operation.
	\item[\textit{(b)}] If $\mySimpleOverArrow{V_b}$ moves into $M^s_a$
	and $|\mySimpleOverArrow{V_a}| + |\mySimpleOverArrow{V_b}| > |\overline{V}_a| + |\overline{V}_b|$, then $\pi^s \rightarrow \pi^t$ is not a min-edit operation.
	\label{prop:Non-min-edit-operation}
\end{property}
\begin{proof}
Let us first handle the condition \textit{(\ref{prop:Non-min-edit-operation}a)}, which is illustrated in Figure~\ref{fig:IllustrativeExampleMinEditOperation}. 
When this condition is satisfied, keeping $\mySimpleOverArrow{V_a}$ in $M^s_a$ and rather moving $\overline{V}_a$ (resp. $\overline{V}_b$) into $ M^s_b$ (resp. $ M^s_a$) results in an $r'$-edit operation with $r'<r$.
In the case where condition \textit{(\ref{prop:Non-min-edit-operation}b)} is satisfied, keeping $\mySimpleOverArrow{V_b}$ in $M^s_b$ while moving $\overline{V}_a$ and $\overline{V}_b$ results in an $r'$-edit operation with $r' < r$. Let us take the following example to illustrate Property~\ref{prop:Non-min-edit-operation}.b. Two modules exchange all of their vertices with each other, such that $\mySimpleOverArrow{V_a} =  M^s_a$, $\overline{V}_a = \emptyset$,  $\mySimpleOverArrow{V_b} =  M^s_b$ and $\overline{V}_b = \emptyset$. In the end, we obtain the same two modules. Since the inequality $|\mySimpleOverArrow{V_a}| + |\mySimpleOverArrow{V_b}| > |\overline{V}_a| + |\overline{V}_b|$ holds in this particular scenario, this is not a min-edit operation.
\end{proof}

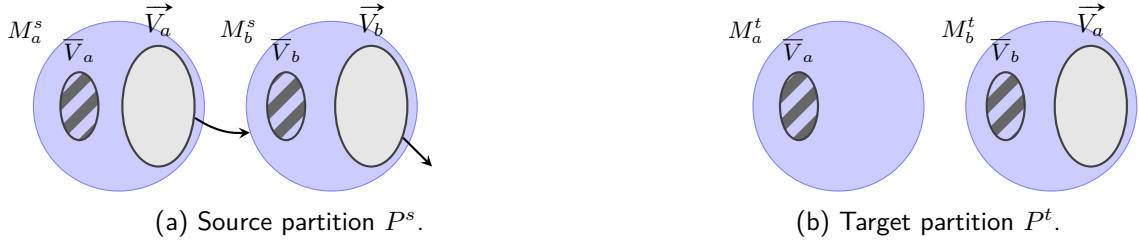
\begin{figure}
    \centering
    \begin{subfigure}[h]{0.45\linewidth}
    \begin{tikzpicture}[->,>=stealth,scale=0.80]
        \usepgflibrary{arrows}
        \usepgflibrary{shapes}
        \usepgflibrary{patterns}
        
        \pgfdeclarepatternformonly{stripes}
        {\pgfpointorigin}{\pgfpoint{0.4cm}{0.4cm}}
        {\pgfpoint{0.4cm}{0.4cm}}
        {
            \pgfpathmoveto{\pgfpoint{0cm}{0cm}}
            \pgfpathlineto{\pgfpoint{0.4cm}{0.4cm}}
            \pgfpathlineto{\pgfpoint{0.4cm}{0.2cm}}
            \pgfpathlineto{\pgfpoint{0.2cm}{0cm}}
            \pgfpathclose%
            \pgfusepath{fill}
            \pgfpathmoveto{\pgfpoint{0cm}{0.2cm}}
            \pgfpathlineto{\pgfpoint{0cm}{0.4cm}}
            \pgfpathlineto{\pgfpoint{0.2cm}{0.4cm}}
            \pgfpathclose%
            \pgfusepath{fill}
        }

        \footnotesize
		\tikzstyle{mn}=[ellipse,thick,draw=black!75,fill=black!10,minimum width=0.95cm,minimum height=1.6cm,align=center]
		\tikzstyle{rn}=[ellipse,thick,draw=black!75,fill=black!10,minimum width=0.5cm,minimum height=0.9cm,align=center,pattern=stripes, pattern color=black!60]
		\tikzstyle{clu}=[circle,draw=blue!50,fill=blue!20,minimum size=2.25cm,inner sep=0pt]
		\tikzstyle{pl}=[thick,draw=green!55!black!75,text=green!55!black!75]
		\tikzstyle{bl}=[thick,draw=black,text=black]
		
        \node (cluster1) at (-2.5,7.75) [clu, label=140:$M_{a}^{s}$] {};
        \node (cluster2) at (1,7.75) [clu, label=140:$M_{b}^{s}$] {};
                
		\node[mn, label=90:$\mySimpleOverArrow{V_a}$] (v1) at (-1.85,7.75) {};
		\node[rn, label=90:$\overline{V}_a$] (v2) at (-3.15,7.75) {};
		\node[mn, label=90:$\mySimpleOverArrow{V_b}$] (v3) at (1.65,7.75) {};
		\node[rn, label=90:$\overline{V}_b$] (v4) at (0.25,7.75) {};
		
        \draw [bl] (v1) to [bend right=18] (cluster2);
        \draw [bl] (v3) --++(1cm,-1cm) node[above,midway]{};
	\end{tikzpicture}
	\caption{Source partition $P^s$.}
	\end{subfigure}
	\hspace{0.5cm}
    \begin{subfigure}[h]{0.45\linewidth}
    \centering
    \begin{tikzpicture}[->,>=stealth,scale=0.80]
        \usepgflibrary{arrows}
        \usepgflibrary{shapes}
        \usepgflibrary{patterns}

        \footnotesize
		\tikzstyle{mn}=[ellipse,thick,draw=black!75,fill=black!10,minimum width=0.95cm,minimum height=1.6cm,align=center]
		\tikzstyle{rn}=[ellipse,thick,draw=black!75,fill=black!10,minimum width=0.5cm,minimum height=0.9cm,align=center,pattern=stripes, pattern color=black!60]
		\tikzstyle{clu}=[circle,draw=blue!50,fill=blue!20,minimum size=2.25cm,inner sep=0pt]
		\tikzstyle{pl}=[thick,draw=green!55!black!75,text=green!55!black!75]
		\tikzstyle{bl}=[thick,draw=black,text=black]
		
        \node (cluster1) at (-2.5,7.75) [clu, label=140:$M_{a}^{t}$] {};
        \node (cluster2) at (1,7.75) [clu, label=140:$M_{b}^{t}$] {};
                
		\node[mn, label=90:$\mySimpleOverArrow{V_a}$] (v1) at (1.65,7.75) {};
		\node[rn, label=90:$\overline{V}_a$] (v2) at (-3.15,7.75) {};
		\node[rn, label=90:$\overline{V}_b$] (v4) at (0.25,7.75) {};
		
	\end{tikzpicture}
	\caption{Target partition $P^t$.}
	\end{subfigure}
	\caption{Illustrative example for Property~\ref{prop:Non-min-edit-operation}a, where vertices in $\protect\mySimpleOverArrow{V_a}$ (resp. $\protect\mySimpleOverArrow{V_b}$) move from their source module $M^{s}_{a}$ (resp. $M^{s}_{b}$)  (Fig.~\ref{fig:IllustrativeExampleMinEditOperation}a) into target module $M^{s}_{b}$ (resp. $M^{s}_{c}$, which is purposely not shown) (Fig.~\ref{fig:IllustrativeExampleMinEditOperation}b). $\overline{V}_a$ and $\overline{V}_b$ represent the non-moving vertices.}
	\label{fig:IllustrativeExampleMinEditOperation}
\end{figure}


\subsection{Decomposable Edit Operation}
\label{subsec:Decomposable_MED}
The next set of properties is related to the decomposability of an $r$-edit operation $\pi^s \rightarrow \pi^t$. We assume that $I(P^s) = I(P^t)$, which is in line with our objective of enumerating all optimal solutions of a given graph $G$.

\begin{definition}[Decomposable edit operation]
    We consider an $r$-edit operation $\pi^s \rightarrow \pi^t$ as \textit{decomposable} if there is an intermediate membership vector $\pi^{t'}$, associated with a partition $P^{t'}$, between $\pi^s$ and $\pi^t$ such that:
    \begin{enumerate}
        \item $I(P^s) = I(P^t) = I(P^{t'})$;
        \item $\myOverArrow{V}{s}{i}{} \subset \myOverArrow{V}{s}{t}{}$; and
        \item $\pi^{t'}(u)=\pi^t(u)$ for each $u \in \myOverArrow{V}{s}{i}{}$.
    \end{enumerate}
    \label{def:decomposableEdit}
\end{definition}

\begin{property}[Atomic edit operation]
	We consider that an $r$-edit operation $\pi^s \rightarrow \pi^t$ is atomic if the condition $I(P^s) < I(P^{t'})$ is satisfied for any $r'$-edit operation $P^s \rightarrow P^{t'}$ with $r' < r$ satisfying (2) and (3) in Definition~\ref{def:decomposableEdit}.
    \label{prop:atomic}
\end{property}
\begin{proof}
We assume that there is an $r'$-edit operation such that $I(P^s) = I(P^{t'})$. Then, we can construct an $r''$-edit operation $\pi^{t'} \rightarrow \pi^t$ satisfying (1) in Definition~\ref{def:decomposableEdit} and $\pi^{t'}(u)=\pi^t(u)$ for each $u \in \myOverArrow{V}{s}{t}{} \setminus \myOverArrow{V}{s}{i}{}$.
\end{proof}

\begin{figure}
    \centering
    \begin{subfigure}[h]{0.8\linewidth}
        \centering
        \begin{tikzpicture}[scale=0.80]
        \footnotesize
        
		\tikzstyle{nnb}=[circle,thick,draw=black!75,fill=black!10,minimum size=3mm]
		\tikzstyle{nna}=[circle,dashed,draw=red,fill=red!10,minimum size=3mm]
		\tikzstyle{pl}=[thick,draw=green!55!black!75,text=green!55!black!75]
		\tikzstyle{nl}=[thick,draw=red,text=red]
		\tikzstyle{clu}=[circle,draw=blue!50,fill=blue!20,minimum size=2.25cm,inner sep=0pt]

        \node (cluster1) at (-3.25,7.75) [clu, label=140:$M_{1}$] {};
        \node (cluster2) at (0.25,7.75) [clu, label=140:$M_{2}$] {};
        \node (cluster3) at (3.75,7.75) [clu, label=140:$M_{3}$] {};
        \node (cluster4) at (7.25,7.75) [clu, label=140:$M_{4}$] {};
        
		\node[nnb] 
		    (v1) at (-3.25,7) {$v_{1}$};
		
		\node[nnb]
		    (v2) at (-3.25,8.5) {$v_{2}$};
		
		\node[nnb]
		    (v3) at (3.75,7) {$v_{3}$};
		
		\node[nnb]
		    (v4) at (7.25,7.75) {$v_{4}$};

		\node[nna] (v5) at (0.75,8.50) {$v_{1}$};
		\node[nna] (v6) at (-0.25,8.50) {$v_{2}$};
		\node[nna] (v7) at (0.25,7) {$v_{3}$};
		\node[nna] (v8) at (3.75,8.5) {$v_{4}$};
		
		\draw[pl] (v1)  -- (v2)  node [midway, below] {};
		\draw[nl] (v3)  -- (v4)  node [midway, below] {};
        \draw [nl] (v1) to [bend right=25] (v4);
        \draw [nl] (v2) to [bend right=8] (v4);
    \end{tikzpicture}
    \caption{An illustration of an $r$-edit operation $\pi^s \rightarrow \pi^t$ with  $\myOverArrow{V}{s}{t}{} = \{v_1, v_2, v_3, v_4\}$. The membership vector $\pi^s$ associated with partition $P^s$ is defined by $\pi^s = (1,1,3,4)$, whereas $\pi^t = (2,2,2,3)$. Positive (resp. negative) edges between moving vertices are drawn in green (resp. red).}
	\label{subfig:decomposable-k-edit-operation}
    \end{subfigure}
    %
    %
    \begin{subfigure}[h]{0.5\linewidth}
    \vspace{0.25cm}
    \centering
        \begin{tikzpicture}[scale=0.80]
            \footnotesize
    		\tikzstyle{nn}=[circle,thick,draw=black!75,fill=black!10,minimum size=3mm]
    		\tikzstyle{bl}=[thick,draw=black,text=black]
    		
    		\node[nn] (v1) at (0,7.75) {$v_{1}$};
    		\node[nn] (v2) at (1.5,7.75) {$v_{2}$};
    		\node[nn] (v3) at (3,7.75) {$v_{3}$};
    		\node[nn] (v4) at (4.5,7.75) {$v_{4}$};

            \draw[bl] (v1)  -- (v2);
            \draw[bl] (v3)  -- (v4);
    	\end{tikzpicture}
    \caption{The corresponding interaction graph $\widetilde{G}[\myOverArrow{V}{s}{t}{}]$ of the $r$-edit operation illustrated in Fig.~\ref{fig:IllustrativeExampleDecomposableEditOperation}a}
	\label{subfig:interaction-graph}
    \end{subfigure}
	\caption{Illustrative example regarding the definition of an interaction subgraph, here $\widetilde{G}[\myOverArrow{V}{s}{t}{}]$.}
	\label{fig:IllustrativeExampleDecomposableEditOperation}
\end{figure}
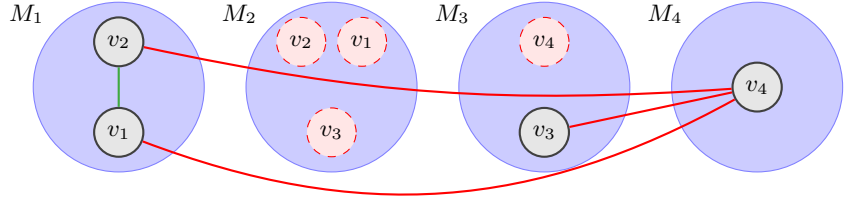
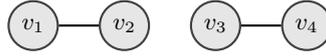

Atomicity requires three types of connectivity conditions. We state the first one in the following property.

\begin{property}[Edge connectivity]
	In an atomic edit operation $\pi^s \rightarrow \pi^t$,  
	$G[\myOverArrow{V}{s}{t}{}]$ is a single connected component. 
	\label{prop:EdgeConnectivity}
\end{property}
\begin{proof}
Let $\myOverArrow{V}{s}{i}{}$ and $\myOverArrow{V}{i}{t}{}$ be any two proper subsets of $\myOverArrow{V}{s}{t}{}$, such that $\myOverArrow{V}{s}{i}{} \cap \myOverArrow{V}{i}{t}{} = \emptyset$ and $E(\myOverArrow{V}{s}{i}{},\myOverArrow{V}{i}{t}{}) = \emptyset$. 
Then, we can construct with $\myOverArrow{V}{s}{i}{}$ an $r'$-edit operation satisfying Definition~\ref{def:decomposableEdit}.
\end{proof}

Nevertheless, $G[\myOverArrow{V}{s}{t}{}]$ being connected is not a sufficient condition. The next property states that the signs of edges between moving vertices play a key role with respect to atomicity. In the following, we assume that $G[\myOverArrow{V}{s}{t}{}]$ is connected.


\begin{property}[Spurious edge connectivity]
    Each condition in the following is a sufficient condition for $\pi^s \rightarrow \pi^t$ to be a decomposable edit operation:
	\vspace{-0.2cm}
	\begin{itemize}
	    \item[\textit{(a)}] 
	    Let $S \subseteq \myOverArrow{V}{s}{t}{}$ be a subset such that $|\pi^s(S)|=1$ and $E(S, \mySimpleOverArrow{V}^s_{\pi^s(S)}) \neq \emptyset$. Let $G' = (\myOverArrow{V}{s}{t}{}, E', w)$, where $E' = E \setminus E(S, \mySimpleOverArrow{V}^s_{\pi^s(S)})$ and $w(e)=1$ for each $e \in E'$.
	    If  $\Omega(S, \mySimpleOverArrow{V}^s_{\pi^s(S)}) = 0$ and 
	    $G'$ is not connected, then $\pi^s \rightarrow \pi^t$ is decomposable.
	    \item[\textit{(b)}] 
	    Let $S \subseteq \myOverArrow{V}{s}{t}{}$ be a subset such that $|\pi^t(S)|=1$ and $E(S, \mySimpleOverArrow{V}^t_{\pi^t(S)}) \neq \emptyset$. Let $G'=(\myOverArrow{V}{s}{t}{}, E', w)$, where $E' = E \setminus E(S, \mySimpleOverArrow{V}^t_{\pi^t(S)})$ and $w(e)=1$ for each $e \in E'$.
	    If $\Omega(S, \mySimpleOverArrow{V}^t_{\pi^t(S)}) = 0$ and 
	    $G'$ is not connected, then $\pi^s \rightarrow \pi^t$ is decomposable.
	\end{itemize}
	\label{prop:SpuriousEdgeConnectivity}
\end{property}
\begin{proof}
Straightforwardly, we can imagine $S$ as a contracted vertex $v$ in both cases, such that the vertices in $S$ and the edges between them are replaced by a single vertex $v$. Since $\Omega(S, \mySimpleOverArrow{V}^s_{\pi^s(S)}) = 0$ ($\Omega(S, \mySimpleOverArrow{V}^t_{\pi^t(S)}) = 0$), the edges between $v$ and $\mySimpleOverArrow{V}^s_{\pi^s(v)}$ (resp. $\mySimpleOverArrow{V}^t_{\pi^t(v)}$) do not play a role (i.e., as if they did not exist). Therefore, the proof is the same as for Property~\ref{prop:EdgeConnectivity}, with subsets $S$ and $\myOverArrow{V}{s}{t}{}\setminus S$. 
\end{proof}

The last connectivity condition 
allows to verify if a moving vertex depends on the simultaneous movement of a subset of other vertices in $\myOverArrow{V}{s}{t}{}$.
It can be checked through the concept of \textit{interaction subgraph} $\widetilde{G}[\myOverArrow{V}{s}{t}{}]$. We define the interaction subgraph defined for the edit operation $\pi^s \rightarrow \pi^t$ as  $\widetilde{G}[\myOverArrow{V}{s}{t}{}] = (\myOverArrow{V}{s}{t}{}, \widetilde{E})$, where $\widetilde{E} = \{(u,v) \ |\ u,v \in \myOverArrow{V}{s}{t}{} \text{and}\ (u,v) \in E\ \ \text{and}\ (\pi^s(u) = \pi^s(v) \lor \pi^t(u) = \pi^s(v) \lor \pi^s(u) = \pi^t(v) \lor \pi^t(u) = \pi^t(v))\}$. Its extraction from its original graph $G$ is illustrated in Figure~\ref{fig:IllustrativeExampleDecomposableEditOperation}.


\begin{property}[Interaction connectivity]
	Consider an atomic edit operation $\pi^s \rightarrow \pi^t$. 
	The interaction subgraph $\widetilde{G}[\myOverArrow{V}{s}{t}{}]$ is a single connected component. 
	\label{prop:InteractionConnectivity}
\end{property}
\begin{proof}
Assume that $\widetilde{G}[\myOverArrow{V}{s}{t}{}]$ is not connected (e.g. Figure~\ref{fig:IllustrativeExampleDecomposableEditOperation}). This implies that there is a subset $\myOverArrow{V}{s}{i}{} \subseteq \myOverArrow{V}{s}{t}{}$, such that  $\widetilde{E}(\myOverArrow{V}{s}{i}{}, \myOverArrow{V}{s}{t}{} \setminus \myOverArrow{V}{s}{i}{}) = \emptyset$. Let $\pi^s \rightarrow \pi^{t'}$ be an edit operation with moving vertex set $\myOverArrow{V}{s}{i}{}$, which satisfies (2)-(3) in Definition~\ref{def:decomposableEdit}. When moving a vertex $u \in \myOverArrow{V}{s}{i}{}$ from $M^s_{\pi^s(u)}$ to $M^{s}_{\pi^{t'}(u)}$
the contribution of vertex $u$ to the imbalance is

\begin{eqnarray}
    \frac{\Omega(\{u\}, \mySimpleOverArrow{V}^s_{\pi^s(u)}) - \Omega(\{u\}, \mySimpleOverArrow{V}^s_{\pi^{t'}(u)}) + \Omega(\{u\}, \mySimpleOverArrow{V}^{t'}_{\pi^{t'}(u)}) - \Omega(\{u\}, \mySimpleOverArrow{V}^{t'}_{\pi^s(u)})}{2}.
    \label{eq:Imbalance_contributionVertexU}
\end{eqnarray}

From the definition of $\widetilde{E}$ and from $\widetilde{E}(\myOverArrow{V}{s}{i}{}, \myOverArrow{V}{s}{t}{} \setminus \myOverArrow{V}{s}{i}{}) = \emptyset$, there is no vertex $v \in \myOverArrow{V}{s}{t}{}\setminus \myOverArrow{V}{s}{i}{}$ which contributes to the contribution of vertex $u$, as shown in Equation~\eqref{eq:Imbalance_contributionVertexU}. Therefore, Equation (1) in Definition~\ref{def:decomposableEdit} is also satisfied for $\pi^s \rightarrow \pi^{t'}$. 
\end{proof}

For the sake of simplicity, in Algorithm~\ref{algo:internalPruning}, we define the function \textit{intAtomic}($G$, $\pi^s$, $\myOverArrow{V}{s}{t}{}$) which is used to indicate whether $\pi^s \rightarrow \pi^t$ satisfies Properties~\ref{prop:EdgeConnectivity} and~\ref{prop:SpuriousEdgeConnectivity}a. Notice that we do so when the target modules of $\myOverArrow{V}{s}{t}{}$ are not known. Likewise, in Algorithm~\ref{algo:externalPruning}, we define the function \textit{extAtomic}($G$, $\pi^s$, $\pi^t$, $\myOverArrow{V}{s}{t}{}$),
which indicates whether $\pi^s \rightarrow \pi^t$ satisfies Properties~\ref{prop:Non-min-edit-operation}, \ref{prop:atomic}, \ref{prop:SpuriousEdgeConnectivity}b and~\ref{prop:InteractionConnectivity}. Notice that we do so when the target modules of $\myOverArrow{V}{s}{t}{}$ are partially or completely known, i.e., in lines~\ref{algoline:CoNS-externalPruning1}, \ref{algoline:CoNS-externalPruning2} and~\ref{algoline:CoNS-optimalityCond2} of Algorithm~\ref{algo:CoNS}.

\subsection{Multiple Vertex Moves between Optima (MVMO) Property}
\label{subsec:MVMO-property}

 
In this section, we introduce our last family of properties, which rely on the assumption that both partitions associated to an edit operation are optimal.

\begin{property}[MVMO on weighted signed networks]
Consider an atomic edit operation $\pi^s \rightarrow \pi^t$ with $r>1$, where $P^s$ and $P^t$ are optimal.
The following property holds for each $u \in \myOverArrow{V}{s}{t}{}$:
    \begin{equation}
        \gamma^{left}_{u} > \gamma^{right}_{u},
    \end{equation}
    where
    \begin{align}
        \gamma^{left}_{u} &= \Omega(\{u\}, \mySimpleOverArrow{V}^s_{\pi^s(u)}) - \Omega(\{u\}, \mySimpleOverArrow{V}^s_{\pi^t(u)}),\\
        \gamma^{right}_{u} &= -\Omega(\{u\}, \mySimpleOverArrow{V}^t_{\pi^t(u)}) + \Omega(\{u\}, \mySimpleOverArrow{V}^t_{\pi^s(u)}).\\
    \end{align}
	\label{prop:MVMO}
\end{property}
\begin{proof}
Let $\pi^s \rightarrow \pi^t$ be an atomic min-$r$-edit operation applied on $P_s$ to obtain $P_t$. First, we recall a simple condition satisfied by any optimal partition for the CC problem~\cite{Charikar2017} when moving a single vertex, i.e. when $r=1$. Given an optimal partition $P_s$, the placement of any vertex $u$ in its module $M^s_{\pi^s(u)}$ is also optimal with respect to any other modules. This means that any vertex $u$ maximizes (resp. minimizes) its number of positive (resp. negative) edges in module $M^s_{\pi^s(u)}$ of $P_s$, so that moving $u$ to any other module does not improve the objective function value. We can write this condition as 
\begin{eqnarray}
\Omega(\{u\}, M^s_{\pi^s(u)}) \geq \Omega(\{u\}, M^s_{\pi^t(u)}).
\label{eq:VertexOptimalityCondition}
\end{eqnarray}
Now, when we move more than one vertex in $\pi^s \rightarrow \pi^t$, i.e. if $r>1$, this condition becomes 
\begin{eqnarray}
\Omega(\{u\}, M^s_{\pi^s(u)}) > \Omega(\{u\}, M^s_{\pi^t(u)}),
\label{eq:VertexOptimalityCondition2}
\end{eqnarray}
for each $u \in \myOverArrow{V}{s}{t}{}$, due to the atomicity assumption. 
Next, in order to take into account the existence of other moving vertices in the source and target modules of vertex $u$, Equation~\eqref{eq:VertexOptimalityCondition2} can be rewritten as $\gamma^{left}_{u} > \myOverArrow{\Delta}{s}{t}{u}$, where 
$\myOverArrow{\Delta}{s}{t}{u} = \Omega(\{u\}, M^s_{\pi^t(u)} \setminus \mySimpleOverArrow{V}^s_{\pi^t(u)}) - \Omega(\{u\}, M^s_{\pi^s(u)} \setminus \mySimpleOverArrow{V}^s_{\pi^s(u)})$. Similarly, considering the atomic min-$r$-edit operation $P^t \rightarrow P^s$, we have also $-\myOverArrow{\Delta}{t}{s}{u} > \gamma^{right}_{u}$
, where $\myOverArrow{\Delta}{t}{s}{u} = \Omega(\{u\}, M^t_{\pi^s(u)} \setminus \mySimpleOverArrow{V}^t_{\pi^s(u)}) - \Omega(\{u\}, M^t_{\pi^t(u)} \setminus \mySimpleOverArrow{V}^t_{\pi^t(u)})$. Since both $\myOverArrow{\Delta}{s}{t}{u}$ and $\myOverArrow{\Delta}{t}{s}{u}$ are defined on relations of vertex $u$ with non-moving vertices in $P^s$ and $P^t$, we have $\myOverArrow{\Delta}{s}{t}{u}=-\myOverArrow{\Delta}{t}{s}{u}$. Finally, by rewriting the previous equations we obtain $\gamma^{left}_{u} > \myOverArrow{\Delta}{s}{t}{u} = -\myOverArrow{\Delta}{t}{s}{u} > \gamma^{right}_{u}$.
\end{proof}

\begin{corollary}[MVMO on unweighted signed networks]
    The inequality $\gamma^{left}_{u} - \gamma^{right}_{u} \geq 2$ holds for each vertex $u$, on top of Property~\ref{prop:MVMO}.
    \label{corollary:MVMO-Integer}
\end{corollary}
\begin{proof}
The proof is straightforward from the proof of Property~\ref{prop:MVMO}.
\end{proof}

\begin{figure}[!h]
    \centering
    \begin{tikzpicture}[scale=0.80]
        \footnotesize
        
		\tikzstyle{nnb}=[circle,thick,draw=black!75,fill=black!10,minimum size=3mm]
		\tikzstyle{nna}=[circle,dashed,draw=red,fill=red!10,minimum size=3mm]
		\tikzstyle{pl}=[thick,draw=green!55!black!75,text=green!55!black!75]
		\tikzstyle{nl}=[thick,draw=red,text=red]
		\tikzstyle{clu}=[circle,draw=blue!50,fill=blue!20,minimum size=3.25cm,inner sep=0pt]
		
        \node (cluster1) at (-3.5,7.75) [clu, label=140:$M_{1}$] {};
        \node (cluster2) at (1.5,7.75) [clu, label=140:$M_{2}$] {};
        \node (cluster3) at (-1,3.25) [clu, label=300:$M_{3}$] {};
        
		\node[nnb]
		    (v1) at (-4.75,7) {$v_{1}$};
		\coordinate (d1UpLeft) at ($ (v1) + (0.75,1.5) $);
		\coordinate (d1UpRight) at ($ (v1) + (0.25,1.5) $);

		\node[nnb]
		    (v2) at (-2.25,7) {$v_{2}$};
		\coordinate (d2UpLeft) at ($ (v2) + (-0.75,1.5) $);
		\coordinate (d2UpRight) at ($ (v2) + (-0.25,1.5) $);
		\node[nnb]
		    (v3) at (1.5,7) {$v_{3}$};
		\coordinate (d3UpLeft) at ($ (v3) + (-0.25,1.5) $);
		\coordinate (d3UpRight) at ($ (v3) + (0.25,1.5) $);
        \node[nnb]
        (v4) at (0.25,3.75) {$v_{4}$};
		\coordinate (d4DownLeft) at ($ (v4) + (-1.15,-1.5) $);
		\coordinate (d4DownRight) at ($ (v4) + (-0.75,-1.5) $);
        \node[nnb]
        (v5) at (-2,3.75) {$v_{5}$};
		\coordinate (d5DownLeft) at ($ (v5) + (0.75,-1.5) $);
		\coordinate (d5DownRight) at ($ (v5) + (0.25,-1.5) $);
 
		\node[nna] (v6) at (2.75,8) {$v_{1}$};
		\node[nna] (v7) at (1.5,9) {$v_{2}$};
		\node[nna] (v8) at (0.25,8) {$v_{4}$};
		\node[nna] (v9) at (-0.85,2.75) {$v_{3}$};
		\node[nna] (v10) at (-3.75,8.25) {$v_{5}$};

        \draw[pl] (v1)  -- (v2)  node [midway, below] {$a_{12}$};
        \draw [nl] (v1) to [bend left=18] (v3);
        \draw[nl] (v2)  -- (v3)  node [midway, below] {};
        
        \draw[nl] (v1)  -- (v4)  node [midway, below] {};
        \draw[nl] (v1)  -- (v5)  node [midway, left] {$a_{15}$};
        \draw[pl] (v2)  -- (v4)  node [midway, below] {};
        \draw[nl] (v2)  -- (v5)  node [midway, below] {};
        
        \draw[nl] (v3)  -- (v4)  node [midway, right] {$a_{34}$};
        \draw[nl] (v3)  -- (v5)  node [midway, below] {};
        \draw[pl] (v4)  -- (v5)  node [midway, below] {};
	\end{tikzpicture}
	\caption{Illustrative example, where five vertices $v_1$, $v_2$, $v_3$, $v_4$ and $v_5$ move from their source modules to their target ones. Dashed circles indicate the moving vertices, when they are in their target modules after the edit operation. Note that we do not show non-moving vertices for the sake of simplicity. Positive (resp. negative) edges between moving vertices are drawn in green (resp. red).}
	\label{fig:IllustrativeExampleEdit}
\end{figure}
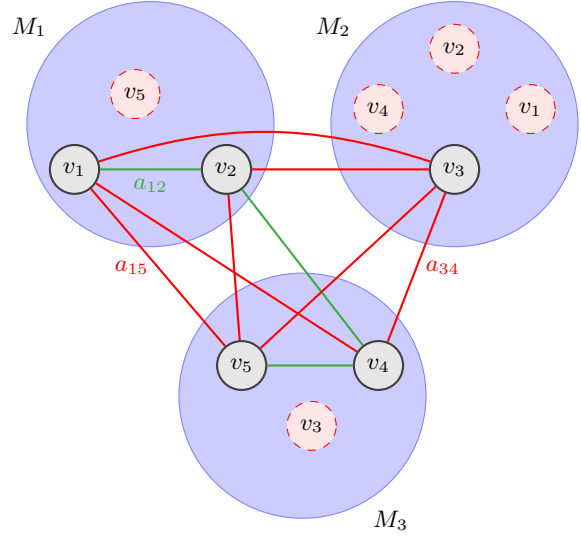

Figure~\ref{fig:IllustrativeExampleEdit} can be used to illustrate Property~\ref{prop:MVMO} and Corollary~\ref{corollary:MVMO-Integer}. For instance, for vertex $v_1$, the equation in Property~\ref{prop:MVMO} becomes 
\begin{equation}
    \label{eq:generalPropertyWithoutDeltaExample}
    \centering
    a_{12} - a_{13} > a_{15} - a_{12} - a_{14}.
\end{equation}
Similarly, for vertex $v_3$, the equation in Property~\ref{prop:MVMO} becomes 
\begin{equation}
    \label{eq:generalPropertyWithoutDeltaExample2}
    \centering
    - a_{34} - a_{35} > a_{13} + a_{23} + a_{34}.
\end{equation}

\subsection{Tractable Cases of the MVMO Property}
\label{subsec:Implications-MVMO-property}

Notice that the MVMO property requires $\pi^s$ and $\pi^t$ to be known, which makes this property not very useful for computational purposes (in Algorithm~\ref{algo:CoNS}).
In this section, we analyze the MVMO property when only partial information of $\pi^t$ is known (lines~\ref{algoline:CoNS-externalPruning1} and \ref{algoline:CoNS-externalPruning2} 
in Algorithm~\ref{algo:CoNS}).

We start by analysing some special relational patterns for $2$-edit and $3$-edit operations. Lemma~\ref{lemma:MVMO-2Vertices} focuses on 2-edit operations. Recall that $\widetilde{E} = \{(u,v) \in E \ |\ u,v \in \myOverArrow{V}{s}{t}{}\ \land\ (\pi^s(u) = \pi^s(v) \lor \pi^t(u) = \pi^s(v) \lor \pi^s(u) = \pi^t(v) \lor \pi^t(u) = \pi^t(v))\}$.
 
\begin{figure*}[!ht]
    \centering
    \includegraphics[width=1\textwidth]{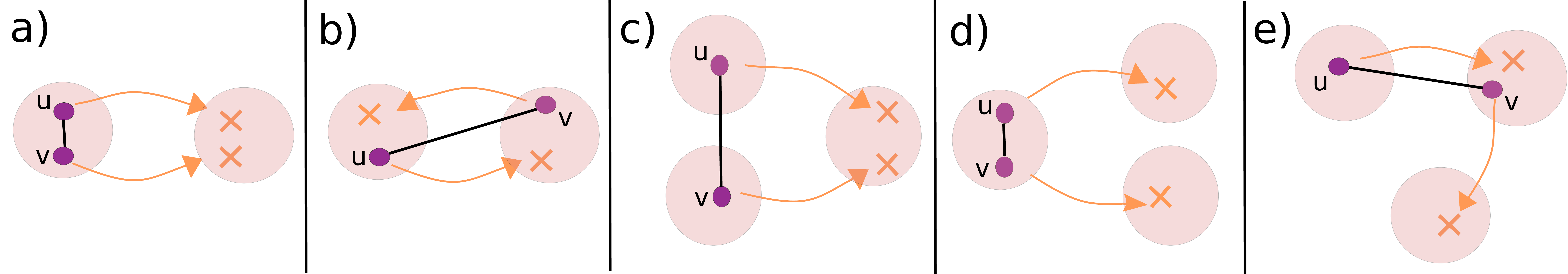}
    \caption[Fig]{All 2-edit operation scenarios, where $(u,v) \in \widetilde{E}$. Note that only scenarios (\ref{fig:2edit-operations}a) and (\ref{fig:2edit-operations}b) are atomic $2$-edit operations, with $a_{uv}>0$ and $a_{uv}<0$, respectively.}
    \label{fig:2edit-operations}
\end{figure*}

\begin{lemma}[MVMO $2$-edit]
	Consider an atomic $2$-edit operation $\pi^s \rightarrow \pi^t$, where $\myOverArrow{V}{s}{t}{} = \{u, v\}$. Since it is an atomic edit operation, we have $(u,v) \in \widetilde{E}$. Then
	\begin{enumerate}
	    \item If $\left( \pi^s(u) = \pi^s(v)) \lor (\pi^t(u) = \pi^t(v) \right)$, then $a_{uv} > 0$. 
	    \item 
	    If $\left( \pi^s(u) = \pi^t(v)) \lor (\pi^t(u) = \pi^s(v) \right)$, then $a_{uv} < 0$.
	\end{enumerate}
	\label{lemma:MVMO-2Vertices}
\end{lemma}
\begin{proof}
From Property~\ref{prop:EdgeConnectivity}, we have $(u,v) \in E$.
In this proof, we use an exhausting strategy:
moving vertices $u$ and $v$, with $(u,v) \in \widetilde{E}$, can be in one of the five scenarios presented in Figure~\ref{fig:2edit-operations}. 
Since Corollary~\ref{corollary:MVMO-Integer} holds for each vertex $u \in \myOverArrow{V}{s}{t}{}$, only scenarios (\ref{fig:2edit-operations}a) and (\ref{fig:2edit-operations}b) are atomic. 
Therefore, $a_{uv}>0$ for scenario (a) whereas $a_{uv}<0$ for scenario (b): 
\vspace{-0.2cm}
\begin{itemize}
    \item[\textit{a)}] We have $(\gamma^{left}_{u} = a_{uv}) > (\gamma^{right}_{u} = -a_{uv})$ and $(\gamma^{left}_{v} = a_{uv}) > (\gamma^{right}_{v} = -a_{uv})$. We see that $a_{uv}$ must be positive. Since Corollary~\ref{corollary:MVMO-Integer} is satisfied with $a_{uv}>0$, it is atomic.
    \item[\textit{b)}] We have $(\gamma^{left}_{u} = -a_{uv}) > (\gamma^{right}_{u} = a_{uv})$ and $(\gamma^{left}_{v} = -a_{uv}) > (\gamma^{right}_{v} = a_{uv})$. We see that $a_{uv}$ must be negative. Since Corollary~\ref{corollary:MVMO-Integer} is satisfied with $a_{uv}<0$, it is atomic.
    \item[\textit{c)}] We have $(\gamma^{left}_{u} = 0) > (\gamma^{right}_{u} = -a_{uv})$ and $(\gamma^{left}_{v} = 0) > (\gamma^{right}_{v} = -a_{uv})$. We see that $a_{uv}$ must be positive and this satisfies Property~\ref{prop:MVMO}. However, Corollary~\ref{corollary:MVMO-Integer} is not satisfied, which makes this edit operation decomposable.
    \item[\textit{d)}] We have $(\gamma^{left}_{u} = a_{uv}) > (\gamma^{right}_{u} = 0)$ and $(\gamma^{left}_{v} = a_{uv}) > (\gamma^{right}_{v} = 0)$. We see that $a_{uv}$ must be positive and this satisfies Property~\ref{prop:MVMO}. However, Corollary~\ref{corollary:MVMO-Integer} is not satisfied, which makes this edit operation decomposable.
    \item[\textit{e)}] We have $(\gamma^{left}_{u} = -a_{uv}) > (\gamma^{right}_{u} = 0)$ and $(\gamma^{left}_{v} = 0) > (\gamma^{right}_{v} = -a_{uv})$. We see that $a_{uv}$ must be negative and this satisfies Property~\ref{prop:MVMO}. However, Corollary~\ref{corollary:MVMO-Integer} is not satisfied, which makes this edit operation decomposable.
\end{itemize}
\end{proof}

Next, we focus on 3-edit operations. We verify some cases in which Lemma~\ref{lemma:MVMO-2Vertices} is also valid.

\begin{figure*}[!ht]
    \centering
    \includegraphics[width=1\textwidth]{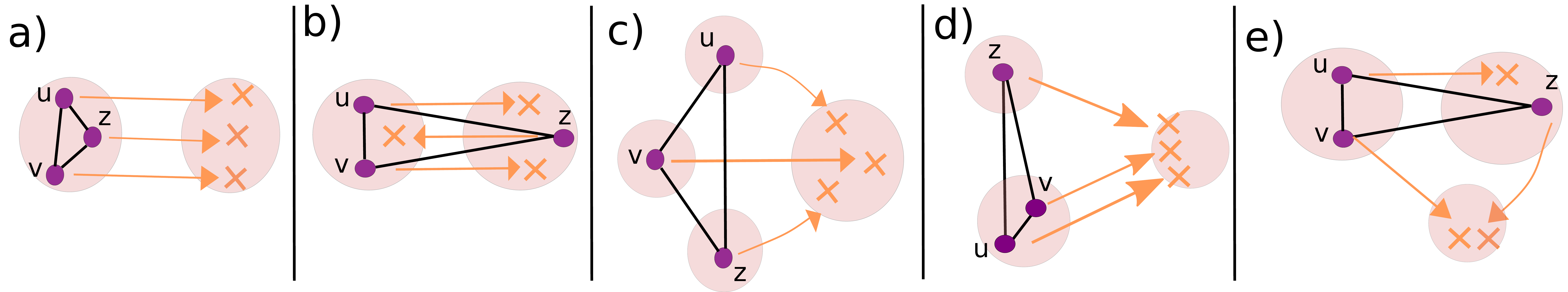}
    \caption[Fig]{Five of the possible $3$-edit operation scenarios. The full list can be found in the Appendix (Figure~\ref{fig-appendix:3edit-operations}). In this figure, all edit operations are atomic.}
    \label{fig:3edit-some-operations}
\end{figure*}

\begin{lemma}[MVMO $3$-edit]
	Consider an atomic $r$-edit operation $\pi^s \rightarrow \pi^t$ with $r=3$, where $\myOverArrow{V}{s}{t}{} = \{u, v, z\}$. Since it is an atomic edit operation, we have $(u,v), (u,z), (v,z) \in \widetilde{E}$. Lemma~\ref{lemma:MVMO-2Vertices} holds true for each pair $(u,v)$ of vertices, with two exceptions:
	\begin{enumerate}
	    \item all vertices in $\myOverArrow{V}{s}{t}{}$ are in the same source module and are moved into the same target module (Figure~\ref{fig:3edit-some-operations}a),
	    \item only two of $\myOverArrow{V}{s}{t}{}$ are in the same source module and are moved into the source module of the third moving vertex. Reciprocally, the third moving vertex moves into the source module of the others' source module (Figure~\ref{fig:3edit-some-operations}b). 
	\end{enumerate}
	\label{lemma:MVMO-3Vertices}
\end{lemma}
\begin{proof}
Without these two exceptions, $G[\myOverArrow{V}{s}{t}{}]$ forms a triangle. Nevertheless, in these two exceptions, $G[\myOverArrow{V}{s}{t}{}]$ can form a path, and, as will see next, this path ensures that moving vertices from the same source module are positively connected.
Similar to Lemma~\ref{lemma:MVMO-2Vertices}, the proof is based on all possible $17$ scenarios of three moving vertices with $(u,v), (u,z), (v,z) \in \widetilde{E}$. Some of those scenarios are shown in Figure~\ref{fig:3edit-some-operations}, and we provide the full list in the Appendix (Figure~\ref{fig-appendix:3edit-operations}). The proof is straightforward when one adapts Corollary~\ref{corollary:MVMO-Integer} to those scenarios. We verify below only those in Figure~\ref{fig:3edit-some-operations}. The full list of verifications is found in Appendix~\ref{subsecapx:MVMO-proofs-3edit}.
%
\begin{itemize}
    \item[\textit{a)}] We have $(\gamma^{left}_{u} = a_{uv} + a_{uz}) > (\gamma^{right}_{u} = -a_{uv} -a_{uz})$, $(\gamma^{left}_{v} = a_{uv} + a_{vz}) > (\gamma^{right}_{v} = -a_{uv} -a_{vz})$ and $(\gamma^{left}_{z} = a_{uz} + a_{vz}) > (\gamma^{right}_{z} = -a_{uz} -a_{vz})$. We see that $a_{uv}$, $a_{uz}$ and $a_{vz}$ cannot be negative. Since Corollary~\ref{corollary:MVMO-Integer} is satisfied with a positive path formed in $G[\myOverArrow{V}{s}{t}{}]$ for each vertex $u \in \myOverArrow{V}{s}{t}{}$, it is atomic.
    \item[\textit{b)}] We have $(\gamma^{left}_{u} = a_{uv} - a_{uz}) > (\gamma^{right}_{u} = -a_{uv} + a_{uz})$, $(\gamma^{left}_{v} = a_{uv} - a_{vz}) > (\gamma^{right}_{v} = -a_{uv} + a_{vz})$ and $(\gamma^{left}_{z} = -a_{uz} - a_{vz}) > (\gamma^{right}_{z} = a_{uz} + a_{vz})$. We see that $a_{uv}$ (resp. $a_{uz}$ and $a_{vz}$) cannot be negative (resp. positive). Since Corollary~\ref{corollary:MVMO-Integer} is satisfied with a path formed in $G[\myOverArrow{V}{s}{t}{}]$, it is atomic.
    \item[\textit{c)}] We have $(\gamma^{left}_{u} = 0) > (\gamma^{right}_{u} = -a_{uv} - a_{uz})$, $(\gamma^{left}_{v} = 0) > (\gamma^{right}_{v} = -a_{uv} - a_{vz})$ and $(\gamma^{left}_{z} = 0) > (\gamma^{right}_{z} = -a_{uz} - a_{vz})$. We see that $a_{uv}$, $a_{uz}$ and $a_{vz}$ must be positive. Since Corollary~\ref{corollary:MVMO-Integer} is satisfied with a positive triangle formed in $G[\myOverArrow{V}{s}{t}{}]$, it is atomic.
    \item[\textit{d)}] We have $(\gamma^{left}_{u} = a_{uv}) > (\gamma^{right}_{u} = -a_{uv} -a_{uz})$, $(\gamma^{left}_{v} = a_{uv}) > (\gamma^{right}_{v} = -a_{uv} -a_{vz})$ and $(\gamma^{left}_{z} = 0) > (\gamma^{right}_{z} = -a_{uz} -a_{vz})$. We see that $a_{uv}$, $a_{uz}$ and $a_{vz}$ must be positive. Since Corollary~\ref{corollary:MVMO-Integer} is satisfied with a positive triangle formed in $G[\myOverArrow{V}{s}{t}{}]$, it is atomic.
    \item[\textit{e)}] We have $(\gamma^{left}_{u} = a_{uv} - a_{uz}) > (\gamma^{right}_{u} = 0)$, $(\gamma^{left}_{v} = a_{uv}) > (\gamma^{right}_{v} = -a_{vz})$ and $(\gamma^{left}_{z} = 0) > (\gamma^{right}_{z} = a_{uz} - a_{vz})$. We see that $a_{uv}$ and $a_{vz}$ (resp. $a_{uz}$) must be positive (resp. negative). Since Corollary~\ref{corollary:MVMO-Integer} is satisfied with a triangle formed in $G[\myOverArrow{V}{s}{t}{}]$, it is atomic.
\end{itemize}
\end{proof}

Unlike Lemma~\ref{lemma:MVMO-3Vertices}, we cannot completely generalize Lemma~\ref{lemma:MVMO-2Vertices} for more than three moving vertices. Nevertheless, there are some circumstances, where Lemma~\ref{lemma:MVMO-2Vertices} is still valid for a subset of $\myOverArrow{V}{s}{t}{}$, and this is formalized in Lemma~\ref{lemma:MVMO-d-Vertices}. 


\begin{lemma}[MVMO $r$-edit]
    Consider an atomic $r$-edit operation $\pi^s \rightarrow \pi^t$ with $r\geq 4$ and a vertex $u\in\myOverArrow{V}{s}{t}{}$.
	If $2 \leq |u \cup \myOverArrow{V}{s}{t}{u}| \leq 3$, Lemma~\ref{lemma:MVMO-2Vertices} holds true for each pair $(u,v)$ with $v \in \myOverArrow{V}{s}{t}{u}$.
	\label{lemma:MVMO-d-Vertices}
\end{lemma}
\begin{proof}
The condition of $2 \leq |u \cup \myOverArrow{V}{s}{t}{u}| \leq 3$ ensures 
that subset $u \cup \myOverArrow{V}{s}{t}{u}$ represents one of the scenarios in $2$- or $3$-edit operation. 
We note that two exceptional scenarios mentioned in Lemma~\ref{lemma:MVMO-3Vertices} are not of interest here, because they necessarily satisfy the definition of atomic edit operation, hence $|u \cup \myOverArrow{V}{s}{t}{u}| > 3$. 
\end{proof}

For the sake of simplicity, in Algorithm~\ref{algo:internalPruning}, we define the function \textit{intMVMO}($G$, $\pi^s$, $\myOverArrow{V}{s}{t}{}$) which indicates whether $\pi^s \rightarrow \pi^t$ satisfies Lemmas~\ref{lemma:MVMO-2Vertices}.1 and~\ref{lemma:MVMO-3Vertices}, when the target modules of $\myOverArrow{V}{s}{t}{}$ are not known. Likewise, in Algorithm~\ref{algo:externalPruning} we define the function \textit{extMVMO}($G$, $\pi^s$, $\pi^t$, $\myOverArrow{V}{s}{t}{}$),
which indicates whether $\pi^s \rightarrow \pi^t$ satisfies Lemmas~\ref{lemma:MVMO-2Vertices}.2, \ref{lemma:MVMO-3Vertices} and~\ref{lemma:MVMO-d-Vertices}, when the target modules of $\myOverArrow{V}{s}{t}{}$ are partially or completely known 
(\ref{algoline:CoNS-externalPruning1}, \ref{algoline:CoNS-externalPruning2} and \ref{algoline:CoNS-optimalityCond2} of Algorithm~\ref{algo:CoNS}).

\section{Dataset}
\label{sec:Dataset}
This section is dedicated to the description of the dataset used in our experiments. As mentioned in the introduction, in this article we focus on \textit{unweighted} graphs as we have done in~\cite{Arinik2020b}. Application-wise, unweighted signed networks fit certain modeling situations and methodological choices. Indeed, in some works, binary values better represent the studied relations (e.g. alliance/conflict between countries in international relationships~\cite{Doreian2015}), or the authors prefer to use such values for practical reasons (e.g. limited or unreliable information~\cite{Esteban2012}).

We conduct our experiments on two datasets of random signed networks. All these data as well as the corresponding optimal solutions are publicly available online\footnote{DOI: \href{https://doi.org/10.6084/m9.figshare.15043911}{10.6084/m9.figshare.15043911}\label{foot:figshare}}.

\textbf{Dataset1:} We generate these networks through the random signed network generator proposed in our previous work~\cite{Arinik2020b}, which is publicly available online\footnote{\url{https://github.com/CompNet/SignedBenchmark}.}. For complete unweighted signed networks, this model relies on only three parameters: $n$ (number of vertices), $\ell_0$ (initial number of modules) and $q_{m}$ (proportion of misplaced edges, i.e. edges meant to be frustrated by construction). Moreover, we make the assumption that the proportion of misplaced edges is the same inside and between the modules. When it comes to incomplete unweighted signed networks, we introduce two more parameters, which are the density $d$ of the graph and the proportion $q_{neg}$ of the negative edges. The last parameter $q_{neg}$ is defined as $|E^-|/|E|$ and allows controlling the ratio of positive to negative edges. For complete unweighted signed networks with $d=1$, we generate $20$ replications for parameter values $\ell_0 = 3$, $n \in \{32, 36, 40, 45, 50\}$ and $q_m = \{0.1, 0.2, 0.3, 0.4, 0.5, 0.6\}$. In these networks, the value of $q_{neg}$ with the considered parameters is approximately $0.7$. For incomplete unweighted signed networks with $d = \{0.25, 0.50\}$, we generate $20$ replications for parameter values $\ell_0 = 3$, $n \in \{32, 36, 40\}$, $q_m = \{0.1, 0.2, 0.3, 0.4, 0.5, 0.6\}$ and $q_{neg} = \{0.3, 0.5, 0.7\}$. In total, we produce $600$ and $1{,}080$ instances for complete and incomplete networks, respectively, which makes a total of $1{,}680$ instances. We use this dataset in Section~\ref{subsec:EvaluationEnumCC}.

\textbf{Dataset2:} This dataset is different than \textit{Dataset1} in that the optimal solution for a generated network is known by construction. This allows to consider relatively large graph orders $n$ without being limited by long running time of an exact partitioning method. For given $n$, $d$ and $\ell_0$, we first create a perfectly structurally balanced (i.e., internally positive and externally negative) signed network with a built-in module structure. Clearly, the underlying module structure constitutes the optimal partition. Then, in order to take into account different positive to negative ratio values for internal and external edges we generate several signed networks by perturbing the initial signed network without affecting its underlying optimal partition, thanks to its definition of stability range~\cite{Nowozin2009}. We generate signed networks with parameter values $n \in \{30, 40, 50, 60, 70, 90\}$, $d \in \{0.25, 1.00\}$ and $\ell_0 = \{2,4,6\}$ through our random signed network generator, which is publicly available online\footnote{\label{github:SignedStabilityBenchmark}\url{https://github.com/CompNet/SignedStabilityBenchmark}.}. In total, we produce $214$ and $184$ instances for complete and incomplete networks, respectively, which makes a total of $398$ instances. In each generated network, the associated optimal solution corresponds to the planted partition defined for the corresponding network without perturbation. We use this dataset only for Section~\ref{subsec:EvaluationMVMO}.

\textbf{Dataset3:} In addition to artificial graphs, we also considered $8$ sparse real-world networks from two sources: four networks from \textit{Correlated of War (CoW)}~\cite{Pevehouse2004} and four biological networks retrieved from Samin Aref's Figshare repository\footnote{DOI: \href{https://doi.org/10.6084/m9.figshare.5700832.v5}{10.6084/m9.figshare.5700832.v5}}, see \cite{Aref2018} for more details. For each real-world signed network $G$, we apply a preprocessing step by retrieving the largest connected component defined in $G^+$, for computational purposes.

\section{Results}
\label{sec:Results}
We now assess the performance of our enumeration method $EnumCC(r_{max})$ when generating the space of optimal solutions to the CC problem. We first investigate the efficiency of the pruning conditions used in Algorithms~\ref{algo:internalPruning} and~\ref{algo:externalPruning} based on the MVMO property (Section~\ref{subsec:EvaluationMVMO}).
Then, we proceed with the evaluation of $EnumCC(r_{max})$, which includes 
all the pruning strategies used in the recurrent neighborhood search. We compare our method with OneTreeCC(), the best enumeration method available in the literature (Section~\ref{subsec:EvaluationEnumCC}).

\subsection{Evaluation of the MVMO-based Pruning Strategies}
\label{subsec:EvaluationMVMO}
The MVMO property has an important role in EnumCC, due to its ability to prune unfeasible edit operations in several parts of $CoNS(r)$. To be able to consider relatively large graph orders $n$ without being limited by the long running time of an exact partitioning method, we evaluate its performance based on \textit{Dataset2}. The complete results\footref{foot:figshare} as well as our source code\footref{github:SignedStabilityBenchmark} are publicly available online.

We apply $CoNS(r)$ \textit{with} and \textit{without} the MVMO property onto all generated networks with $r\in\{3,4\}$. The version \textit{with} MVMO property refers to the method described in Section~\ref{sec:EnumarationCCspace}, whereas the version \textit{without} it is obtained by removing this property from Algorithms~\ref{algo:internalPruning} and~\ref{algo:externalPruning}.
Due to space considerations, Figure~\ref{fig:benchmark-MVMO-4Edit} illustrates only the results related to $4$-edit distance.
The results obtained for $d = 0.25$ and $d = 1.00$ are shown in separated subfigures. In each subfigure, the $x$-axis represents the graph order $n$, and execution times (in seconds) are shown on the log-scaled $y$-axis. The solid (resp. dashed) plot lines correspond to $CoNS(r)$ with (resp. without) the MVMO property. Each plot line corresponds to a specific value of $\ell_0$ and is represented with a specific color. Each shaded region depicts a range of execution times around the plot line of the same color, based on the corresponding initial signed network and its perturbed versions.

\begin{figure}
    \centering
    \begin{subfigure}[h]{1.00\linewidth}
    \centering
    \includegraphics[width=1\textwidth]{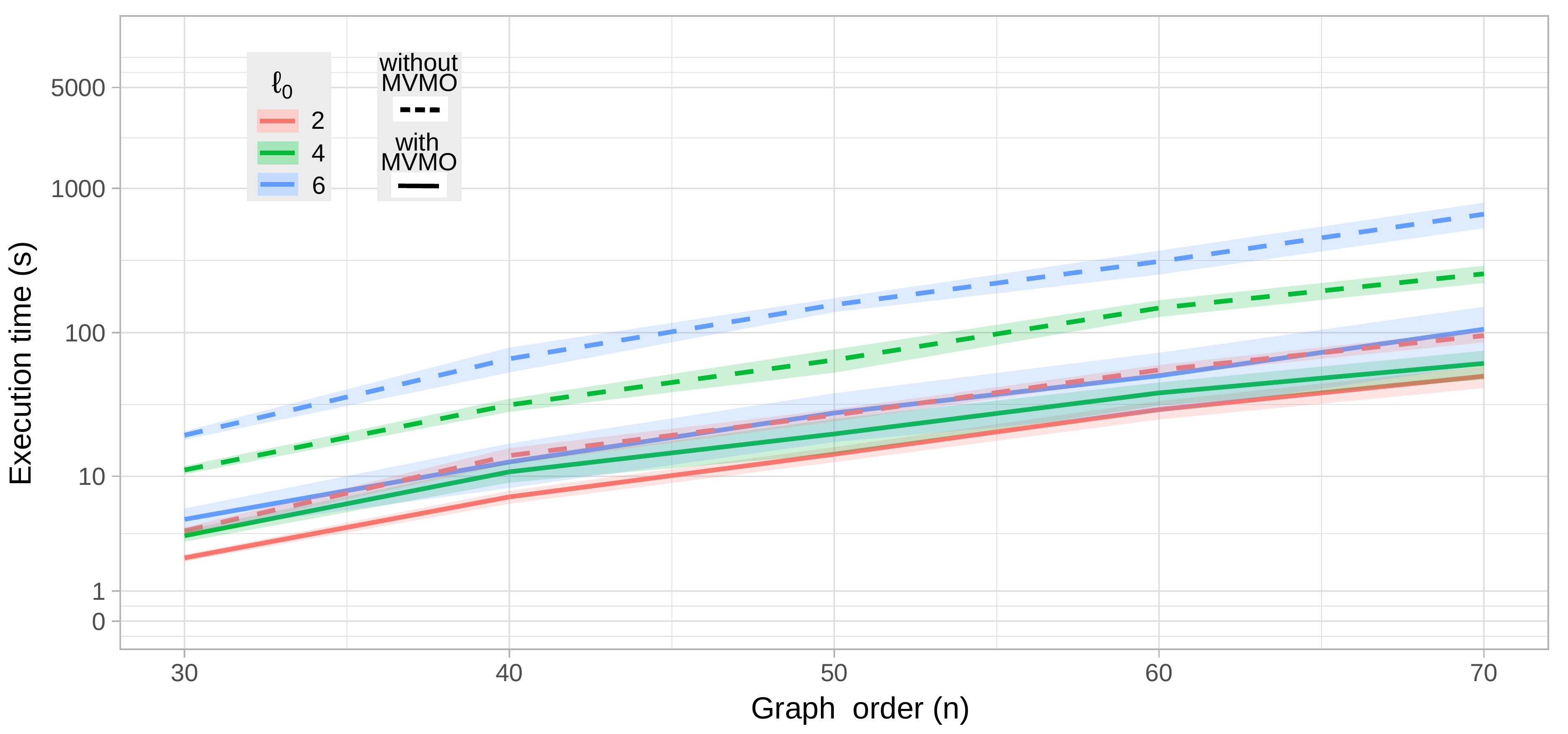}
    \caption{Instances with $d=0.25$.}
	\label{subfig:benchmark-sparse-4Edit}
    \end{subfigure}
    %
    %
    \begin{subfigure}[h]{1.00\linewidth}
    \centering
    \includegraphics[width=1\textwidth]{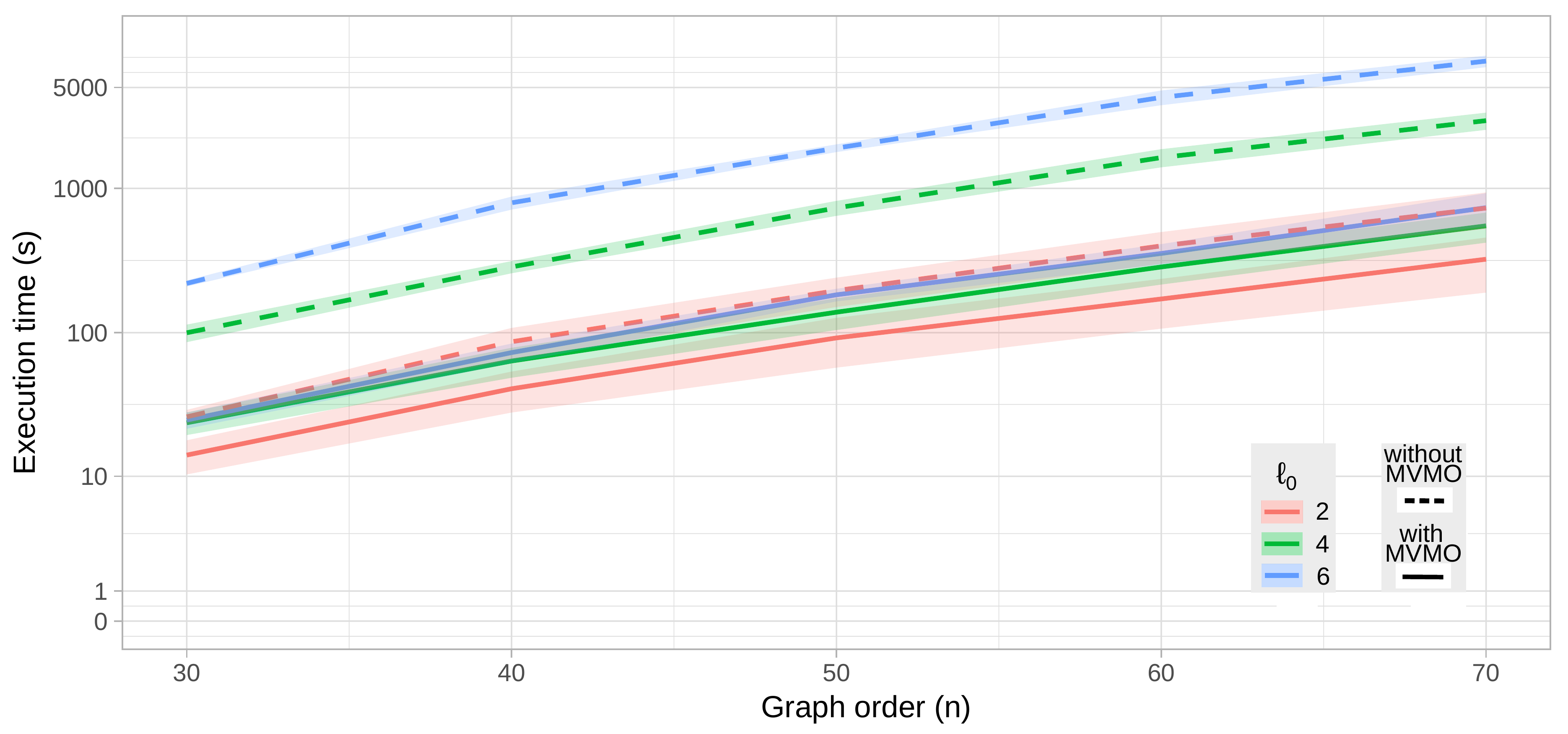}
    \caption{Instances with $d=1.00$. }
	\label{subfig:benchmark-dense-4Edit}
    \end{subfigure}
	\caption{Benchmark results for \textit{CoNS($r=4$)} with vs. without MVMO-based pruning for $d = 0.25$ (\ref{fig:benchmark-MVMO-4Edit}a) and $d = 1.00$ (\ref{fig:benchmark-MVMO-4Edit}b). The $x$-axis represents graph order $n$, and execution times (in seconds) are shown in the log-scaled $y$-axis.}
	\label{fig:benchmark-MVMO-4Edit}
\end{figure}

The first thing to notice is how graph density affects the execution times with increasing values of $n$. Indeed, the computational costs with and without the MVMO property are much larger with $d = 1.00$ than those with $d=0.25$. Second, substantially increasing $n$ affects the execution times without the MVMO property, whereas including the MVMO property allows to better handle this effect. Indeed, considering $d = 0.25$ and the $4$-edit distance, we observe that average execution times without the MVMO property are approximately $\ell_0$ times larger than with the MVMO property. Finally, including the MVMO property also allows to better handle increasing values of $\ell_0$. Indeed, when $n = 70$ and $d = 0.25$, the difference of average execution times between $\ell_0 = 2$ and $\ell_0 = 6$ without (resp. with) the MVMO property is $3.8$s (resp. $0.6$s) with the $3$-edit distance, and $595$s (resp. $63$s) with the $4$-edit distance. 

To conclude this part, the MVMO property makes a substantial improvement on $CoNS(r)$. This improvement is apparent even with small values of $n$. Note that, even a small improvement ($1s$ or $2s$) can make a clear difference in terms of execution time for a solution space with $100$ or more solutions, since $CoNS(r)$ is repeated for each solution.
Nevertheless, our results suggest that $CoNS(4)$ should not be used for computational purposes, even with the MVMO property. 
In the following, we investigate if the improvements brought by the MVMO property and other pruning strategies allow $EnumCC(r_{max})$ to compete with OneTreeCC().

\subsection{Evaluation of EnumCC}
\label{subsec:EvaluationEnumCC}
We now compare the results of $EnumCC(r_{max})$ against OneTreeCC() based on the synthetic signed networks from \textit{Dataset1} and the real-world signed networks from \textit{Dataset3}. As shown in Section~\ref{subsec:EvaluationMVMO}, since $CoNS(4)$ takes a substantial time, we rather apply EnumCC(3). We run both methods with a time limit of 12h and a limit of $50,000$ on the number of optimal solutions found. We first evaluate the results from \textit{Dataset1} for several values of density $d$ with a fixed value of $n$ (Section~\ref{subsubsec:EvaluationEnumCCbyDensity}) and several values of $n$ with a fixed value of $d$ (Section~\ref{subsubsec:EvaluationEnumCCbyGraphOrder}), then pass to the evaluation of the results from \textit{Dataset3} (Section~\ref{subsubsec:eval_realworld_instances}).

Regarding the synthetic graphs, we present a selection of the most relevant results in Figures~\ref{subfig:results-n36-sparse-neg03}--\ref{subfig:results-n50-dense}, for $d \in \{0.25, 1.00\}$. The complete results\footref{foot:figshare} as well as our source code\footnote{\url{https://github.com/CompNet/Sosocc}}\textsuperscript{,}\footnote{\url{https://github.com/CompNet/EnumCC}} are available online, though. 
We first describe these plots generically here, before interpreting them.
In these figures, there are 3 subfigures. Each subfigure corresponds to a specific value of $q_m$ (hence, a specific parameter set), and displays the difference of execution times between  EnumCC(3) and OneTreeCC() (i.e., EnumCC(3) minus OneTreeCC()), represented on the log-scaled $y$-axis of the plots. When such difference takes a negative value, this means that our proposed method runs faster than OneTreeCC(). The set of $20$ graphs generated for each parameter set are indexed and shown on the $x$-axis. Finally, to guide our discussion, we show for each graph the maximal number of solutions found by the method(s) within a time limit and the number of jumps related to EnumCC(3), i.e. $n_{jump}(EnumCC(3))$, where the latter is shown in parentheses.

\subsubsection{Evaluation of EnumCC by Density}
\label{subsubsec:EvaluationEnumCCbyDensity}
In this section, the results are for $n=36$ and $d \in \{0.25, 0.5, 1.0\}$. They are representative enough of the results obtained with other values of $n$.
We first consider the $d = 0.25$ results, shown in Figure~\ref{fig:Results-sparse-n36}. The first thing that we notice is how the performance is affected by $q_{neg}$, independently of $q_m$. Indeed, we first see for $q_{neg}=0.3$ that EnumCC(3) runs faster than OneTreeCC() in the overwhelming majority of instances. Then, for $q_{neg} = 0.5$, OneTreeCC() increases the number of instances, where it runs faster, but the dominance of EnumCC(3) is still preserved to a lesser extent. Finally, for $q_{neg}=0.7$, OneTreeCC() dominates EnumCC(3) on almost all instances. 

We observe that increasing $q_{neg}$ essentially results in the emergence of three features, which advantages OneTreeCC() over EnumCC(3). The first one is large values of $n_{jump}(EnumCC(3))$. Since a branch-and-bound tree needs to be built from scratch, each additional jump has an extra cost for EnumCC(3) in terms of execution time. We observe that the values of $n_{jump}(EnumCC(3))$ are relatively larger for mainly $q_{neg}=0.5$ and $q_m=\{0.4, 0.6\}$. This shows that increasing $q_m$ is likely to increase the dissimilarity between solutions, a fact that we have already pointed out in~\cite{Arinik2020b}, for complete signed networks. 

The second feature is a very large size of the solution space. The results show that OneTreeCC() does a much better job in this case. Indeed, when $q_{neg} = 0.7$, OneTreeCC() can solve instances associated with a very large number of solutions (e.g. $50,000$) within 5 minutes, whereas the same process often takes several hours for EnumCC(3). Nevertheless, such an extreme case happens only with some specific graph topology, as in $q_{neg} = 0.7$.
Indeed, in the random network generation, increasing $q_{neg}$ with a low graph density is more likely to produce instances having a large number of vertices with only few positive edges. Then, these vertices are often placed at the periphery of modules in the solutions, i.e. they can easily change their module from one solution to another. OneTreeCC() seems to handle well these vertices in the enumeration process, thanks to the mathematical modeling.

Finally, the third feature, complementary to the previous one, is that instances having vertices with only few positive edges are often easier to solve, so that an initial optimal solution is often found at root relaxation, before passing to branch-and-bound. This usually results in a branch-and-bound tree with fewer branches for enumerating alternative optimal solutions in OneTreeCC(). It seems that this substantially advantages OneTreeCC() over EnumCC(3), when there are many optimal solutions to enumerate.

For space considerations, we do not present the results for $d \in \{0.5, 1.0\}$. Note that for these values, we do not observe any extreme case with a very large number of solutions (as happens with $d = 0.25$ and $q_{neg} = 0.7$). Indeed, the maximum number of solutions for those instances is $2,455$. This is probably because it is not as easy to break the underlying partition structure when graph density is large. 
This fact seems to advantage EnumCC(3) over OneTreeCC(). Indeed, the dominance of EnumCC(3) persists for $q_{neg} \in \{0.3, 0.5\}$ and even better than those for $d = 0.25$ (see Table~\ref{tab:OverviewEvalEnumCC3edit}). This holds for $q_{neg} = 0.7$, too. OneTreeCC() outperforms EnumCC(3) only for $q_m = 0.1$, whereas this was true for all values of $q_m$ for $d = 0.25$. These results confirm our previous observation: OneTreeCC() outperforms EnumCC(3) on instances with specific graph topology, where low density and large proportion of negative edges produce a very large number of solutions. In the other cases, EnumCC(3) performs much better. In the following, we study whether increasing graph order $n$ affects these observations.



\begin{figure}
    \centering
    \begin{subfigure}[h]{0.85\linewidth}
    \centering
    \includegraphics[width=1\textwidth]{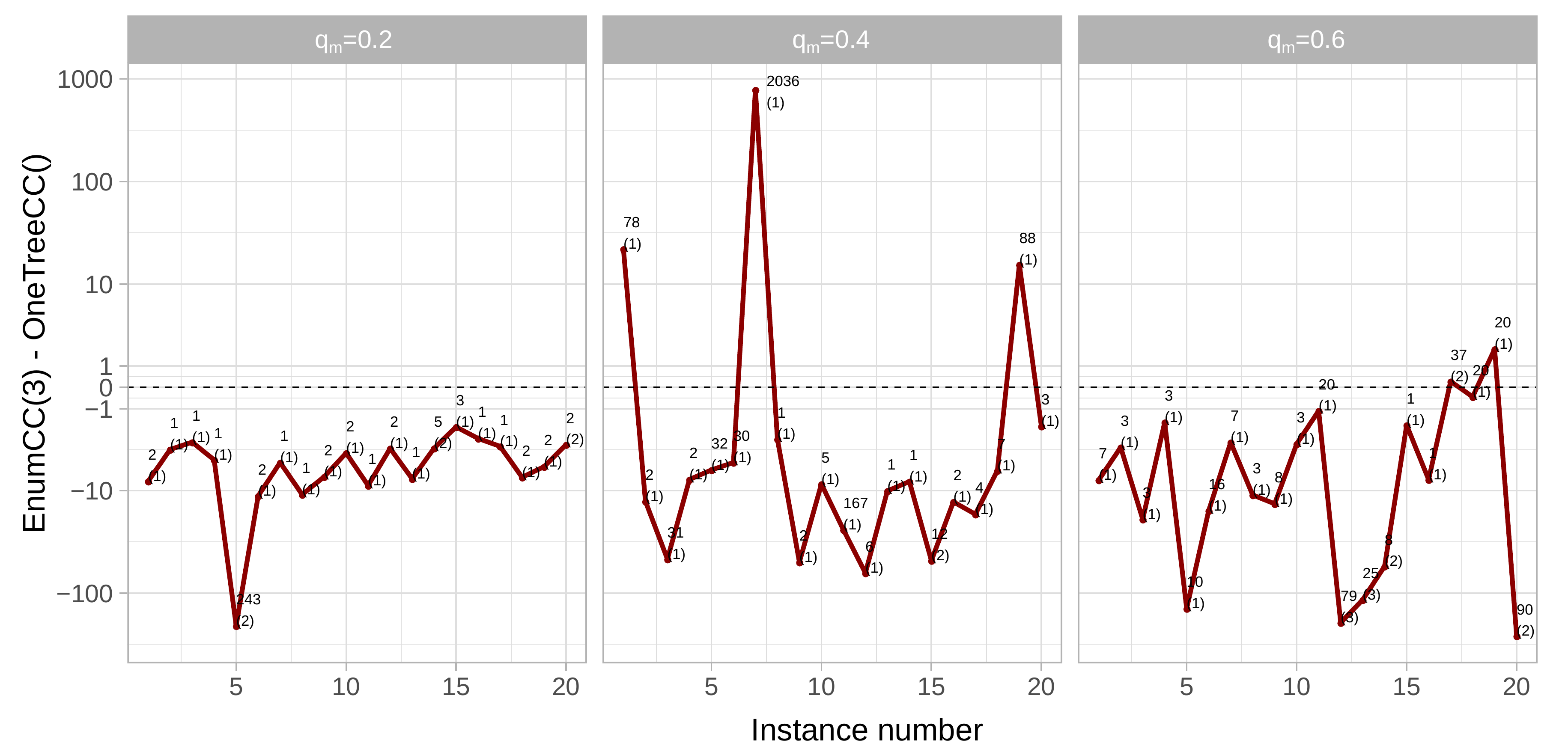}
    \caption{Instances with $d=0.25$ and $q_{neg}=0.3$.}
	\label{subfig:results-n36-sparse-neg03}
    \end{subfigure}
    %
    %
    \begin{subfigure}[h]{0.85\linewidth}
    \centering
    \includegraphics[width=1\textwidth]{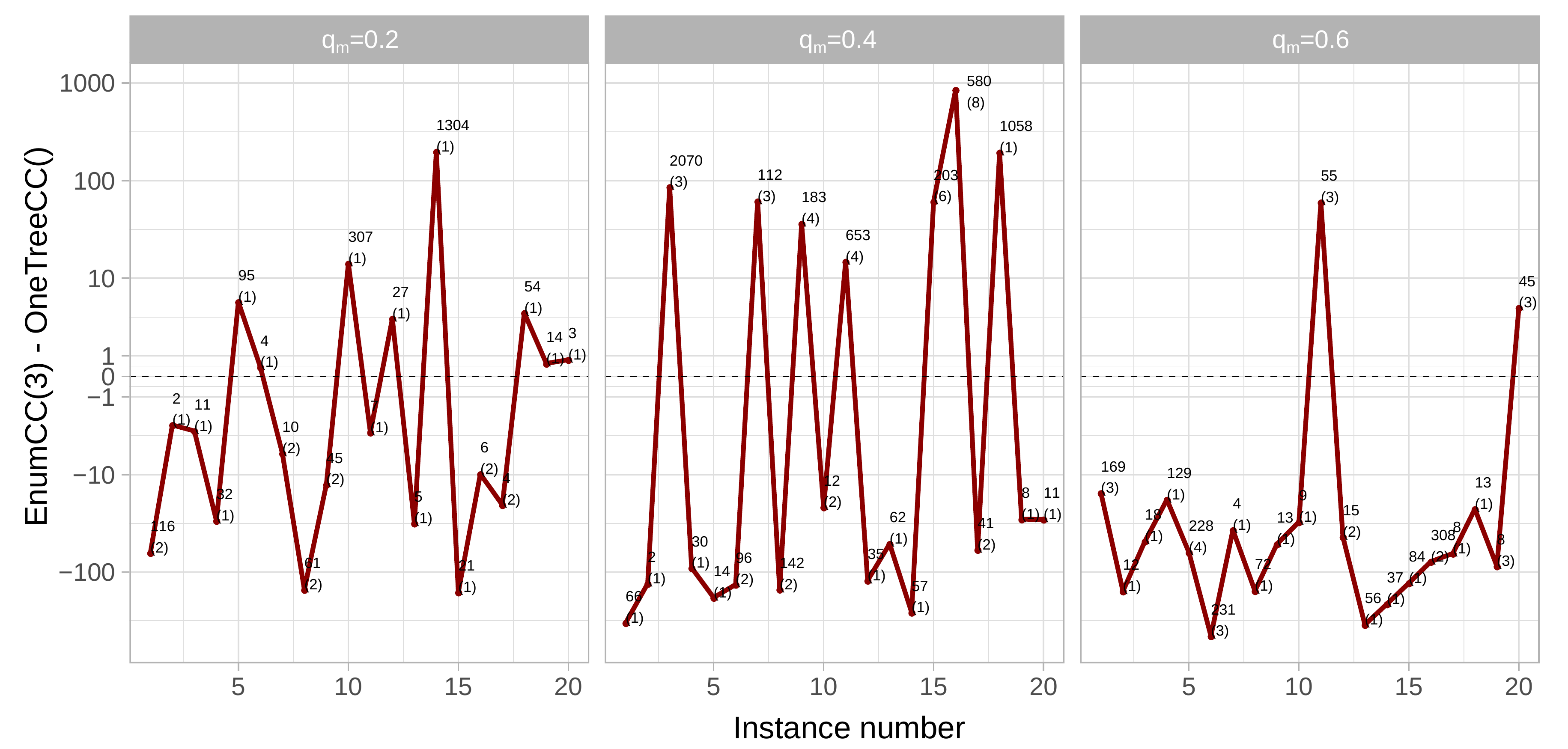}
    \caption{Instances with $d=0.25$ and $q_{neg}=0.5$.}
	\label{subfig:results-n36-sparse-neg05}
    \end{subfigure}
    %
    %
    \begin{subfigure}[h]{0.85\linewidth}
    \centering
    \includegraphics[width=1\textwidth]{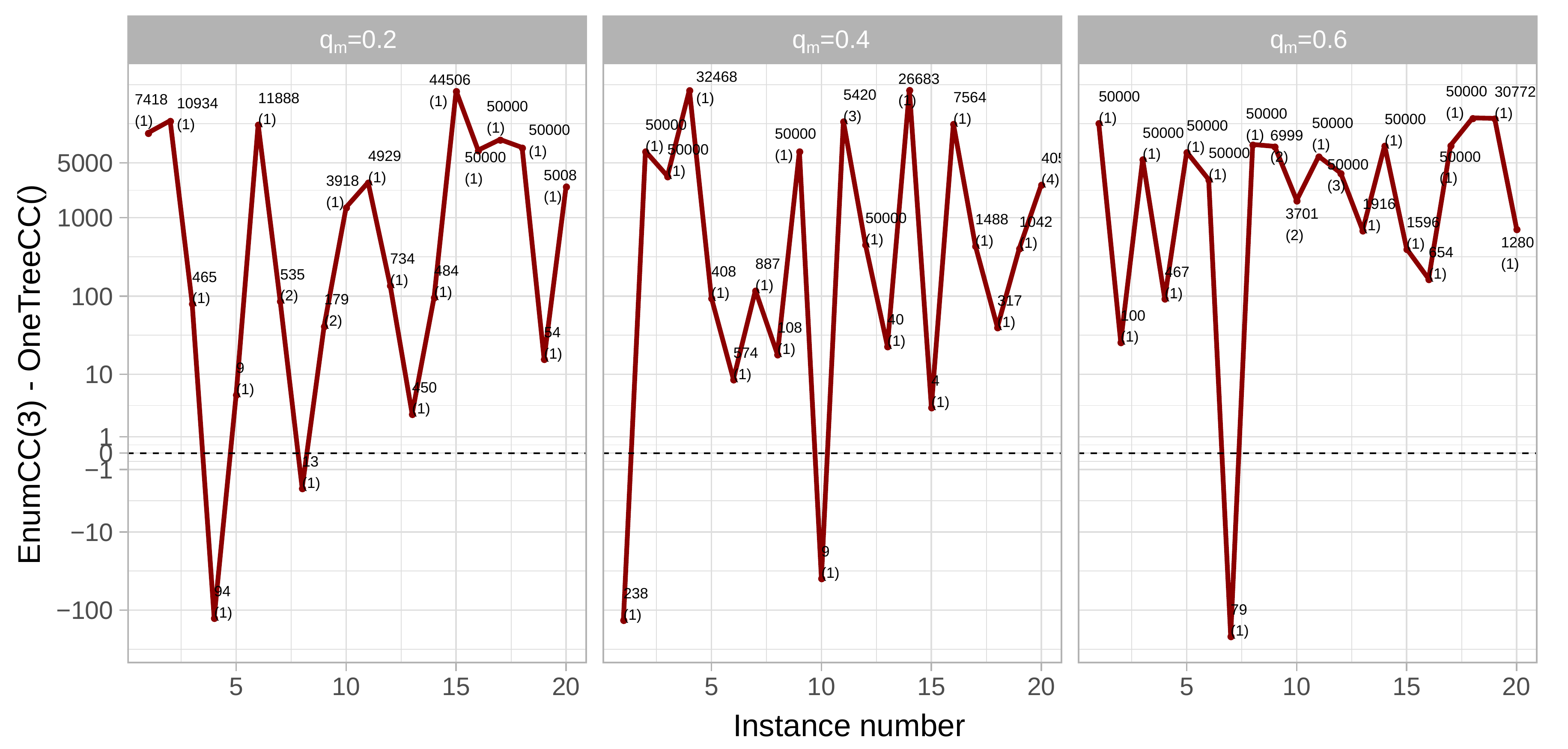}
    \caption{Instances with $d=0.25$ and $q_{neg}=0.7$.}
	\label{subfig:results-n36-sparse-neg07}
    \end{subfigure}
    \caption{Results for $n = 36$, $d = 0.25$ and $q_{neg} \in \{0.3,0.5,0.7\}$. The log-scaled $y$-axis
    indicates the difference of execution times between  EnumCC(3) and OneTreeCC() (i.e., EnumCC(3) minus OneTreeCC()), in seconds. We show for each graph the maximal number of solutions found by EnumCC(3) within a 12h time limit, as well as the corresponding number of jumps, i.e. $n_{jump}(EnumCC(3))$, between parentheses.
    }
    \label{fig:Results-sparse-n36}
\end{figure}

\subsubsection{Evaluation of EnumCC by Graph Order}
\label{subsubsec:EvaluationEnumCCbyGraphOrder}

In this section, we analyze the effect of increasing the graph order, for $n \in \{40, 45, 50\}$. To do so, we focus only on instances with $d = 1.0$ in order to reduce the total number of instances, hence the processing time of our analysis. Note that $q_{neg}$ approximately equals $0.7$ in these instances. The corresponding results are shown in Figure~\ref{fig:Results-dense-n40_50}. Overall, we see that the observations we made for $n = 36$ and $d \in \{0.5, 1.0\}$ in Section~\ref{subsubsec:EvaluationEnumCCbyDensity} are still valid when increasing $n$: EnumCC(3) performs much better.
Furthermore, unlike the results obtained with $d=0.25$, EnumCC(3) handles instances with a large value of $n_{jump}(EnumCC(3))$ much better (when we exclude some exceptional instances with more than $10$ jumps), and runs faster than OneTreeCC() in most of the instances. This point is even more valid when increasing $n$ further.

\begin{table*}[h]
	\caption{Number of optimal solutions found by the considered methods within the time limit of 12h for unsolved instances with $n=50$.
	}
    \small
	\centering
    \setlength{\tabcolsep}{0.45em} 
    \begin{tabular}{|r | r r r r r r | r r r r r r r | r r r r r r r r |}
    	\hline
    	\multirow{2}{*}{\diagbox{methods}{instance no}} &
        \multicolumn{6}{c|}{$q_m=0.4$} & 
        \multicolumn{7}{c|}{$q_m=0.5$} & 
        \multicolumn{8}{c|}{$q_m=0.6$} \\
            & 1 & 5 & 6 & 11 & 16 & 19 & 2 & 11 & 12 & 13 & 15 & 18 & 20 & 2 & 6 & 10 & 12 & 13 & 17 & 19 & 20 \\
        \hline
        \textbf{OneTreeCC()} & 4 & 4 & 1 & 1 & 13 & 2 & 7 & 21 & 3 & 1 & 10 & 2 & 2 & 3 & 1 & 3 & 4 & 5 & 3 & 1 & 3 \\
        \textbf{EnumCC(3)} & 10 & 13 & 6 & 5 & 255 & 16 & 217 & 44 & 8 & 1 & 115 & 46 & 13 & 45 & 10 & 32 & 6 & 60 & 7 & 4 & 65 \\
        \hline
	\end{tabular}
	\label{tab:n50-harder-instances}
\end{table*} 

Another interesting point is the hardness of instances, when increasing $n$. We say that an instance is hard to solve when the resolution process takes too long, which amounts to exceeding a time limit. It is very important to analyze such cases in terms of the number of optimal solutions found by both methods. Moreover, the existence of such instances also indicates the maximum graph order, for which the enumeration methods can produce full enumeration of all alternate optimal solutions. In our experiments, a total of $21$ such instances with $n = 50$ and $q_{m} > 0.3$ are not solved within the time limit of $12h$ by both methods. These instances can be identified as the ones with $y = 0$ in Figure~\ref{subfig:results-n50-dense}, and they are summarized in Table~\ref{tab:n50-harder-instances}. Each row corresponds to one of these methods, and each column represents an instance, identified by its id and its associated $q_m$ value. We see from Table~\ref{tab:n50-harder-instances} that EnumCC(3) handles these time-consuming instances better than OneTreeCC() does. Indeed, among the $21$ instances not solved by any method, EnumCC(3) finds more solutions than OneTreeCC() for $20$ instances. Moreover, EnumCC(3) could solve $7$ instances that OneTreeCC() could not, within the limit of $12h$. 


\begin{figure}
    \centering
    %
    \begin{subfigure}[h]{0.85\linewidth}
    \centering
    \includegraphics[width=1\textwidth]{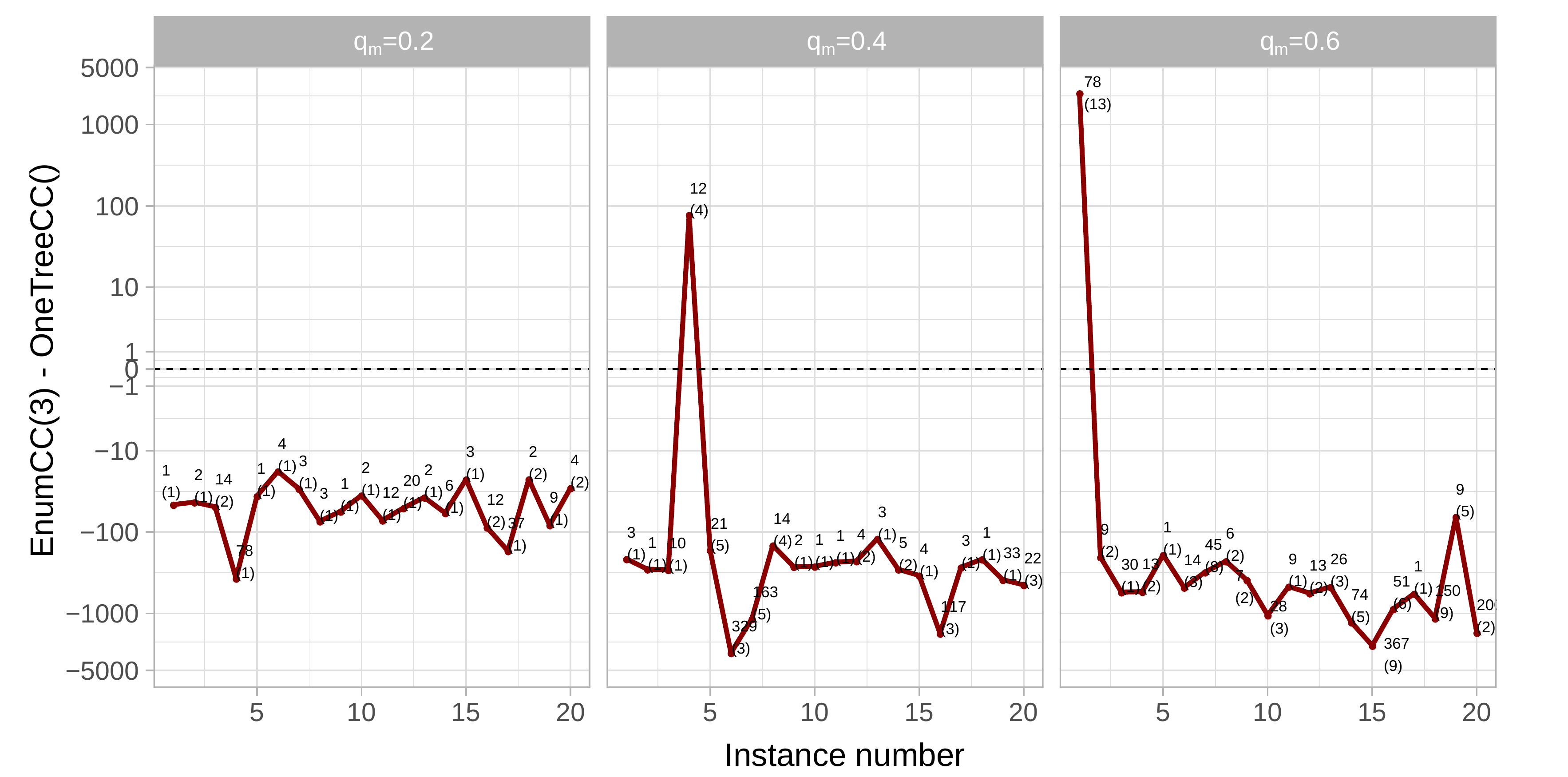}
    \caption{Instances with $n=40$ and $d=1.00$.}
	\label{subfig:results-n40-dense}
    \end{subfigure}
    %
    %
    \begin{subfigure}[h]{0.85\linewidth}
    \centering
    \includegraphics[width=1\textwidth]{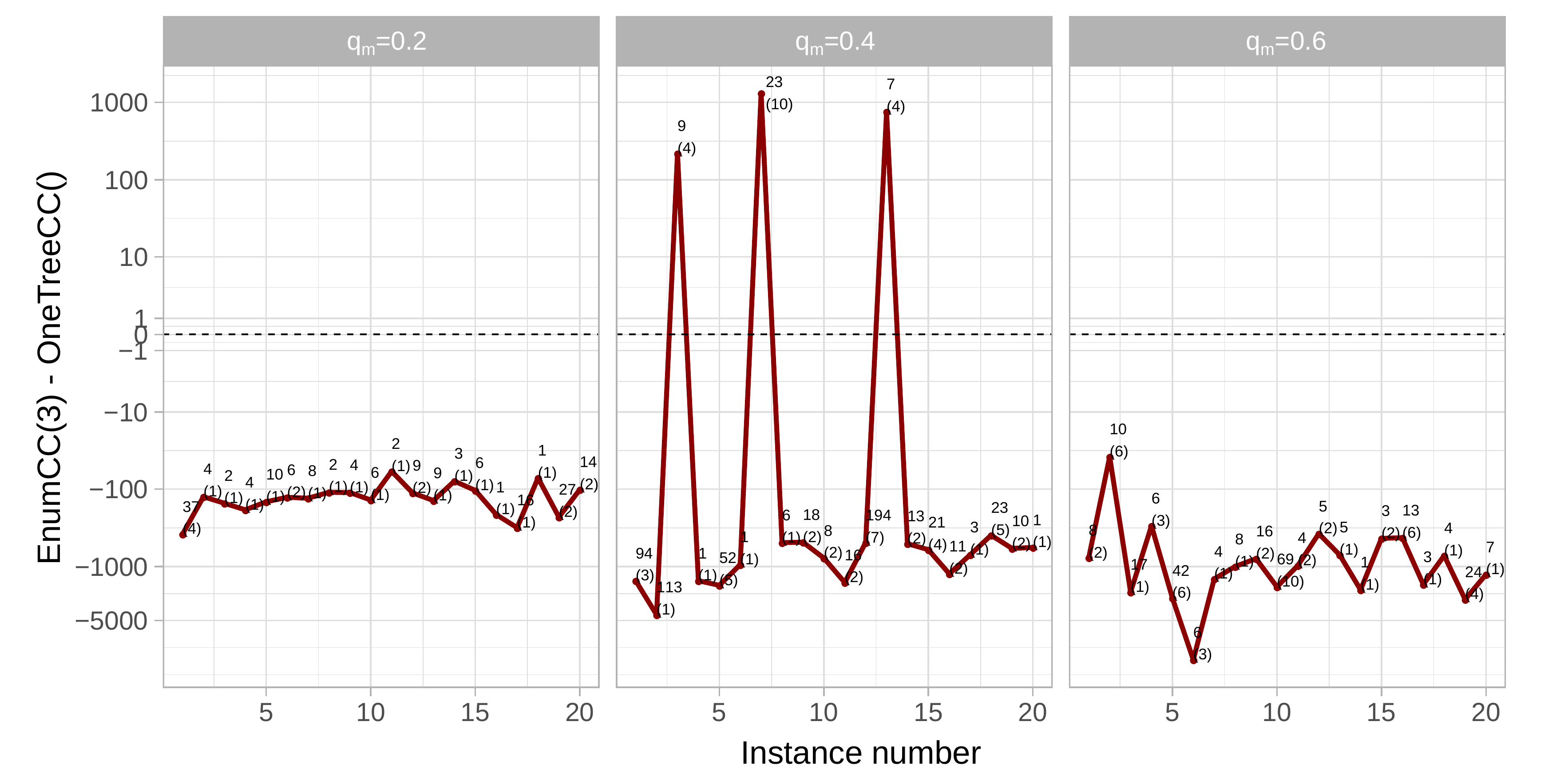}
    \caption{Instances with $n=45$ and $d=1.00$.}
	\label{subfig:results-n45-dense}
    \end{subfigure}
    %
    %
    \begin{subfigure}[h]{0.85\linewidth}
    \centering
    \includegraphics[width=1\textwidth]{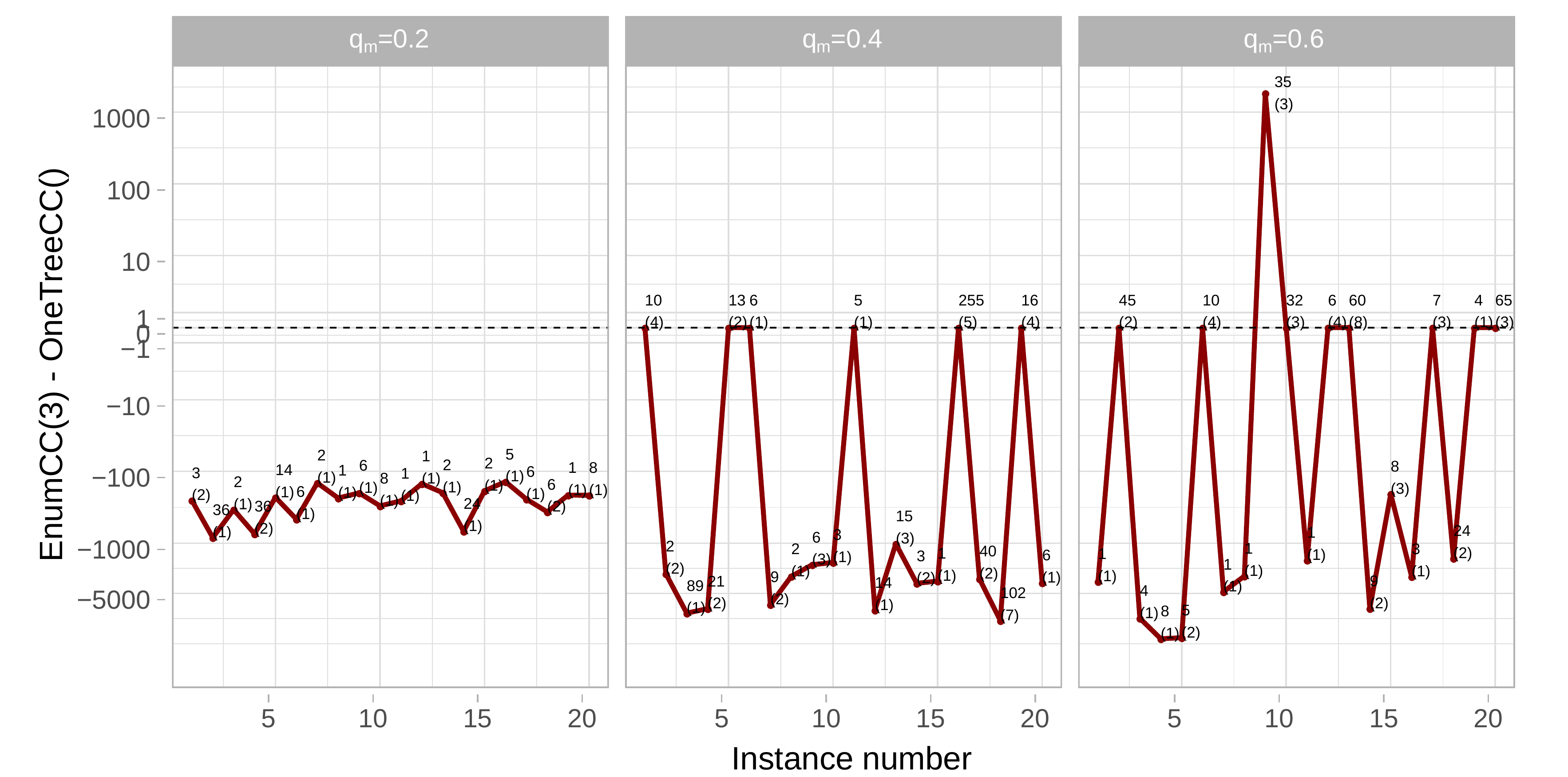}
    \caption{Instances with $n=50$ and $d=1.00$.}
	\label{subfig:results-n50-dense}
    \end{subfigure}
	\caption{Results for $n \in \{40, 45, 50\}$ and $d=1.00$. The log-scaled $y$-axis indicates the difference of execution times between  EnumCC(3) and OneTreeCC() (i.e., EnumCC(3) minus OneTreeCC()), in seconds. We show for each graph the maximal number of solutions found by EnumCC(3) within a time limit, as well as the corresponding number of jumps, i.e. $n_{jump}(EnumCC(3))$, between parentheses.
	}
	\label{fig:Results-dense-n40_50}
\end{figure}

To summarize the evaluation of EnumCC(3) with synthetic signed networks, there is not a single method which always gives the best results. On the one hand, OneTreeCC() handles very well the instances with a very large solution space. This extreme case is associated with a specific graph topology, based on low graph density and large $q_{neg}$. On the other hand, EnumCC(3) performs better in the rest of the instances, which constitutes the overwhelming majority. This point is illustrated by Table~\ref{tab:OverviewEvalEnumCC3edit}, which summarizes our results by showing the proportion of cases where EnumCC(3) runs faster.

\begin{table*}[ht]
	\caption{Summary regarding the proportion of the cases where EnumCC(3) is faster than OneTreeCC. We obtain these results by aggregating the instances by graph order $n$ (excluding the instances, when the difference of execution times between EnumCC(3) and OneTreeCC() is in the margin of 5 seconds). N/A indicates that there is no available entry due to the exclusion of the instances, when the time difference is in the margin of 5 seconds.}
    \label{tab:OverviewEvalEnumCC3edit}
	\hspace{-0.40cm}
	 \footnotesize
    \renewcommand{\arraystretch}{1.5}
    \begin{tabular}{| p{1.25cm}  p{1.25cm} || p{1.5cm} | p{1.5cm} | p{1.5cm} | p{1.5cm} | p{1.5cm} | p{1.5cm} |}
	    \hline
          &  & \textbf{$q_m=0.1$} &  \textbf{$q_m=0.2$} & \textbf{$q_m=0.3$} & \textbf{$q_m=0.4$} & \textbf{$q_m=0.5$} & \textbf{$q_m=0.6$} \\
        \hline\hline
    	\multirow{3}{*}{\parbox[c]{1.5cm}{$d=0.25$}}
    	    & \multirow{1}{*}{\parbox[c]{1.5cm}{$q_{neg}=0.3$}}
               & 1.00 & 0.92 & 1.00 & 0.87 & 0.96 & 0.92 \\
            & \multirow{1}{*}{\parbox[c]{1.5cm}{$q_{neg}=0.5$}}
               & 0.00 & 0.52 & 0.77 & 0.61 & 0.61 & 0.86\\
            & \multirow{1}{*}{\parbox[c]{1.5cm}{$q_{neg}=0.7$}}
               & 0.00 & 0.08 & 0.08 & 0.05 & 0.08 & 0.03\\
        \hline
    	\multirow{3}{*}{\parbox[c]{1.5cm}{$d=0.50$}}
    	    & \multirow{1}{*}{\parbox[c]{1.5cm}{$q_{neg}=0.3$}}
                & N/A & 1.00 & 1.00 & 1.00 & 1.00 & 1.00\\
            & \multirow{1}{*}{\parbox[c]{1.5cm}{$q_{neg}=0.5$}}
                & N/A & 1.00 & 1.00 & 0.92 & 0.84 & 0.75 \\
            & \multirow{1}{*}{\parbox[c]{1.5cm}{$q_{neg}=0.7$}}
                & 0.11 & 0.80 & 0.95 & 0.90 & 0.89 & 0.87 \\
        \hline
    	\multirow{1}{*}{\parbox[c]{1.5cm}{$d=1.00$}}
    	    & \multirow{1}{*}{\parbox[c]{1.5cm}{$q_{neg} \approx 0.7$}}
                & 1.00 & 0.98 & 0.98 & 0.94 & 0.95 & 0.98 \\
        \hline
	\end{tabular}\\
\end{table*}


\subsubsection{Evaluation Based on Real-World Instances}
\label{subsubsec:eval_realworld_instances}


Finally, we compare the results of EnumCC(3) against OneTreeCC() obtained on the real-world networks of \textit{Dataset3}. We run both methods with a time limit of 12h and a limit of $50,000$ on the number of optimal solutions found. The results are summarized in Table~\ref{tab:real-signed-net-experiments}. Each column corresponds to a real-world network, and they are sorted by increasing graph order $n$. We consider the first five networks as medium-sized and the last three as large-sized. The rows of Table~\ref{tab:real-signed-net-experiments} are organized in two parts. In the first part, we detail the characteristics of the considered networks, as well as the graph imbalance, for the sake of completeness. The second part corresponds to the evaluation results of OneTreeCC() and EnumCC(3). Therein, we indicate the execution time of the enumeration process, in seconds, and the number of optimal solutions found by these two methods within a time limit of 12h.

\begin{table}[h]
	\caption{Evaluation results for sparse real-world signed networks from Dataset3. Each column corresponds to a real-world network and they are sorted by graph order $n$. The rows are organized in two parts. The first six rows detail the characteristics of the considered networks, as well as the graph imbalance. The last four rows indicate the execution time of the enumeration process in seconds and the number of optimal solutions found by OneTreeCC() and EnumCC(3) within a time limit of 12h.}
	\renewcommand{\arraystretch}{1.2}
	\begin{tabular}{| r r | p{1cm} p{1cm} p{1cm} p{1cm} p{1cm} | p{1cm} p{1cm} p{1cm} |}
    	\hline
    	 & \multicolumn{1}{c|}{} & \multicolumn{5}{c|}{\textbf{Medium-sized networks}} & \multicolumn{3}{c|}{\textbf{Large-sized networks}}\\
         & & CoW 51-54 & CoW 54-57 & CoW 55-58 & CoW 61-64 & Senate 108\textsuperscript{th} & EGFR & Yeast & Macro-phage\\\cline{3-10}
        \multirow{6}{*}{\rotatebox[origin=c]{90}{ \parbox[c]{2.4cm}{\centering \textbf{Characteristics}}}}
            & graph order $n=|V|$ & 61 & 67 & 72 & 75 & 99 & 313 & 575 & 660\\
            & graph size $m=|E|$ & 413 & 461 & 508 & 544 & 2,413 & 755 & 910 & 1,397\\
            & density $d$ & 0.23 & 0.21 & 0.20 & 0.19 & 0.50 & 0.015 & 0.005 & 0.006\\
            & prop. of neg. edges $q_{neg}$ & 0.08 & 0.09 & 0.08 & 0.08 & 0.41 & 0.34 & 0.09 & 0.33\\
            & \# frustrated edges $I(P)$ & 15 & 27 & 29 & 34 & 1157 & 181 & 35 & 309\\
            & prop. of frustrated edges $I(P)/|E|$ & 0.04 & 0.06 & 0.06 & 0.06 & 0.48 & 0.24 & 0.04 & 0.22\\
        \hline
        \multirow{4}{*}{\rotatebox[origin=c]{90}{ \parbox[c]{1.2cm}{\centering \textbf{Results}}}}
        & \# opt. solutions for OneTreeCC() & 46 & 34 & 41 & 201 & 22 & 112 & 50,000 & 251\\
        & \# opt. solutions for EnumCC(3) & 46 & 34 & 41 & 201 & 22 & 50,000 & 50,000 & 50,000\\
        & exec. time for OneTreeCC() & 87.70s & \textbf{14.67s} & \textbf{15.62s} & \textbf{17.75s} & \textbf{858s} & 43,200s & 24,600s & 43,200s\\
        & exec. time for EnumCC(3) & \textbf{66.05s} & 67.38s & 134.87s & 755.81s & 2,424s & \textbf{3,751s} & \textbf{1,871s} & \textbf{2,280s}\\
        \hline
	\end{tabular}
	\label{tab:real-signed-net-experiments}
\end{table}

We can summarize these results in three points. First, we notice that although the medium-sized networks are larger than the synthetic graphs from \textit{Dataset2}, in practice both OneTreeCC() and EnumCC(3) can handle real-world medium-sized networks in a very reasonable time. Second, we observe for medium-sized networks that EnumCC(3) is faster than OneTreeCC() for the first one, whereas OneTreeCC() dominates EnumCC(3) for the last four ones. EnumCC(3) takes longer because, in average, it spends $86\%$ of its execution time to prove that there is no alternate optimal solution in the end. Finally, as expected, both methods could not solve all three large-sized networks within the time limit of 12h. The particularity of these networks is that they are very sparse ($d \approx 1\%$) and their respective solution spaces contain at least $50,000$ optimal solutions. For all these networks, EnumCC(3) always finds more optimal solutions than OneTreeCC() does within the time limit, and it quickly explores a set of $50,000$ optimal solutions thanks to its RNS component.
Since OneTreeCC() struggles to find more alternate solutions within the time limit in two networks (i.e. EGFR and Macrophage), RNS appears to be useful in these cases.

To conclude this part, all these results also support our conclusion from Section~\ref{subsubsec:EvaluationEnumCCbyGraphOrder}, in that there is not a single method which always gives the best results and both methods are complementary. Moreover, these results allow us to make our conclusions from Section~\ref{subsubsec:EvaluationEnumCCbyGraphOrder} more precise, in two ways. First, it is more appropriate to use EnumCC(3), rather than OneTreeCC(), to explore a large number of optimal solutions within a time limit, for large signed networks. Nevertheless, it is more suitable to favor OneTreeCC() over EnumCC(3) for sparse signed networks with $65 < n < 100$.

\section{Conclusion}
\label{sec:Conclusion}
For most clustering problems, due to their NP-hard nature, exact approaches do not scale well even when looking for a \textit{single} optimal solution. In this work, we proposed an efficient enumeration method to identify all optimal solutions of the CC problem for a given signed graph. It combines an exhaustive enumeration strategy with neighborhoods of different sizes, designed for our problem.
In our experiments based on synthetic and real-world signed networks, we first confirmed the findings of our previous work~\cite{Arinik2020b} about the existence of multiple optimal solutions for incomplete signed networks. We showed that such networks can have a very large number of optimal solutions for the CC problem, e.g., $50{,}000$ and more for graphs containing only tens of vertices. Further investigation indicates that this extreme case of a very large number of optimal solutions is associated with two specific graph topologies, corresponding to 1) low graph density and large proportion of negative edges for small- and medium-sized networks (i.e. $n \leq 100$), and 2) very low graph density for large-sized networks (i.e. $n > 100$). Otherwise, multiple solutions can still exist, mostly for the graphs with considerably more imbalance.
Finally, we also showed that our method EnumCC(3) performs better than CPLEX's OneTreeCC() in the overwhelming majority of the considered networks. Nevertheless, there is not a single method which always gives the best results. 
Indeed, OneTreeCC() handles very well sparse medium-sized signed networks in general. We conclude that it is more appropriate to use OneTreeCC() in this case, whereas EnumCC(3) is more suitable for all the remaining cases.

We believe that this work opens new directions for future research. First, the most straightforward perspective is to consider weighted signed graphs. This would require dealing with the generation of the edge weights in our random graph model, and to extend the pruning strategies proposed in this work to weighted signed graphs. Second, as we see in our experiments, the process of complete enumeration can be very costly, particularly because of the process of finding an undiscovered solution. To accelerate this process, it could be very beneficial to use a heuristic. Existing heuristics from the literature are designed to find a single high-quality solution though, so a better approach would be to redesign them by considering some tabu search-based operations allowing to escape from the discovered optimal solutions. Third, one could study how robust the solution spaces are, by slightly introducing some perturbations into the considered networks. This would also allow identifying critical vertices when the underlying space of optimal solutions is changed, akin to the concept of vitality~\cite{Koschutzki2005}. This could be done through either repeating the process from scratch for each perturbed signed graph, or based on the definition of stability range~\cite{Nowozin2009}. Fourth, exploring the space of optimal solutions for other clustering problems with unsigned networks would be another interesting research line. Indeed, the work of Good \textit{et al.}~\cite{Good2010} showed the need of identifying multiple quasi-optimal solutions for the Modularity Maximization problem, which consists in detecting a community structure in an unsigned graph. From this perspective, an efficient enumeration method similar to ours could be implemented to see whether their empirical findings apply to optimal solutions too.

\paragraph{Acknowledgments.}This research benefited from the support of Agorantic research federation (FR 3621), as well as the FMJH (Jacques Hadamard Mathematics Foundation) through PGMO (Gaspard Monge Program for Optimisation and operational research), and from the support to this program from EDF, Thales, Orange and Criteo.

\appendix

\section{2-partition and 2-chorded cycle inequalities}
\label{secapx:2Partition2ChordedCycleIneqs}
For the sake of completeness, we detail below the 2-partition and 2-chorded cycle inequalities:

\begin{itemize}
    \item Let $S, T \subseteq V$ be two nonempty disjoint subsets of $V$. Then, the $2$-partition inequality, illustrated in Figure~\ref{subfig:2partitionIneqs}, is defined as
        \begin{equation}
        \sum\limits_{u \in S} \sum\limits_{v \in T} x_{uv} - \sum_{\substack{(u,v) \in S\\u \neq v}} x_{uv} - \sum_{\substack{(u,v) \in T\\u \neq v}} x_{uv} \leq min\{|S|,|T|\}.
        \label{eq:2partitionIneqs}
        \end{equation}
    \item Let $C$ $\subseteq$ $E$ be a cycle of length at least 5, and $\overline{C} = \{v_i v_{i+2} | i=1,..,|C|-2\} \cup \{v_1 v_{|C|-1}, v_2 v_{|C|}\}$ be a $2$-chorded cycle of $C$. Then, the $2$-chorded cycle inequality, illustrated in Figure~\ref{subfig:2chordedCycleIneqs}, is defined as
        \begin{equation}
        \sum\limits_{(u,v) \in C} x_{uv} - \sum\limits_{(u,v) \in \overline{C}} x_{uv} \leq \floor{\frac{|C|}{2}}.
        \label{eq:2chordedCycleIneqs}
        \end{equation}
\end{itemize}

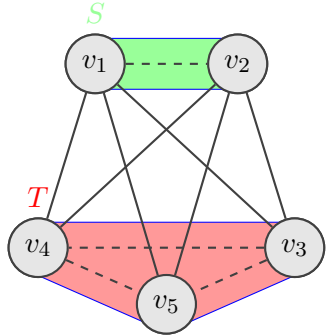
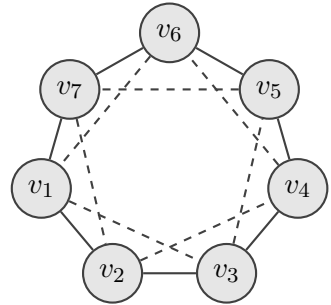
\begin{figure*}[ht!]
    \centering
	\begin{subfigure}[b]{0.45\linewidth}
    \centering
    \begin{tikzpicture}[scale=0.75]
		\tikzstyle{nn}=[circle,thick,draw=black!75,fill=black!10,minimum size=7mm]
		\tikzstyle{sl}=[thick,draw=black!75,text=black!75]
		\tikzstyle{dl}=[thick,dashed,draw=black!75,text=black!75]

		\node[nn] (v1) at (-4.5,10) {$v_1$};
		\node[nn] (v2) at (-2,10) {$v_2$};
		
		\node[nn] (v3) at (-1,6.75) {$v_3$};
		\node[nn] (v4) at (-5.5,6.75) {$v_4$};
		\node[nn] (v5) at (-3.25,5.75) {$v_5$};
		
		\draw[blue,fill=green!40] \convexpath{v1,v2}{0.45cm};
		\draw[blue,fill=red!40] \convexpath{v4,v3,v5}{0.45cm};
		
		\node[nn,label=\color{green!40}{$S$}] (v1) at (-4.5,10) {$v_1$};
		\node[nn] (v2) at (-2,10) {$v_2$};
		
		\node[nn] (v3) at (-1,6.75) {$v_3$};
		\node[nn,label=\color{red}{$T$}] (v4) at (-5.5,6.75) {$v_4$};
		\node[nn] (v5) at (-3.25,5.75) {$v_5$};

		\draw[sl] (v1)  -- (v3)  node [near end, above] {};
		\draw[sl] (v1)  -- (v4)  node [near end, above] {};
		\draw[sl] (v1)  -- (v5)  node [near end, above] {};
		\draw[sl] (v2)  -- (v3)  node [near end, above] {};
		\draw[sl] (v2)  -- (v4)  node [near end, above] {};
		\draw[sl] (v2)  -- (v5)  node [near end, above] {};

	    \draw[dl] (v1)  -- (v2)  node [near end, above] {};
	    \draw[dl] (v3)  -- (v4)  node [near end, above] {};
	    \draw[dl] (v3)  -- (v5)  node [near end, above] {};
	    \draw[dl] (v4)  -- (v5)  node [near end, above] {};
	\end{tikzpicture}
	\caption{2-partition inequality.}
	\label{subfig:2partitionIneqs}
	\end{subfigure}
	\begin{subfigure}[b]{0.45\linewidth}
	\centering
    \begin{tikzpicture}[scale=0.75]
		\tikzstyle{nn}=[circle,thick,draw=black!75,fill=black!10,minimum size=7mm]
		\tikzstyle{sl}=[thick,draw=black!75,text=black!75]
		\tikzstyle{dl}=[thick,dashed,draw=black!75,text=black!75]
		
		\node[nn] (v1) at (-2,8) {$v_1$};
		\node[nn] (v2) at (-0.75,6.5) {$v_2$};
		\node[nn] (v3) at (1.25,6.5) {$v_3$};
		\node[nn] (v4) at (2.5,8) {$v_4$};
		\node[nn] (v5) at (2,9.75) {$v_5$};
		\node[nn] (v6) at (0.25,10.75) {$v_6$};
		\node[nn] (v7) at (-1.5,9.75) {$v_7$};

		\draw[sl] (v1)  -- (v2)  node [near end, above] {};
		\draw[sl] (v2)  -- (v3)  node [near end, above] {};
		\draw[sl] (v3)  -- (v4)  node [near end, above] {};
		\draw[sl] (v4)  -- (v5)  node [near end, above] {};
		\draw[sl] (v5)  -- (v6)  node [near end, above] {};
		\draw[sl] (v6)  -- (v7)  node [near end, above] {};
		\draw[sl] (v7)  -- (v1)  node [near end, above] {};
	
	    \draw[dl] (v1)  -- (v3)  node [near end, above] {};
	    \draw[dl] (v2)  -- (v4)  node [near end, above] {};
	    \draw[dl] (v3)  -- (v5)  node [near end, above] {};
	    \draw[dl] (v4)  -- (v6)  node [near end, above] {};
	    \draw[dl] (v5)  -- (v7)  node [near end, above] {};
	    \draw[dl] (v6)  -- (v1)  node [near end, above] {};
	    \draw[dl] (v7)  -- (v2)  node [near end, above] {};
	\end{tikzpicture}
	\caption{2-chorded cycle inequality.}
	\label{subfig:2chordedCycleIneqs}
	\end{subfigure}
	\caption{Illustrations of the 2-partition (left) and 2-chorded cycle (right) inequalities.}
	\label{fig:IllustrativeExampleFacetDefiningInequalities}
\end{figure*}

\section{Edit Distance Between Two Membership Vectors}
\label{secapx:CalculatingEditDistance}

Before calculating the edit distance between two membership vectors, we need to select one of them as the \textit{reference vector}, in order to adapt the module assignments of the other vector based accordingly. Hence, the edit distance is calculated between the reference vector and this newly changed one, that we call \textit{relative vector}. 

The task of adapting the relative vector w.r.t the reference one can be expressed as an assignment problem, also known as maximum weighted bipartite matching problem, as already done in the literature~\cite{Liu2019}. 
Let $\pi^s$ and $\pi^t$ be two membership vectors of length $n$, associated with the partitions $P^s$ and $P^t$, which contain $\ell^s$ and $\ell^t$ modules, respectively. Also, since the edit distance is symmetric, without loss of generality, let $\ell^s \leq \ell^t$.
Moreover, let $CM$ be the $\ell^s\ \times\ \ell^t$ confusion matrix of $\pi^s$ and $\pi^t$. The term $CM_{ij}$, with $1\leq i \leq \ell^s$ and $1\leq j \leq \ell^t$, represents the number of vertices in the intersection of modules $M^s_{i}$ and $M^t_{j}$, i.e. $|M^s_{i} \cap M^t_{j}|$.
Then, we look for a bijection $f : \{ 1, 2, ..., \ell^s \} \rightarrow \{ 1, 2, ..., \ell^t \}$ that maximizes the number of vertices common to pairs of modules from both membership vectors, i.e.

\begin{equation}
	\max \sum_{i=1}^{\ell^s} CM_{i,f(i)}.
	\label{eq-appendix-chapter5:AssignmentProblem}
\end{equation}

Since this problem can be modelled as an assignment or a maximum weighted bipartite matching problem, it can be solved in various ways. One of them is through the well-known Hungarian algorithm, whose complexity is $O(n^3)$~\cite{Kuhn1955}. Nevertheless, the best polynomial time algorithm is currently based on the network simplex method, and it runs in $O(|V||E| + |V|^2 log(|V|))$ time using a Fibonacci heap data structure~\cite{Fredman1987}. 
One final remark is about the case of $\ell^s < \ell^t$, in which there are $|\ell^t - \ell^s|$ unassigned module labels in $\pi^t$. In this case, one can arbitrarily renumber these labels, starting from $\ell^{s}+1$. 

Finally, the edit distance between two membership vectors is calculated by simply counting the number of cases where the module labels of the vertices in the reference and relative vectors are different.

\section{Proofs}
\label{secapx:proofs}

\subsection{All Proofs Related to the MVMO Property for $3$-edit Operations on Complete Unweighted Signed Graphs}
\label{subsecapx:MVMO-proofs-3edit}
We complete the proof of Lemma~\ref{lemma:MVMO-3Vertices} by verifying below  the conditions of all atomic $3$-edit operations for unweighted complete signed networks (illustrated in Figure~\ref{fig-appendix:3edit-operations}), where $\myOverArrow{V}{s}{t}{} = \{u, v, z\}$ and $(u,v), (u,z), (v,z) \in \widetilde{E}$. Recall that $\widetilde{E} = \{(u,v) \ |\ (u,v) \in E\ \text{and}\ u,v \in \myOverArrow{V}{s}{t}{}\ \text{and}\ (\pi^s(u) = \pi^s(v) \lor \pi^t(u) = \pi^s(v) \lor \pi^s(u) = \pi^t(v) \lor \pi^t(u) = \pi^t(v))\}$.
 
\begin{figure*}
    \centering
    \includegraphics[width=1\textwidth]{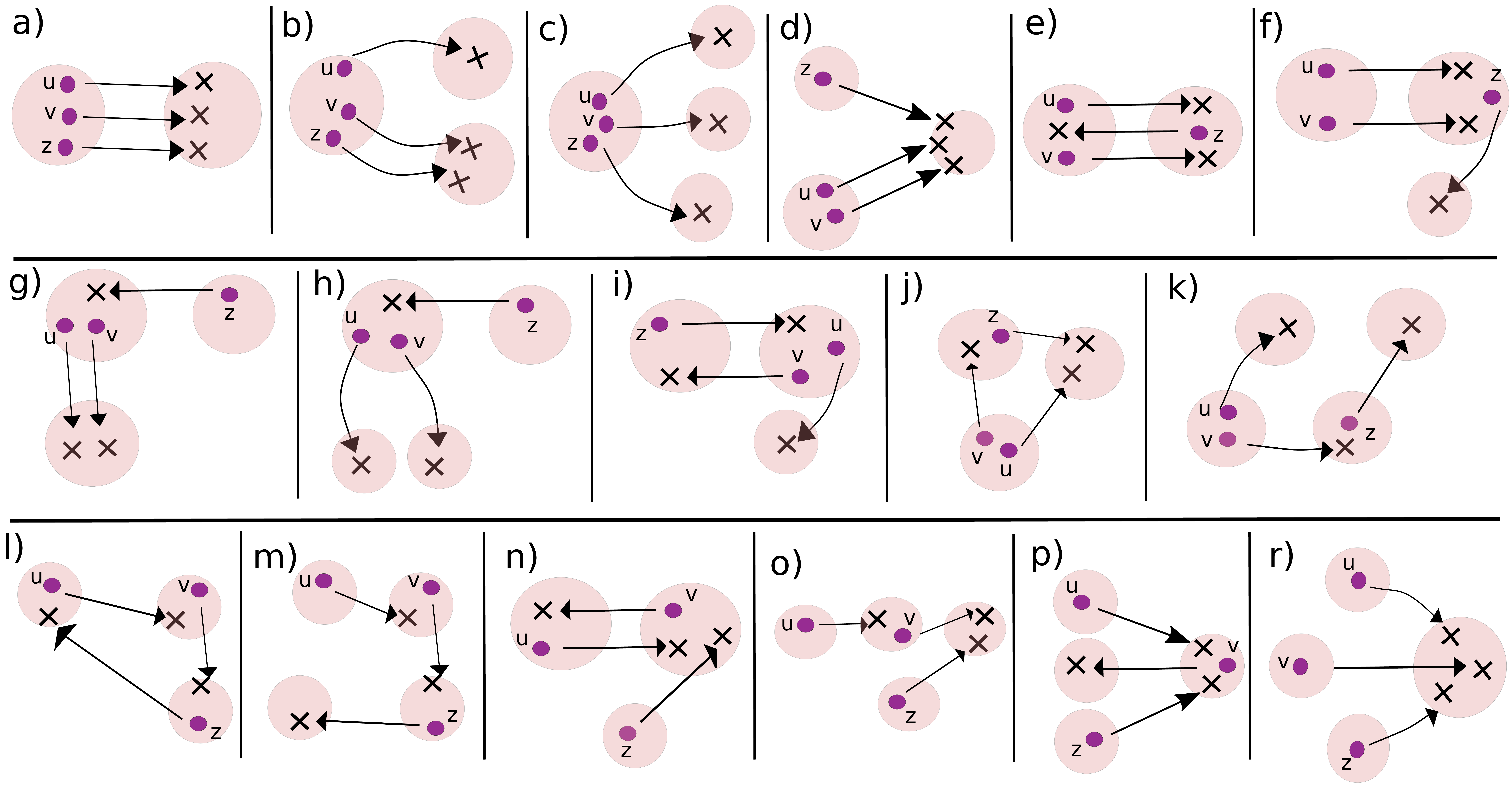}
    \caption[Fig]{All atomic 3-edit operations.}
    \label{fig-appendix:3edit-operations}
\end{figure*}

\begin{itemize}
    \item[\textit{a)}] We have $(\gamma^{left}_{u} = a_{uv} + a_{uz}) > (\gamma^{right}_{u} = -a_{uv} -a_{uz})$, $(\gamma^{left}_{v} = a_{uv} + a_{vz}) > (\gamma^{right}_{v} = -a_{uv} -a_{vz})$ and $(\gamma^{left}_{z} = a_{uz} + a_{vz}) > (\gamma^{right}_{z} = -a_{uz} -a_{vz})$. We see that $a_{uv}$, $a_{uz}$ and $a_{vz}$ cannot be negative.
    \item[\textit{b)}] We have $(\gamma^{left}_{u} = a_{uv} + a_{uz}) > (\gamma^{right}_{u} = 0)$, $(\gamma^{left}_{v} = a_{uv} + a_{vz}) > (\gamma^{right}_{v} = -a_{vz})$ and $(\gamma^{left}_{z} = a_{uz} + a_{vz}) > (\gamma^{right}_{z} = -a_{vz})$. We see that $a_{uv}$, $a_{uz}$ and $a_{vz}$ cannot be negative.
    \item[\textit{c)}] We have $(\gamma^{left}_{u} = a_{uv} + a_{uz}) > (\gamma^{right}_{u} = 0)$, $(\gamma^{left}_{v} = a_{uv} + a_{vz}) > (\gamma^{right}_{v} = 0)$ and $(\gamma^{left}_{z} = a_{uz} + a_{vz}) > (\gamma^{right}_{z} = 0)$. We see that $a_{uv}$, $a_{uz}$ and $a_{vz}$ cannot be negative.
    \item[\textit{d)}] We have $(\gamma^{left}_{u} = a_{uv}) > (\gamma^{right}_{u} = -a_{uv} -a_{uz})$, $(\gamma^{left}_{v} = a_{uv}) > (\gamma^{right}_{v} = -a_{uv} -a_{vz})$ and $(\gamma^{left}_{z} = 0) > (\gamma^{right}_{z} = -a_{uz} -a_{vz})$. We see that $a_{uv}$, $a_{uz}$ and $a_{vz}$ must be positive.
    \item[\textit{e)}] We have $(\gamma^{left}_{u} = a_{uv} - a_{uz}) > (\gamma^{right}_{u} = -a_{uv} + a_{uz})$, $(\gamma^{left}_{v} = a_{uv} - a_{vz}) > (\gamma^{right}_{v} = -a_{uv} + a_{vz})$ and $(\gamma^{left}_{z} = -a_{uz} - a_{vz}) > (\gamma^{right}_{z} = a_{uz} + a_{vz})$. We see that $a_{uv}$ (resp. $a_{uz}$ and $a_{vz}$) cannot be negative (resp. positive).
    \item[\textit{f)}] We have $(\gamma^{left}_{u} = a_{uv} - a_{uz}) > (\gamma^{right}_{u} = -a_{uv})$, $(\gamma^{left}_{v} = a_{uv} - a_{vz}) > (\gamma^{right}_{v} = -a_{uv})$ and $(\gamma^{left}_{z} = 0) > (\gamma^{right}_{z} = a_{uz} + a_{vz})$. We see that $a_{uv}$, $a_{uz}$ and $a_{vz}$ cannot be negative.
    \item[\textit{g)}] We have $(\gamma^{left}_{u} = a_{uv}) > (\gamma^{right}_{u} = a_{uz} -a_{uv})$, $(\gamma^{left}_{v} = a_{uv}) > (\gamma^{right}_{v} = a_{vz} - a_{uv})$ and $(\gamma^{left}_{z} = -a_{uz} - a_{vz}) > (\gamma^{right}_{z} = 0)$. We see that $a_{uv}$, $a_{uz}$ and $a_{vz}$ cannot be negative.
    \item[\textit{h)}] We have $(\gamma^{left}_{u} = a_{uv}) > (\gamma^{right}_{u} = -a_{uv})$, $(\gamma^{left}_{v} = a_{uv}) > (\gamma^{right}_{v} = - a_{uv})$ and $(\gamma^{left}_{z} = -a_{uz} - a_{vz}) > (\gamma^{right}_{z} = 0)$. We see that $a_{uv}$, $a_{uz}$ and $a_{vz}$ cannot be negative.
    \item[\textit{i)}] We have $(\gamma^{left}_{u} = a_{uv}) > (\gamma^{right}_{u} = a_{uz})$, $(\gamma^{left}_{v} = a_{uv} - a_{vz}) > (\gamma^{right}_{v} = a_{vz})$ and $(\gamma^{left}_{z} = -a_{uz} - a_{vz}) > (\gamma^{right}_{z} = + a_{vz})$. We see that $a_{uv}$ (resp. $a_{uz}$ and $a_{vz}$) cannot be negative (resp. positive).
    \item[\textit{j)}] We have $(\gamma^{left}_{u} = a_{uv}) > (\gamma^{right}_{u} =  - a_{uz})$, $(\gamma^{left}_{v} = a_{uv}) > (\gamma^{right}_{v} = -a_{vz})$ and $(\gamma^{left}_{z} = 0) > (\gamma^{right}_{z} = -a_{uz} + a_{vz})$. We see that $a_{uv}$ and $a_{vz}$ (resp. $a_{uz}$) must be positive (resp. negative).
    \item[\textit{k)}] We have $(\gamma^{left}_{u} = a_{uv}) > (\gamma^{right}_{u} = 0)$, $(\gamma^{left}_{v} = a_{uv} - a_{vz}) > (\gamma^{right}_{v} = 0)$ and $(\gamma^{left}_{z} = 0) > (\gamma^{right}_{z} = a_{vz})$. We see that $a_{uv}$ and $a_{vz}$ (resp. $a_{uz}$) must be positive (resp. negative).
    \item[\textit{l)}] We have $(\gamma^{left}_{u} = -a_{uv}) > (\gamma^{right}_{u} = a_{uz})$, $(\gamma^{left}_{v} = -a_{vz}) > (\gamma^{right}_{v} = a_{uv})$ and $(\gamma^{left}_{z} = -a_{uz}) > (\gamma^{right}_{z} = a_{vz})$. We see that $a_{uv}$ and $a_{vz}$ (resp. $a_{uz}$) must be positive (resp. negative).
    \item[\textit{m)}] We have $(\gamma^{left}_{u} = -a_{uv}) > (\gamma^{right}_{u} = 0)$, $(\gamma^{left}_{v} = -a_{vz}) > (\gamma^{right}_{v} = a_{uv})$ and $(\gamma^{left}_{z} = 0) > (\gamma^{right}_{z} = a_{vz})$. We see that $a_{uv}$ and $a_{vz}$ (resp. $a_{uz}$) must be positive (resp. negative).
    \item[\textit{n)}] We have $(\gamma^{left}_{u} = -a_{uv}) > (\gamma^{right}_{u} = a_{uv} - a_{uz})$, $(\gamma^{left}_{v} = -a_{uv}) > (\gamma^{right}_{v} = a_{uv} + a_{vz})$ and $(\gamma^{left}_{z} = -a_{vz}) > (\gamma^{right}_{z} = -a_{uz})$. We see that $a_{uv}$ (resp. $a_{uz}$ and $a_{vz}$) cannot be negative (resp. positive).
    \item[\textit{o)}] We have $(\gamma^{left}_{u} = -a_{uv}) > (\gamma^{right}_{u} = 0)$, $(\gamma^{left}_{v} = 0) > (\gamma^{right}_{v} = a_{uv} - a_{vz})$ and $(\gamma^{left}_{z} = 0) > (\gamma^{right}_{z} = -a_{vz})$. We see that $a_{uv}$ (resp. $a_{uz}$ and $a_{vz}$) cannot be negative (resp. positive).
    \item[\textit{p)}] We have $(\gamma^{left}_{u} = -a_{uv}) >  (\gamma^{right}_{u} = -a_{uz})$, $(\gamma^{left}_{v} = 0) > (\gamma^{right}_{v} = a_{uv} + a_{vz})$ and $(\gamma^{left}_{z} = -a_{vz}) > (\gamma^{right}_{z} = -a_{uz})$. We see that $a_{uv}$ (resp. $a_{uz}$ and $a_{vz}$) cannot be negative (resp. positive).
    \item[\textit{r)}] We have $(\gamma^{left}_{u} = 0) > (\gamma^{right}_{u} = -a_{uv} - a_{uz})$, $(\gamma^{left}_{v} = 0) > (\gamma^{right}_{v} = -a_{uv} - a_{vz})$ and $(\gamma^{left}_{z} = 0) > (\gamma^{right}_{z} = -a_{uz} - a_{vz})$. We see that $a_{uv}$, $a_{uz}$ and $a_{vz}$ must be positive.
\end{itemize}

\phantomsection\addcontentsline{toc}{section}{References}
\printbibliography

\end{document}